\documentclass[10pt]{amsart}
\usepackage{amssymb}
\usepackage{amscd}
\usepackage{amsmath}
\usepackage{amsmath}
\usepackage[all]{xy}
\usepackage{mathrsfs}
\usepackage{graphicx}
\usepackage{color} 

\newcommand{\Der}{\catname{D}}
\newcommand{\DerC}{\catname{C}}
\theoremstyle{plain}
\newtheorem{theorem}{Theorem}[section]
\newtheorem{lemma}[theorem]{Lemma}

\newtheorem{proposition}[theorem]{Proposition}
\newtheorem{definition}[theorem]{Definition}

\newtheorem{example}[theorem]{Example}

\newtheorem{remark}[theorem]{Remark}

\newtheorem*{acknowledgments}{Acknowledgements}
  1

\newcommand{\vir}[1]{{{\rm d}(} #1{)}}
\newcommand{\eins}{\boldsymbol{1}}
\DeclareMathOperator{\Aut}{Aut}

\DeclareMathOperator{\ce}{\mathcal{E}}

\DeclareMathOperator{\Hom}{Hom}
\DeclareMathOperator{\End}{End}

\DeclareMathOperator{\im}{im}

\newcommand{\CC}{\mathbb{C}}

\newcommand{\GG}{\mathbb{G}}

\newcommand{\QQ}{\mathbb{Q}}

\newcommand{\ZZ}{\mathbb{Z}}

\newcommand{\cA}{\mathcal{A}}

\newcommand{\frp}{\mathfrak{p}}

\newcommand{\id}{\mathrm{id}}

\newcommand{\nr}{\mathrm{Nrd}}

\newcommand{\bz}{\mathbb{Z}}

\def\bigcapp{\raise1ex\hbox{\rotatebox{180}{$\biguplus$}}}
 \def\bigcappd{\raise1ex\hbox{\rotatebox{180}{$\displaystyle\biguplus$}}}

\newcommand{\catname}[1]{\textnormal{{\textsf{#1}}}}

\makeatletter
 
  \@addtoreset{equation}{subsection}
\makeatother

\begin{document}

%%%%%%%%%%%%%%%%%%%%%%%%%%%%%%%
%%% TITLE, AUTHOR, ABSTRACT %%%
%%%%%%%%%%%%%%%%%%%%%%%%%%%%%%%
\title[]{On non-commutative Euler systems, I: \\
preliminaries on `det' and `Fit'}
%\\ for $p$-adic representations}

%\title[On derivatives of $p$-adic $L$-series]{On Mordell-Weil, Tate-Shafarevich\\and Bloch-Kato}\footnote{Preliminary version of September 2011}

\author{David Burns and Takamichi Sano}
%
%\centerline{\em Dedicated to the memory of Jan Nekov\'a\v r}

\begin{abstract} We extend some classical constructions in commutative algebra to the setting of modules over orders in (non-commutative) semisimple algebras. Our theory incorporates, inter alia, `reduced' versions of the notions of higher Fitting invariants and higher exterior powers and of the Grothendieck-Knudsen-Mumford determinant functor on perfect complexes. In a companion article, these results are used to develop a theory of non-commutative Euler systems for $p$-adic representations.  
\end{abstract}

\address{King's College London,
Department of Mathematics,
London WC2R 2LS,
U.K.}
\email{david.burns@kcl.ac.uk}

\address{Osaka Metropolitan University, 
Department of Mathematics, 
3-3-138 Sugimoto\\Sumiyoshi-ku\\Osaka\\558-8585, 
Japan}
\email{tsano@omu.ac.jp}

%\classification{}

\maketitle

%\tableofcontents

%\section{Introduction}\label{Intro}%This is an updated version of the submitted file dburns.tex
%\tableofcontents
%
\section{Introduction}

Let $R$ be a Dedekind domain whose field of fractions is a number field $F$, and $\mathcal{A}$ an $R$-order that spans a finite dimensional semisimple $F$-algebra. Then the main aim of this article is to extend the determinant functor constructed for commutative rings by Knudsen and Mumford in \cite{knumum} (and \cite{knudsen}), following initial ideas of Grothedieck, to the derived category of finitely generated locally-free $\mathcal{A}$-modules. 

We recall that more general versions of the Grothendieck-Knudsen-Mumford determinant functor have been constructed previously, both by Deligne \cite{delignedet} in terms of the Picard category of `virtual objects'  and by Fukaya and Kato \cite{fukaya-kato} via a theory of `localized $K_1$-groups'. Our approach is, however, different from these important earlier theories (see, for example, Remark \ref{comparison of categories}) and, being somewhat more concrete, seems particularly well-suited to explicit arithmetic applications. For example, in the companion article \cite{sbes1-2}, the results proved here play a key role in the development of a theory of non-commutative Euler systems for $p$-adic representations over number fields. 

To give a few more details we shall, for simplicity, assume $\mathcal{A}$ is the group ring $R[G]$ for a finite group $G$. Then, as a first step, in \S\ref{app higher fits} we use reduced norms of matrices with coefficients in $R[G]$ to define a canonical $R$-order $\xi(R[G])$ in the centre $\zeta(F[G])$ of $F[G]$. We call this order the `Whitehead order' of $R[G]$ and it plays an essential role in the specification of integral structures in our theory. For context, we note that $\xi(R[G])$ is, in general, neither contained in, nor contains,  $\zeta(R[G])$. 

As a first application, we then use Whitehead orders to develop an analogue for finitely generated $R[G]$-modules of the theory of higher Fitting ideals over commutative rings, as discussed by Northcott \cite{North}. This theory of `non-commutative Fitting invariants' occurs naturally in arithmetic applications (see \cite{sbes1-2}) and its main aspects are established in Theorem \ref{main fit result}.  

%We next use $\xi(R[G])$ in \S\ref{red ext powers sec} to construct a well-behaved theory of `(reduced) exterior powers' on the category of $R[G]$-modules. 
We next fix an algebraic closure $F^c$ of $F$ and, for each irreducible $F^c$-valued character $\chi$ of $G$, a corresponding representation $G\to {\rm GL}_{\chi(1)}(F^c)$. 
%of character $\chi$, where $E$ is a finite extension of $\QQ_p$. (If $e$ denotes the exponent of $G$, then one can take $E$ to be the filed generated by a primitive $e$-th root of unity.)
%
%where ${\rm Ir}_{F^c}(G)$ is the set of irreducible $F^c$-valued characters of $G$ and $\rho_\chi$ has character $\chi$. %en fix an associated set of representationseach $\rho_\chi$ has character $\chi$.
%
We use this data to define, for each finitely generated $F[G]$-module $W$ and each non-negative integer $r$, a canonical `$r$-th reduced exterior power' ${\bigwedge}^r_{F[G]}W$  that is a $\zeta(F[G])$-module, %for each $s$ with $0\le s\le r$ a canonical pairing of $\zeta(F[G])$-modules
%
%\[ {\bigwedge}^r_{F[G]}W \times {\bigwedge}^s_{F[G]^{\rm op}}\Hom_{F[G]}(W,F[G]) \to {\bigwedge}^{r-s}_{F[G]}W,\]
%
and for each ordered subset $\{w_i\}_{1\le i\le r}$ of $W$ a `reduced exterior product' element $\wedge_{i = 1}^{i=r}w_i$ in ${\bigwedge}^r_{F[G]}W$. The main properties of these constructions are proved in \S\ref{sec basic}. 

An essential difficulty is then to show that a full $R[G]$-lattice $M$ in $W$ gives rise, in a functorial manner, to a corresponding integral structure on ${\bigwedge}^r_{F[G]}W$. % only depends on (\ref{reps}).
%Each exterior powerThe next result concerns an exterior power `over $\cA$',
  We resolve this problem by using reduced exterior products to define, in terms of the order $\xi(R[G])$, a generalization of the  notion of `Rubin lattice' introduced (for commutative orders) in \cite{R} that has played a key role in the theory of higher rank Euler and Kolyvagin systems via an associated concept of `exterior power bidual'. These `reduced Rubin lattices' are finitely generated $\xi(R[G])$-modules and their main properties are established in  Theorem \ref{exp prop}. 

Finally, in \S\ref{non-comm dets}, we use the theory of reduced Rubin lattices to construct a canonical extended determinant functor from the derived category of bounded complexes of finitely generated locally-free $R[G]$-modules to the 
Picard category of graded invertible $\xi(R[G])$-modules. The main properties of this functor are described in 
Theorem \ref{ext det fun thm} and are proved by combining the properties of reduced Rubin lattices established in Theorem \ref{exp prop} 
with many of the original arguments used by Knudsen and Mumford in \cite{knumum}. 

%Finally, we combine the above results to obtain a following result, which has an important arithmetic application.
%
%\begin{theorem}\label{th complex}
%Let $C$ be a complex of $\cA$-modules of the form
%$$C=[\cA^d \xrightarrow{M} \cA^d],$$
%where the first term is placed in degree zero and $M \in {\rm M}_d(\cA)$. Assume there is a decomposition $H^1(C)=X\oplus \cA^r$ for some $\cA$-module $X$ and integer $r$. Then ${\rm d}_\cA(C)$ is a free $\xi(\cA)$-module of rank one and there is a canonical homomorphism
%$$\Theta_C: {\rm d}_\cA(C) \to {\bigcap}_\cA^r H^0(C)$$
%such that any generator $\varepsilon$ of  $ \im(\Theta_C)$ satisfies the following properties.
%\begin{itemize}
%\item[(i)] The annihilator of $\varepsilon$ in $\xi(\cA)$ is equal to
%$$\xi(\cA)\cap \zeta(A)(1-e),$$
%where $e \in A$ is the sum of primitive central idempotents that annihilate $X \otimes_{\ZZ_p}\QQ_p$.
%\item[(ii)] One has
%$$\xi(\cA)\cdot \{(\varphi_1\wedge\cdots\wedge \varphi_r)(\varepsilon) : \varphi_i \in \Hom_\cA(H^0(C),\cA)\}={\rm Fit}_\cA^r(M) \subseteq {\rm Fit}_\cA^0(X).$$
%In particular, one has
%$$\mathfrak{A}(\cA)\cdot \{(\varphi_1\wedge\cdots\wedge \varphi_r)(\varepsilon) : \varphi_i \in \Hom_\cA(H^0(C),\cA)\} \subseteq {\rm Ann}_\cA(X).$$
%\end{itemize}
%\end{theorem}
%
%We remark that the theory above can be extended directly to a semisimple ring of the form $\QQ[G]$ and its order $\ZZ[G]$, by localizing at $p$ for every prime number $p$. (In this case, we should consider `locally free modules' rather than free modules.)

Whilst our initial motivations for these constructions related to Euler systems (as discussed in \cite{sbes1-2}), the underlying ideas do seem of independent interest and indeed already have other arithmetic applications. For instance, the articles \cite{ffmcmm} of de Frutos-Fern\'andez, Macias Castillo and Martinez Marqu\'es and \cite{mct} of Macias Castillo and Tsoi use aspects of our approach to respectively study class number formulas for Drinfeld modules and Hasse-Weil-Artin $L$-series of elliptic curves over number fields. In another direction, reduced determinant functors lead naturally to a more concrete formulation of the central conjectures formulated in \cite{bf} and \cite{fukaya-kato} (thereby avoiding relative $K$-theory and the sophisticated theories of `virtual  objects' and `localized  $K_1$-groups'). However, applications of this sort relating to special values conjectures relative to non-commutative coefficient rings will be considered elsewhere.

\begin{acknowledgments}
{\em Much of this article developed from earlier joint work of ours with Masato Kurihara and we are extremely grateful to him for encouragement and many helpful  discussions. In addition, the first author is very grateful to Kazuya Kato for his generous encouragement at an early stage of this project (long ago). It is also a pleasure to thank Henri Johnston, Daniel Macias Castillo and Andreas Nickel for  helpful comments and Daniel Puignau for a careful reading of an earlier version of this article. Finally, we are very grateful indeed to the referee for a very thorough report and, in particular, for numerous suggestions that improved the exposition considerably.
%
%Finally, we would like to point out that much of this article comprises updated versions of results obtained in a previously circulated, but  unpublished, preprint of ours entitled `On non-abelian zeta elements for $\mathbb{G}_m$'.
}
\end{acknowledgments}

\section{Semisimple algebras}\label{semi alg section}

For a ring $R$ we write $R^{\rm op}$ for its opposite ring and $\zeta(R)$ for its centre (so that $\zeta(R^{\rm op}) = \zeta(R)$). For natural numbers $d'$ and $d$ we write ${\rm M}_{d',d}(R)$ for the set of $d'\times d$ matrices over $R$. We abbreviate ${\rm M}_{d,d}(R)$ to ${\rm M}_{d}(R)$ and write ${\rm GL}_d(R)$ for its unit group.
%\medskip

\subsection{Simple rings}\label{Semisimple rings}
\subsubsection{}\label{morita} We first review relevant facts from Morita theory (and for more details see, for example, \cite[\S3]{curtisr}). %, restricting in an important special case.

Let $E$ be a field and $V$ an $E$-vector space of dimension $d$. Then $V$ is naturally a (simple) left module over the $E$-algebra $A :=\End_E(V)$.
 The linear dual $V^\ast := \Hom_E(V,E)$ of $V$ is a right $A$-module via the rule
$$( v^\ast \cdot a )(v):=v^\ast (a\cdot v),$$
for $a\in A$, $v^\ast \in V^\ast$, and $v \in V$. There are also pairings
$$(- , -)_E : V^\ast \times V \to E \quad \text{and}\quad (- , -)_A : V \times V^\ast \to A,$$
given, for $v,v' \in V$ and $v^\ast \in V^\ast$, by
$$ (v^\ast, v)_E:=v^\ast(v) \quad\text{and}\quad (v,v^\ast)_A(v'):=v^\ast(v')\cdot v.$$

The pairing $(-, -)_E$, respectively $(-, -)_A$, induces an isomorphism of $E$-vector spaces, respectively two-sided $A$-modules of the form
$$V^\ast \otimes_{A} V \stackrel{\sim}{\rightarrow} E, \quad\text{respectively}\quad V \otimes_E V^\ast \stackrel{\sim}{\rightarrow} A.$$

The `Morita functor' $V^\ast \otimes_{A} -$ from the category of left $A$-modules to that of $E$-vector spaces gives an equivalence of categories.

\subsubsection{}Let now $K$ be a field of characteristic zero and $A$ a finite-dimensional simple $K$-algebra. All simple left $A$-modules are isomorphic and for any such module $M$ the $K$-algebra
\begin{equation}\label{div ring def} D:=\End_A(M)\end{equation}
is a division ring (that is unique up to isomorphism). By using a slightly more general version of the Morita theory recalled above one derives a canonical ring isomorphism
$$A \stackrel{\sim}{\rightarrow} \End_D(M) ; \ a \mapsto (m \mapsto am).$$

The centre $F:=\zeta(D)$ of $D$ is a field canonically isomorphic to $\zeta(A)$.  An extension field $E$ of $F$ is a `splitting field' for $A$ if  $D\otimes_F E$ (or, equivalently, $A \otimes_F E$) is isomorphic to a matrix ring ${\rm M}_m(E)$ for some $m$.  %If $E$ is such a field, then we say ``$E$ splits $A$". If $F$ splits $A$, or equivalently, if $D=F$, then we say`$A$ is split'. 
Such a field $E$ always exists and can be taken to be of finite degree over $K$ (see Remark \ref{exp splitting} below). In addition, the integer $m$ is independent of $E$ and  referred to as the `Schur index' of $A$. Finally, we recall that there exists a composite isomorphism
\begin{equation}\label{rn iso}A \otimes_F E \cong \End_D(M) \otimes_F E \cong {\rm M}_n(D^{\rm op}) \otimes_F E \cong {\rm M}_n({\rm M}_m(E))={\rm M}_{nm}(E),\end{equation}
where $n$ is the dimension of the (free) left $D$-module $M$, via which one can regard $A$ as a subalgebra of ${\rm M}_{nm}(E)$.

\begin{remark}\label{exp splitting} {\em Any choice of an algebraic closure $K^c$ of $K$ is a splitting field for $A$. In addition, there are canonical choices of finite extensions of $K$ in $K^c$ that are splitting fields for $A$. % ext powers sec} will involve the choice of splitting fields. Whilst these constructions vary functorially with changes of splitting field, it is also possible to specify a canonical choice.
 For example, the composite of all extensions of $K$ in $K^c$ that are isomorphic (as a $K$-algebra) to a maximal subfield of any division ring $D$ as in (\ref{div ring def}) is a splitting field for $A$ that is of finite degree and Galois over $K$. In this regard see also Remark \ref{group ring rem}.}\end{remark}

\subsubsection{}The behaviour under scalar extension of a finite-dimensional simple $K$-algebra $A$ is described in the following result.

\begin{lemma}\label{propsch0}
Let $K'$ be an extension of $K$ and $\Omega$ an algebraic closure of $K'$.  Set $F:= \zeta(A)$ and consider the (finite) set $\Sigma(F/K,K')$ of equivalence classes of $K$-embeddings $F \to \Omega$ under the relation $\sigma \sim \sigma' \Longleftrightarrow \sigma=\tau\circ \sigma' \,\mbox{ for some } \tau \in \Aut_{K'}(\Omega).$

Then, for each $\sigma$ in $\Sigma(F/K,K')$, the $K'$-algebra $A\otimes_{F}\sigma(F)K'$ is a simple artinian ring with centre the composite field $\sigma(F)K'$ of $\sigma(F)$ and $K'$ (this field is independent of the choice of $\sigma$), and there is a product decomposition of $K'$-algebras
$$A\otimes_KK'\cong {\prod}_{\sigma \in \Sigma(F/K,K')} (A\otimes_{F}\sigma(F)K').$$
\end{lemma}

\begin{proof}
Since $F$ is separable over $K$, we have an isomorphism
$$F\otimes_K K' \cong {\prod}_{\sigma \in \Sigma(F/K,K')}\sigma(F)K'.$$
Hence we have
$$A \otimes_K K'\cong A \otimes_{F}(F\otimes_K K')\cong {\prod}_{\sigma \in \Sigma(F/K,K')}(A \otimes_{F} \sigma(F)K').$$
Since $A$ is a central simple algebra over $F$, $A \otimes_{F}\sigma(F)K'$ is also a central simple algebra over $\sigma(F)K'$.
\end{proof}

\subsubsection{} If $M$ is a finitely generated left $A$-module, then there exists a canonical composite homomorphism
$$\End_A(M) \to \End_{A\otimes_FE}(M\otimes_FE) \stackrel{\sim}{\to} \End_E(V^\ast \otimes_{(A\otimes_FE)}(M\otimes_FE)) \stackrel{}{\to} E,$$
with the second map induced by the Morita functor and the last by taking determinants.

One checks that the image of this map factors through the inclusion $F \subseteq E$ and that the induced `reduced norm' map 
$\nr_{\End_A(M)} :\End_A(M) \to F$ is such that, for all $\theta_1$ and $\theta_2$ in $\End_A(M)$, one has
$\nr_{\End_A(M)}(\theta_1\circ\theta_2) = \nr_{\End_A(M)}(\theta_1)\cdot\nr_{\End_A(M)}(\theta_2).$

\begin{remark}{\em  If $M=A^{\rm op}$ and one identifies $A$ with $\End_{A^{\rm op}}(A^{\rm op})$,
then for each element $a$ of $A$ one can check that $\nr_A(a)$ is equal to the determinant of the image of $a$ under the isomorphism (\ref{rn iso}). This is the classical definition of reduced norm. }
\end{remark}

\begin{remark}{\em For a natural number $n$ we shall often abbreviate $\nr_{{\rm M}_n(A)}$ to $\nr_A$. Since the algebras ${\rm M}_n(A^{\rm op})$ and ${\rm M}_n(A)^{\rm op}$ are isomorphic and $\nr_A=\nr_{A^{\rm op}}$, we shall also sometimes write $\nr_A$ for $\nr_{{\rm M}_n(A^{\rm op})}$.}
\end{remark}

The `reduced rank' of a finitely generated left $A$-module $M$ is the non-negative integer obtained by setting
\begin{equation}\label{rr def} {\rm rr}_A(M) := {\rm dim}_E\bigl(V^\ast \otimes_{(A\otimes_FE)}(M\otimes_FE)\bigr).\end{equation}

\begin{remark}\label{reduced rank rem}{\em  One can check, by explicit computation, that if $M$ is a simple left $A$-module, then ${\rm rr}_A(M)$ is equal to the Schur index $\sqrt{{\rm dim}_{F}(D)}$ of $A$. This implies, in particular, that ${\rm rr}_A(A)$ is equal to ${\rm dim}_D(M)\cdot \sqrt{{\rm dim}_{F}(D)}$ for any simple left $A$-module $M$.   }\end{remark}

\subsection{Semisimple rings}\label{ss rings section}
 In the sequel we shall use `module' to mean `left module'.
%\subsubsection{}

A module $M$ over a ring $A$ is said to be semisimple if it is a direct sum of
simple modules. A ring $A$ is said to be semisimple if every nonzero $A$-module is semisimple and this is true if and only if $A$ decomposes as a direct product
\begin{equation}\label{wedderburn} A \cong {\prod}_{i\in I} A_i,\end{equation}
in which the index set $I$ is finite and the rings $A_i$ are simple Artinian (and unique up to isomorphism). In particular, Lemma \ref{propsch0} shows that simple rings naturally give rise to semisimple rings under scalar extension.

The `Wedderburn decomposition' (\ref{wedderburn}) of $A$ induces an identification $\zeta(A) = {\prod}_{i \in I}\zeta(A_i)$ and can be used to define
%One sees that the definition of $\nr_{\End_A(M)}$ above coincides with the reduced norm of semisimple ring $\End_A(M)$.
(componentwise) generalizations of the above notions of reduced norm and reduced rank. In this way one obtains a reduced norm ${\rm Nrd}_A$ for the algebra $A$ that is valued in $\zeta(A)$ and defines a reduced rank  ${\rm rr}_A(M)$ of a finitely generated $A$-module $M$ that is an integer-valued function on ${\rm Spec}(\zeta(A))$.

This reduced norm induces a homomorphism (which we denote by the same symbol)
$$\nr_A:K_1(A) \to \zeta(A)^\times$$
from the Whitehead group $K_1(A)$ of $A$ (cf. \cite[\S 45A]{curtisr}).

\section{Whitehead orders and non-commutative Fitting invariants}\label{app higher fits}

%For the convenience of the reader, in this section we quickly recall the definition and relevant properties of the higher non-commutative Fitting invariants that are introduced by Sano and the first author in \cite{bs1}.

In this section we define a canonical $R$-order in $\zeta(A)$ and then use it to construct a non-commutative generalization of the classical theory of `higher Fitting ideals' (from \cite{North}). This construction is natural, has many of the same properties as the classical commutative construction (see Proposition \ref{main fit result}) and is also, as we show in \cite{sbes1-2}, well-suited to arithmetic applications.

%We fix a Dedekind domain $R$ of characteristic zero and write $F$ for its quotient field. We also fix a finite-dimensional semisimple $F$-algebra $A$ and an $R$-order $\mathcal{A}$ in $A$.

Throughout the section we fix a Dedekind domain $R$ with field of fractions $F$ that is a finite extension of either $\QQ$ or $\QQ_p$ for some prime $p$.  We also fix a finite-dimensional semisimple $F$-algebra $A$ and an $R$-order $\mathcal{A}$ in $A$  (in the sense of \cite[Def. (23.2)]{curtisr}).

For each prime ideal $\mathfrak{p}$ of $R$ we respectively write $R_{(\mathfrak{p})}$ and $R_\mathfrak{p}$ for the localization and completion of $R$ at $\mathfrak{p}$. %For any finite set of prime ideals $\mathcal{P}$ of $R$ we set $R_{\langle\mathcal{P}\rangle} := \bigcap_\mathfrak{p} R_{(\mathfrak{p})}$ where the intersection is over all prime ideals of $R$ that do not belong to $\mathcal{P}$.
 For each $\mathcal{A}$-module $M$ and each $\mathfrak{p}$ %and each finite set of prime ideals $\mathcal{P}$ of $R$
 we then set $M_{(\mathfrak{p})} := R_{(\mathfrak{p})}\otimes_R M$ and  $M_\mathfrak{p}:= R_\mathfrak{p}\otimes_R M$. % and $M_{\langle\mathcal{P}\rangle}:= R_{\langle\mathcal{P}\rangle}\otimes_R M$.
 We regard these modules as respectively endowed with natural actions of the algebras $\mathcal{A}_{(\mathfrak{p})} =R_{(\mathfrak{p})} \otimes_R \mathcal{A}$ and $\mathcal{A}_{\mathfrak{p}} =R_\mathfrak{p}\otimes_R \mathcal{A}$. %and $\mathcal{A}_{\langle\mathcal{P}\rangle}$
In particular, the localisation $M_{(0)}$ of $M$ at the zero prime ideal of $R$ is equal to the $A$-module generated by $M$ and will often be written as  $M_F$.

We recall that, if $F$ is a $p$-adic field (for some $p$), then $A$ is said to be `ramified' if some simple component in its Wedderburn decomposition (as an $F$-algebra) is a matrix ring over a non-commutative division algebra. For such fields $F$, we set ${\rm Ram}(A) = \{\mathfrak{p}\}$, with $\mathfrak{p}$ the unique maximal ideal of $R$, if $A$ is ramified, and we set ${\rm Ram}(A) = \emptyset$ if $A$ is not ramified. If $F$ is a number field, then there are only finitely many non-archimedean places $v$ of $F$ for which the $F_v$-algebra $F_v\otimes_F A$ is ramified  (cf. \cite[Th. (25.7)]{reiner}) and we write ${\rm Ram}(A)$ for the (finite) set of prime ideals of $R$ that correspond to these places.\\

\subsection{The Whitehead order}\label{integral structures section}

%We shall extend an idea used (in the commutative case) by Rubin in \cite{R} to introduce, for each finitely generated $\mathcal{A}$-module $M$ and each non-negative integer $r$, a canonical integral structure on the reduced  exterior power ${\bigwedge}_{A}^r M_F$.

\subsubsection{}We first introduce a canonical $R$-submodule of $\zeta(A)$.

\begin{definition}{\em For each prime ideal $\mathfrak{p}$ of $R$ the `Whitehead order' $\xi(\mathcal{A}_{(\mathfrak{p})})$ of $\mathcal{A}_{(\mathfrak{p})}$ is the $R_{(\mathfrak{p})}$-submodule of $\zeta(A)$ that is generated by the elements ${\rm Nrd}_A(M)$ as $M$ runs over all matrices in $\bigcup_{n \ge 1}{\rm M}_n(\mathcal{A}_{(\mathfrak{p})})$.

The `Whitehead order' of $\mathcal{A}$ is then defined by the intersection
\[ \xi(\mathcal{A}) := {\bigcap}_{\mathfrak{p}\in {\rm Spec}(R)} \xi(\mathcal{A}_{(\mathfrak{p})}).\]
%
%where $\mathfrak{p}$ runs over all prime ideals of $R$.
}\end{definition}

The basic properties of this module are described in the following result.

\begin{lemma}\label{xi lemma} The following claims are valid. \
\begin{itemize}
\item[(i)] $\xi(\mathcal{A})$ is an $R$-order in $\zeta(A)$.
\item[(ii)] For every prime ideal $\mathfrak{p}$ of $R$ there are equalities 
\[ \xi(\mathcal{A})_{(\mathfrak{p})} = \xi(\mathcal{A}_{(\mathfrak{p})}),\,\,\,\xi(\mathcal{A})_{\mathfrak{p}} = \xi(\mathcal{A}_{\mathfrak{p}})\,\,\,\text{ and }\,\,\, \xi(\mathcal{A}_{(\mathfrak{p})}) = \zeta(A)\cap \xi(\mathcal{A}_{\mathfrak{p}}).\]
\item[(iii)] If $\mathcal{A}$ is commutative, then $\xi(\mathcal{A}) = \zeta(\mathcal{A}) =\mathcal{A}$.
 \item[(iv)] If $\mathfrak{p}\notin {\rm Ram}(A)$ and $\mathcal{A}_{(\mathfrak{p})}$ is a maximal $R_{(\mathfrak{p})}$-order, then $\xi(\mathcal{A})_{(\mathfrak{p})}$ is the (unique) maximal $R_{(\mathfrak{p})}$-order in $\zeta(A)$. %Further, the group $\xi(\mathcal{M})^\times$ has finite $2$-power index in $\mathfrak{M}^\times$. If $R$ is a discrete valuation ring, then $\xi(\mathcal{M})^\times = {\rm Nrd}_A(\mathcal{M}^\times) = \zeta(\mathcal{M})^\times$.
\item[(v)] Any surjective homomorphism of $R$-orders $\varrho: \mathcal{A} \to \mathcal{B}$ induces, upon restriction, a
surjective homomorphism $\xi(\mathcal{A}) \to \xi(\mathcal{B})$.
 \end{itemize}
\end{lemma}

\begin{proof} We first make some preliminary observations. 

The integral closure $\mathfrak{M}$ of $R$ in $\zeta(A)$ is the maximal $R$-order in $\zeta(A)$. For every maximal ideal $\mathfrak{p}$ of $R$, the ring $\mathfrak{M}_{(\mathfrak{p})}$ is therefore the maximal $R_{(\mathfrak{p})}$-order in $\zeta(A)$ and hence a direct product of discrete valuation rings. In particular, since ${\rm M}_n(\mathcal{A}_{(\mathfrak{p})})$ is an $R_{(\mathfrak{p})}$-order in ${\rm M}_n(A)$, one has ${\rm Nrd}_A(M) \in \mathfrak{M}_{(\mathfrak{p})}$ for every $M \in {\rm M}_n(\mathcal{A}_{(\mathfrak{p})})$ (cf. \cite[Cor. (26.2)]{curtisr}). It follows that $\xi(\mathcal{A}_{(\mathfrak{p})})\subseteq \mathfrak{M}_{(\mathfrak{p})}$ and hence that $\xi(\mathcal{A}_{(\mathfrak{p})})$ is finitely generated over $R_{(\mathfrak{p})}$.  

Next we fix a choice of maximal $R$-order $\mathcal{M}$ in $A$ that contains $\mathcal{A}$ as a submodule of finite index (cf.  \cite[Th. (26.5)]{curtisr}). Then, for every maximal ideal $\mathfrak{p}$ of $R$, the $R_{(\mathfrak{p})}$-order $\mathcal{M}_{(\mathfrak{p})}$ is maximal (by \cite[Th. (26.21)(ii)]{curtisr}). Hence, if $\mathfrak{p} \notin {\rm Ram}(A)$, then $\mathcal{M}_{(\mathfrak{p})}$ is conjugate in $A$ to a direct product of full matrix rings over the respective (discrete valuation ring) components of $\mathfrak{M}_{(\mathfrak{p})}$ (cf. \cite[Th. (18.7)(iii)]{reiner}) and so the reduced norm (over $A$) coincides with taking the respective matrix determinants on each component of this product. In particular, if both $\mathfrak{p}\notin {\rm Ram}(A)$ and the order $\mathcal{A}_{(\mathfrak{p})}$ is maximal, then one has $\zeta(\mathcal{A})_{(\mathfrak{p})} = \mathfrak{M}_{(\mathfrak{p})} \subseteq \xi(\mathcal{A}_{(\mathfrak{p})})$ and hence 
$\xi(\mathcal{A}_{(\mathfrak{p})}) = \zeta(\mathcal{A})_{(\mathfrak{p})} = \mathfrak{M}_{(\mathfrak{p})}$, as required to prove (iv). 

%In particular, for every $\mathfrak{p}$ the $R_{(\mathfrak{p})}$-module $\xi(\mathcal{A}_{(\mathfrak{p})})$ is contained in $\mathfrak{M}_{(\mathfrak{p})}$ and so is finitely generated. 
Finally, we note it is straightforward to check $\xi(\mathcal{A}_{(\mathfrak{p})})$ is closed under multiplication, that it contains $R_{(\mathfrak{p})}$ and that it has finite index in $\mathfrak{M}_{(\mathfrak{p})}$ (since it contains ${\rm Nrd}_A(x\cdot \mathcal{M}_{(\mathfrak{p})})$ for a non-zero element of $R$). This shows each $\xi(\mathcal{A}_{(\mathfrak{p})})$ is an $R_{(\mathfrak{p})}$-order in $\zeta(A)$ and hence also implies that $\xi(\mathcal{A})$ is an $R$-order in $\zeta(A)$ if and only if one has $\xi(\mathcal{A})_F = \zeta(A)$.

At this stage, it is clear that (i) and the first assertion of (ii) both follow from the general result  of \cite[Prop. (4.21)(vii)]{curtisr} and the fact that $\xi(\mathcal{A}_{(\mathfrak{p})}) = \mathfrak{M}_{(\mathfrak{p})}$ for almost all $\mathfrak{p}$ (as follows directly from the above observations).
%order in $\zeta(A)$ since subring of $\zeta(A)$ that contains $R_{(\mathfrak{p})}$ and has finite index in $\zeta(A)$. so is an $R_{(\mathfrak{p})}$-order if and only if it is finitely generated over $R_{(\mathfrak{p})}$. This is true since for any $n > 0$ and any matrix $M$ in ${\rm M}_n(\mathcal{A}_{(\mathfrak{p})})$ the element ${\rm Nrd}_A(M)$ is integral over $R_{(\mathfrak{p})}$.

In order to prove $\xi(\mathcal{A}_{\mathfrak{p}}) = \xi(\mathcal{A})_{\mathfrak{p}}$, we note that (i) implies $\xi(\mathcal{A}_{\mathfrak{p}})$ is a finitely generated $R_\mathfrak{p}$-module and hence $\mathfrak{p}$-adically complete. Since any matrix $M$ in ${\rm M}_n(\mathcal{A}_{\mathfrak{p}})$ is equal to the $\mathfrak{p}$-adic limit of a sequence of matrices $(M_t)_t$ in ${\rm M}_n(\mathcal{A})$, it is therefore enough to show that ${\rm Nrd}_{A_\mathfrak{p}}(M)$ is equal to the limit (over $t$)  of the associated reduced norms 
${\rm Nrd}_{A}(M_t) = {\rm Nrd}_{A_\mathfrak{p}}(M_t)\in \xi(\mathcal{A}_{\mathfrak{p}})$. This in turn follows directly from the fact that, for any natural number $a$, there exists a natural number $b(a)$ such that for all matrices $N$ and $N'$ in ${\rm M}_n(\mathcal{A}_\mathfrak{p})$ and all integers $b \ge b(a)$ one has an implication 
\[ N-N' \in \mathfrak{p}^{b}\cdot {\rm M}_n(\mathcal{A}_\mathfrak{p}) \Longrightarrow {\rm Nrd}_{A_\mathfrak{p}}(N) - {\rm Nrd}_{A_\mathfrak{p}}(N') \in \mathfrak{p}^a\cdot \zeta(\mathcal{A}_\mathfrak{p}).\]
(The verification of the latter fact is a straightforward exercise involving the explicit description of reduced norms that we leave to the reader.)

To complete the proof of (ii), it is now enough to note that the equalities $\xi(\mathcal{A})_F = \zeta(A)$, 
$\xi(\mathcal{A}_{\mathfrak{p}}) = \xi(\mathcal{A})_{\mathfrak{p}}$ and $\xi(\mathcal{A})_{(\mathfrak{p})} = \xi(\mathcal{A}_{(\mathfrak{p})})$ already proved combine to imply that 
\[ \zeta(A)\cap \xi(\mathcal{A}_\mathfrak{p}) = \xi(\mathcal{A})_F\cap \xi(\mathcal{A})_\mathfrak{p} = \xi(\mathcal{A})_{(\mathfrak{p})} = \xi(\mathcal{A}_{(\mathfrak{p})}),\]
where the second equality follows from the general result of \cite[Prop. (4.21)(vi)]{curtisr}.

Next we note that if $\mathcal{A}$ is commutative, then for every $\mathfrak{p}$ and every  matrix $M$ in ${\rm M}_n(\mathcal{A}_{(\mathfrak{p})})$ one has ${\rm Nrd}_A(M) = {\rm det}(M) \in \mathcal{A}_{(\mathfrak{p})}$. In this case it is therefore clear that $\xi(\mathcal{A})$ is equal to ${\bigcap}_{\mathfrak{p}}\mathcal{A}_{(\mathfrak{p})}$ and hence (by \cite[Prop. (4.21)(vi)]{curtisr}) to $\mathcal{A} = \zeta(\mathcal{A})$, as required to prove (iii).

%To prove claim (iii) we assume that $\mathcal{A}$ is equal to a maximal $R$-order $\mathcal{M}$.
%
%The first assertion of claim (iii) is true because  $\zeta(\mathcal{M})$ is equal to the maximal $R$-order $\Lambda$ of $\zeta(A)$. For the same reason, the second assertion follows from the fact that the subgroup of $\Lambda^\times$ generated by the set $\{{\rm Nrd}_A(M): M \in {\rm GL}_n(\mathcal{M}), n \ge 1\}$ has finite $2$-power index (as a consequence of \cite[Th. (45.7)]{curtisr}) and the final assertion is true because if  $R$ is a discrete valuation ring, then the arithmetic of local division algebras implies ${\rm Nrd}_A(\mathcal{M}^\times) = \Lambda^\times$ (cf. \cite[Prop. (45.8)]{curtisr}), and hence that $\xi(\mathcal{M})^\times = \Lambda^\times = {\rm Nrd}_A(\mathcal{M}^\times)$.

%Claim (iv) is proved in the course of the general observations made at the beginning of this proof. % is true since if $\mathcal{A}$ is a maximal $R$-order in $A$, then $\mathcal{A}_{(\mathfrak{p})}$ is a maximal $R_{(\mathfrak{p})}$-order in $A$ for every $\mathfrak{p}$ and so
%
%\[ \xi(\mathcal{A}) = {\bigcap}_{\mathfrak{p}\in {\rm Spec}(R)}\xi(\mathcal{A}_{(\mathfrak{p})}) = {\bigcap}_{\mathfrak{p}\in {\rm Spec}(R)}\mathfrak{M}_{({\mathfrak{p}})} = \mathfrak{M}.\]
%
%where in both intersections $\mathfrak{p}$ runs over all prime ideals.
Finally, to prove (v) we note first that the claim makes sense since the surjectivity of $\varrho$ implies that the $F$-algebra $B= \mathcal{B}_F$ is a quotient of $A$ so that $B$ is semisimple (and hence the order $\xi(\mathcal{B})$ is defined). This also implies that $\varrho$ restricts to give a surjective homomorphism $\varrho': \zeta(A)\to \zeta(B)$ and (ii) implies that the claimed equality $\varrho'(\xi(\mathcal{A})) = \xi(\mathcal{B})$ is valid if for every $\mathfrak{p}$ one has $\varrho'(\xi(\mathcal{A}_{(\mathfrak{p})})) = \xi(\mathcal{B}_{(\mathfrak{p})})$. This equality is in turn true since $\varrho$ induces, for each $n$, a surjective ring homomorphism $\varrho_{n}: {\rm M}_n(\mathcal{A}_{(\mathfrak{p})}) \to {\rm M}_n(\mathcal{B}_{(\mathfrak{p})})$ with the property that $\varrho'({\rm Nrd}_A(M)) = {\rm Nrd}_B(\varrho_{n}(M))$ for every $M$ in ${\rm M}_n(\mathcal{A}_{(\mathfrak{p})})$. \end{proof}

\begin{remark}\label{JN comp rem}{\em  In the case that $R$ is a discrete valuation ring, Johnston and Nickel \cite[\S3.4]{JN} consider the $R$-order $\mathcal{I}(\mathcal{A})$ in $\zeta(A)$ that is generated over $\zeta(\mathcal{A})$ by the elements ${\rm Nrd}_A(M)$ as $M$ runs over all  matrices in $\bigcup_{n \ge 1}{\rm M}_n(\mathcal{A})$. In this case one  therefore has $\mathcal{I}(\mathcal{A}) = \zeta(\mathcal{A})\cdot \xi(\mathcal{A})$ and also $\mathcal{I}(\mathcal{A}) = \xi(\mathcal{A})$ if and only if $\zeta(\mathcal{A})\subseteq \xi(\mathcal{A})$. Whilst it certainly seems possible that there exist $R$-orders $\mathcal{A}$ for which $\zeta(\mathcal{A})\not\subset \xi(\mathcal{A})$, %(and hence $\mathcal{I}(\mathcal{A}) \not= \xi(\mathcal{A}))$
 at this stage we do not  know a concrete example for which this is true. }\end{remark}

%At this stage we One therefore has $\xi(\mathca
 %that is hat occurs in the work of Nickel \cite{N3} and Johnston and Nickel \cite{JN} relating to non-commutative initial Fitting invariants.

\begin{example}\label{dj examples} {\em Assume $\mathcal{A} = R[\Gamma]$ for a discrete valuation ring $R$ that has residue characteristic $p$ (and quotient field $F$) and a finite group $\Gamma$. Then, under certain hypotheses on $\Gamma$, such as in the  following examples, the order $\xi(\mathcal{A})$ can be described explicitly. 

(i) If $\Gamma$ has an abelian Sylow $p$-subgroup and a normal $p$-complement (or,  equivalently, $p$ does not divide the order of the commutator subgroup of $\Gamma$), then $\mathcal{A}$ is a direct product of matrix rings over commutative $R$-algebras %and hence Azumaya over $\ZZ_p$ 
(cf. Demeyer and Janusz \cite[p. 390, Cor]{dj}). In this case, an explicit computation of reduced norms shows that  $\xi(\mathcal{A}) = \zeta(\mathcal{A})$. 

(ii) If $\Gamma$ is the dihedral group of order $2p$, then the computations in \cite[\S6.4, Exam. 6]{JN} show that $\xi(\mathcal{A})$ is the maximal $R$-order in $\zeta(F[\Gamma])$ and hence properly contains $\zeta(\mathcal{A})$. }
\end{example}

\begin{example}\label{azumaya exam}{\em As a generalization of Example \ref{dj examples}(i), assume that $R$ is a discrete valuation ring, $S$ is a finitely generated $R$-submodule of $\mathcal{A}$ that is a commutative local ring and that the $R$-order $\mathcal{A}$ is an Azumaya algebra over $S$ with $\zeta(\mathcal{A}) = S$. Then the  maximal commutative subalgebra $S'$ of $\mathcal{A}$ that is separable over $S$ is a projective $S$-module and the algebra $S'\otimes_S\mathcal{A}$ is isomorphic to ${\rm M}_t(S')$ for some natural number $t$ (cf. \cite[Lem. 5.1.17]{knus}). Using this isomorphism one can show that ${\rm Nrd}_A(M)$ belongs to $S'\cap S_F = S = \zeta(\mathcal{A})$ for every $M$ in ${\rm M}_n(\mathcal{A})$ and hence $\xi(\mathcal{A}) \subseteq \zeta(\mathcal{A})$. }\end{example}

\subsubsection{}\label{denominator ideal section} As Example \ref{dj examples}(ii) demonstrates, the order  $\xi(\mathcal{A})$ is not, in general, contained in $\zeta(\mathcal{A})$. However, to bound the `denominators' of its elements one can proceed as follows.

For each natural number $m$ and matrix $M$ in ${\rm M}_m(A)$ there exists a unique matrix $M^*$ in ${\rm M}_m(A)$ with
\begin{equation}\label{adjoint matrix} M\cdot M^* = M^*\cdot M = {\rm Nrd}_{A}(M)\cdot I_m\end{equation}
and such that, for every primitive central idempotent $e$ of $A$, one has 
\begin{equation}\label{adjoint matrix2} M^*e \not= 0 \Longleftrightarrow {\rm Nrd}_{A}(M)e \not= 0.\end{equation} 
To be explicit, the latter condition on $e$ is equivalent to the invertibility of $Me$ in ${\rm M}_m(A)e = {\rm M}_m(Ae)$ (cf. \cite[\S7, Exer. 4]{curtisr}) and we define $M^\ast$ so that $M^*e = (Me)^{-1}{\rm Nrd}_{A}(M)$. 

The following definition is motivated by a result \cite[Th. 4.2]{N3} of Nickel (see, in particular, Lemma \ref{module lemma}(iii) and, especially, the result of Theorem  \ref{main fit result}(iii) below).

\begin{definition} \label{def denom}
{\em For each prime ideal $\mathfrak{p}$ of $R$ the `ideal of denominators' of  $\mathcal{A}_{(\mathfrak{p})}$ is the subset of $\zeta(A)$ obtained by setting
\[ \delta(\mathcal{A}_{(\mathfrak{p})}) := \{ x\in \zeta(A): \forall\, d \ge 1, \, \forall\, M \in {\rm M}_d(\mathcal{A}_{(\mathfrak{p})}) \text{ one has }x\cdot M^* \in {\rm M}_d(\mathcal{A}_{(\mathfrak{p})})\}.\]
The `ideal of denominators' of $\mathcal{A}$ is then defined by the intersection
\[ \delta(\mathcal{A}) = {\bigcap}_{\mathfrak{p}\in {\rm Spec}(R)}\delta(\mathcal{A}_{(\mathfrak{p})})\]
%
%where $\mathfrak{p}$ runs over all primes. 
}\end{definition}

The basic properties of these sets are described in the following result.

%It is also useful to note that $\mathfrak{A}(\mathcal{A})$ is stable under multiplication by the order $\xi(\mathcal{A})$.

\begin{lemma}\label{module lemma}\
\begin{itemize}
\item[(i)] $\delta(\mathcal{A})$ is an ideal of finite index in $\zeta(\mathcal{A})$.
\item[(ii)] For every prime ideal $\mathfrak{p}$ of $R$ one has $\delta(\mathcal{A})_{(\mathfrak{p})} = \delta(\mathcal{A}_{(\mathfrak{p})})$.
\item[(iii)] For each prime $\mathfrak{p}$ an element $x$ of $\zeta(A)$ belongs to $\delta(\mathcal{A})_{(\mathfrak{p})}$ if and only if there exists a non-negative integer $m_x = m_{\mathfrak{p},x}$ such that for all $a\ge m_x$ and all $M \in {\rm M}_a(\mathcal{A}_{(\mathfrak{p})})$ one has $x\cdot M^* \in {\rm M}_a(\mathcal{A}_{(\mathfrak{p})})$.
\item[(iv)] $\delta(\mathcal{A})\cdot \xi(\mathcal{A}) = \delta(\mathcal{A})$.
\item[(v)] If $\mathcal{A}$ is commutative, then $\delta(\mathcal{A}) = \xi(\mathcal{A}) = \mathcal{A}$.
\item[(vi)] If $\mathfrak{p} \notin {\rm Ram}(A)$ and $\mathcal{A}_{(\mathfrak{p})}$ is a maximal $R_{(\mathfrak{p})}$-order, then $\delta(\mathcal{A})_{(\mathfrak{p})} = \xi(\mathcal{A})_{(\mathfrak{p})}$.
\item[(vii)] Any surjective homomorphism of $R$-orders $\varrho: \mathcal{A} \to \mathcal{B}$ induces, upon restriction, a
homomorphism $\delta(\mathcal{A}) \to \delta(\mathcal{B})$.
\end{itemize}
\end{lemma}

\begin{proof} For each $\mathfrak{p}$  the set $\delta(\mathcal{A}_{(\mathfrak{p})})$ is clearly an additive subgroup of $\zeta(A)$ that is stable under multiplication by $\zeta(\mathcal{A}_{(\mathfrak{p})})$. One also has $\delta(\mathcal{A}_{(\mathfrak{p})})\subseteq \zeta(\mathcal{A}_{(\mathfrak{p})})$ since if $M$ is the $1\times 1$ identity matrix, then $x = x\cdot M = x\cdot M^*$ and so $x = x\cdot M^*\in {\rm M}_1(\mathcal{A}_{(\mathfrak{p})})$ implies  $x \in \mathcal{A}_{(\mathfrak{p})}\cap \zeta(A) = \zeta(\mathcal{A}_{(\mathfrak{p})})$. This proves  $\delta(\mathcal{A}_{(\mathfrak{p})})$ is an ideal of $\zeta(\mathcal{A}_{(\mathfrak{p})})$ and we next show it has finite index.

We first consider the special case that $\mathcal{A}_{(\mathfrak{p})}$ is maximal (as will be the case for all but finitely many $\mathfrak{p}$). In this case $\zeta(\mathcal{A}_{(\mathfrak{p})}) = \mathfrak{M}_{(\mathfrak{p})}$, where $\mathfrak{M}$ is the integral closure of $R$ in $\zeta(A)$, and for every $M$ in ${\rm M}_d(\mathcal{A}_{(\mathfrak{p})})$ and each primitive idempotent $e$ of  $\zeta(A)$ for which ${\rm Nrd}(M)e$ is non-zero, the defining property (\ref{adjoint matrix}) implies that $eM^\ast$ belongs to ${\rm M}_d(e\mathcal{A}_{(\mathfrak{p})}) \subseteq {\rm M}_d(\mathcal{A}_{(\mathfrak{p})})$ (see, for example, the discussion of \cite[\S3.6]{JN}). In this case, therefore, it follows that $\delta(\mathcal{A}_{(\mathfrak{p})})$ contains, and is therefore equal to, $\mathfrak{M}_{(\mathfrak{p})}$. To deal with the general case, we fix a maximal $R_{(\mathfrak{p})}$-order $\mathcal{M}$ in $A$ that contains $\mathcal{A}_{(\mathfrak{p})}$ and write $n$ for the (finite) index of $\mathcal{A}_{(\mathfrak{p})}$ in $\mathcal{M}$. Then for each $M$ in ${\rm M}_{d}(\mathcal{A}_{(\mathfrak{p})})$, the above argument implies that $M^*$ belongs to ${\rm M}_d(\mathcal{M})$ and hence that $n\cdot M^\ast$ belongs to ${\rm M}_d(\mathcal{A}_{(\mathfrak{p})})$. This implies $n\cdot \zeta(\mathcal{A}_{(\mathfrak{p})}) \subseteq \delta(\mathcal{A}_{(\mathfrak{p})})$ and hence that  $\delta(\mathcal{A}_{(\mathfrak{p})})$ has finite index in $\zeta(\mathcal{A}_{(\mathfrak{p})})= \zeta(\mathcal{A})_{(\mathfrak{p})}$.

At this stage we know that $\delta(\mathcal{A})= \bigcap_\mathfrak{p}\delta(\mathcal{A}_{(\mathfrak{p})})$ is an ideal of $\bigcap_\mathfrak{p}\zeta(\mathcal{A})_{(\mathfrak{p})} = \zeta(\mathcal{A})$ and that (as a consequence of \cite[Prop. (4.21)(vi)]{curtisr}) its index is finite and for every $\mathfrak{p}$ one has $\delta(\mathcal{A})_{(\mathfrak{p})} = \delta(\mathcal{A}_{(\mathfrak{p})})$. This proves claims (i) and (ii).

To prove (iii) it obviously suffices (in view of (ii)) to show that the stated condition is sufficient to imply $x$ belongs to $\delta(\mathcal{A}_{(\mathfrak{p})})$. To do this we fix a natural number $d$ and a matrix $M$ in ${\rm M}_d(\mathcal{A}_{(\mathfrak{p})})$ and note that in ${\rm M}_{d+m_x}(A)$ one has

\[ x\begin{pmatrix}
 M& 0\\
0& {\rm I}_{m_x}\end{pmatrix}^{\!*} = x\begin{pmatrix}
M^*& 0\\
0& {\rm Nrd}_A(M)\cdot {\rm I}_{m_x}\end{pmatrix} = \begin{pmatrix}
x\cdot M^*& 0\\
0& x\cdot{\rm Nrd}_A(M)\cdot {\rm I}_{m_x}\end{pmatrix}.\]
In particular, since $d+m_x>m_x$, the stated condition on $x$ (with $a = d+m_x$ and $M$ replaced by $\begin{pmatrix}
 M& 0\\
0& {\rm I}_{m_x}\end{pmatrix}$) implies that $x\cdot M^*$ belongs to ${\rm M}_d(\mathcal{A}_{(\mathfrak{p})})$, as required.

In view of (ii) and Lemma \ref{xi lemma}(ii), it is enough to prove the equality in (iv) after replacing $\mathcal{A}$ by $\mathcal{A}_{(\mathfrak{p})}$ for each $\mathfrak{p}$. Since $1$ belongs to $\xi(\mathcal{A}_{(\mathfrak{p})})$, it is then enough to show that for any $x$ in $\delta(\mathcal{A}_{(\mathfrak{p})})$, any natural number $n$, and any matrix $N$ in ${\rm M}_n(\mathcal{A}_{(\mathfrak{p})})$, the element $x' := x\cdot {\rm Nrd}_A(N)$ belongs to $\delta(\mathcal{A}_{(\mathfrak{p})})$. We do this by showing that $x'$ satisfies the condition described in (iii) with $m_{x'}$ taken to be $n$.

We thus fix an integer $d$ with $d\ge n$ and choose $N'$ in ${\rm M}_{d}(\mathcal{A}_{(\mathfrak{p})})$ with ${\rm Nrd}_A(N') = {\rm Nrd}_A(N)$. Then, for any $M$ in ${\rm M}_d(\mathcal{A}_{(\mathfrak{p})})$ one has $M^*\cdot (N')^* = (N'\cdot M)^*$ and hence
\[ x' \cdot M^* = x\cdot {\rm Nrd}_A(N)M^* = x\cdot {\rm Nrd}_A(N')M^* =x\cdot M^*((N')^*N') = (x\cdot (N'M)^*) N'. \]
In particular, since $x$ belongs to $\delta(\mathcal{A}_{(\mathfrak{p})})$ one has $x\cdot (N'M)^*\in {\rm M}_d(\mathcal{A}_{(\mathfrak{p})})$, and hence $x'\cdot  M^*\in {\rm M}_d(\mathcal{A}_{(\mathfrak{p})})$, as required.

In view of (i) and Lemma \ref{xi lemma}(iii), (v) is reduced to showing that if $\mathcal{A}$ is commutative, then $\delta(\mathcal{A})$ contains $\mathcal{A}$. This follows directly from the fact that, in this case, for every prime $\mathfrak{p}$ and every $M$ in ${\rm M}_d(\mathcal{A}_{(\mathfrak{p})})$ the adjoint matrix $M^\ast$ also belongs to ${\rm M}_d(\mathcal{A}_{(\mathfrak{p})})$.

Claim (vi) is true since (as already observed above) if $\mathcal{A}_{(\mathfrak{p})}$ is a maximal $R_{(\mathfrak{p})}$-order in $A$, then $\delta(\mathcal{A})_{(\mathfrak{p})} = \delta(\mathcal{A}_{(\mathfrak{p})})$ is equal to $\mathfrak{M}_{(\mathfrak{p})}$ and hence, if $\mathfrak{p}\notin {\rm Ram}(A)$, to $\xi(\mathcal{A}_{(\mathfrak{p})}) = \xi(\mathcal{A})_{(\mathfrak{p})}$ by claims (ii) and (iv) of Lemma \ref{xi lemma}.

Finally, to prove (vii) we write $A$ and $B$ for the $F$-algebras that are respectively spanned by $\mathcal{A}$ and $\mathcal{B}$ and we consider the ring homomorphisms $\varrho': \zeta(A) \to \zeta(B)$ and $\varrho_d: {\rm M}_d(A) \to {\rm M}_d(B)$ for each natural number $d$ that are induced by $\varrho$.

It is enough to show that $\varrho'(\delta(\mathcal{A}_{(\mathfrak{p})})) = \delta(\mathcal{B}_{(\mathfrak{p})})$ for all $\mathfrak{p}$. Then, since  $\zeta(B)$ is a direct factor of $\zeta(A)$ (see the proof of Lemma
\ref{xi lemma}(v)), for any matrix $M$ in ${\rm M}_d(\mathcal{A}_{(\mathfrak{p})})$  the defining equality (\ref{adjoint matrix}) implies that $\varrho_{d}(M^\ast) = \varrho_d(M)^\ast$. By using this last equality, the required equality $\varrho'(\delta(\mathcal{A}_{(\mathfrak{p})})) = \delta(\mathcal{B}_{(\mathfrak{p})})$ follows directly from the definition of the respective ideals $\delta(\mathcal{A}_{(\mathfrak{p})})$ and $\delta(\mathcal{B}_{(\mathfrak{p})})$ and the fact that $\varrho_d\bigl({\rm M}_d(\mathcal{A}_{(\mathfrak{p})})\bigr) = {\rm M}_d(\mathcal{B}_{(\mathfrak{p})})$. \end{proof}

\begin{remark}{\em The ideal $\delta(\mathcal{A})$ defined above differs from an ideal $\mathcal{H}(\mathcal{A})$ defined (in the case $R$ is a discrete valuation ring) by Johnston and Nickel in \cite{JN} since our definition of the matrices $M^*$ via the conditions (\ref{adjoint matrix}) and (\ref{adjoint matrix2}) differs slightly from the `generalized adjoint matrices' used in loc. cit. To be specific, if $M\in {\rm M}_m(A)$ and $e$ is any primitive central idempotent of $A$ for which $Me$ is not invertible (over $Ae$), then one has $M^\ast e = 0$ whilst the $e$-component of the generalized adjoint matrix of $M$ defined in \cite[\S3.6]{JN} can be non-zero (for more details, see \cite[Rem. 10]{JN}). Despite this difference, however, the computations of $\mathcal{H}(\mathcal{A})$ in loc. cit. can be used to give concrete information about  $\delta(\mathcal{A})$, as the following examples show.  %=  \zeta(R_{(\mathfrak{p})}[\Gamma])$ and hence, by Lemma \ref{module lemma}(ii) and (iv), also $\mathfrak{d}(R[\Gamma])_{(\mathfrak{p})} = \zeta(R_{(\mathfrak{p})}[\Gamma]) = \zeta(R[\Gamma])_{(\mathfrak{p})} =\xi(R[\Gamma])_{(\mathfrak{p})}$.
}\end{remark}

\begin{example}\label{dj examples 2}{\em Let $\mathcal{A}$ be a group ring $R[\Gamma]$ of the form discussed in Example \ref{dj examples}. \

(i) If $\Gamma$ has an abelian Sylow $p$-subgroup and a normal $p$-complement, then the argument of \cite[Prop. 4.1]{JN} shows $\delta(\mathcal{A}) = \zeta(\mathcal{A})$ and hence, by Example \ref{dj examples}(i), that $\delta(\mathcal{A}) = \xi(\mathcal{A})$. \

(ii) If $\Gamma$ is the dihedral group of order $2p$, then Example \ref{dj examples}(ii) combines with claims (i) and (iv) of Lemma \ref{module lemma} to imply $\delta(\mathcal{A})$ is contained in the conductor of the maximal $R$-order of $\zeta(F[\Gamma])$ into $\zeta(\mathcal{A})$. In this case, therefore, $\delta(\mathcal{A})$  is a proper ideal of $\zeta(\mathcal{A})$. 
%Let $\Gamma$ be a finite group and $\mathfrak{p}$ a prime ideal of the Dedekind domain $R$ that has residue characteristic prime to the order of the commutator subgroup of $\Gamma$. Then the observation in Example \ref{dj examples}(i) combines with claim (ii) of Lemma \ref{xi lemma}(ii) and claims (ii) and (iv) of Lemma \ref{module lemma} to imply that $ \xi(R[\Gamma])_{(\mathfrak{p})} = \xi(R_{(\mathfrak{p})}[\Gamma]) = \zeta(R_{(\mathfrak{p})}[\Gamma]) = \delta(R_{(\mathfrak{p})}[\Gamma]) =\delta(R[\Gamma])_{(\mathfrak{p})}$.
}\end{example}

\subsection{Locally-free modules}\label{enoch section}

\begin{definition}\label{lf def}{\em A finitely generated module $M$ over an $R$-order $\mathcal{A}$ will be said to be `locally-free' if $M_{(\mathfrak{p})}$ is a free $\mathcal{A}_{(\mathfrak{p})}$-module, or equivalently (as an easy consequence of Maranda's Theorem - see \cite[Prop. (30.17)]{curtisr}) if $M_{\mathfrak{p}}$ is a free $\mathcal{A}_{\mathfrak{p}}$-module, for all prime ideals $\mathfrak{p}$ of $R$. In the sequel we write ${\rm Mod}^{\rm lf}(\mathcal{A})$ for the category of locally-free $\mathcal{A}$-modules.}
\end{definition}

For any module $M$ in ${\rm Mod}^{\rm lf}(\mathcal{A})$ the rank of the (finitely generated) free $\mathcal{A}_{(\mathfrak{p})}$-module $M_{(\mathfrak{p})}$ is independent of $\mathfrak{p}$ and equal to the rank of the (free) $A$-module $M_F$. We refer to this common rank as the  `rank' of $M$ and denote it by ${\rm rk}_\mathcal{A}(M)$. A locally-free $\mathcal{A}$-module of rank one is often referred to as an `invertible' $\mathcal{A}$-module. 

Since localization at $\mathfrak{p}$ is an exact functor a locally-free $\mathcal{A}$-module is projective. As the following examples show, there are also important cases for which the converse is  true.

\begin{example}\label{loc free exam}{\em  \

(i) If $\mathcal{A}=R$, % is a Dedekind domain, with quotient field $E$,
then every finitely generated torsion-free $\mathcal{A}$-module $M$ is locally-free, with ${\rm rk}_\mathcal{A}(M)$ equal to the dimension of the $F$-space spanned by $M$.

(ii) If $G$ is a finite group for which no prime divisor of $|G|$ is invertible in $R$ and $\mathcal{A} = R[G]$ then, by a fundamental result of Swan \cite{swan0} (see also \cite[Th. (32.11)]{curtisr}), a finitely generated projective $\mathcal{A}$-module is locally-free. For any such module $M$ the product ${\rm rk}_{R[G]}(M)\cdot |G|$ is equal to the dimension of the $F$-space spanned by $M$. 

(iii) There are also several classes of order $\mathcal{A}$ for which a finitely generated projective
 $\mathcal{A}$-module is locally-free if and only if it spans a free $A$-module. This is the case, for example, if $\mathcal{A}$ is commutative (cf. \cite[Prop. 35.7]{curtisr}), or if $\mathcal{A}_{(\mathfrak{p})}$ is a maximal $R_{(\mathfrak{p})}$-order in $A$ for every prime ideal $\mathfrak{p}$ of $R$ (cf. \cite[Th. 26.24(iii)]{curtisr}), or if $\mathcal{A} = R[G]$ for any finite group $G$ (cf. \cite[Th. 32.1]{curtisr}). 
}
\end{example}

\subsection{Fitting invariants of locally-free presentations}

\subsubsection{}Let $M$ be a matrix in ${\rm M}_{d', d}(A)$ with $d' \ge d$. Then for any integer $t$ with $0 \le t \le d$ and any $\varphi = (\varphi_i)_{1\le i\le t}$ in $\Hom_\mathcal{A}(\mathcal{A}^{d'},\mathcal{A})^t$ we write ${\rm SM}^{d}_{\varphi}(M)$ for the set of all $d\times d$ submatrices of the matrices $M(J,\varphi)$ that are obtained from $M$ by choosing any $t$-tuple of integers $J = \{i_1,i_2,\ldots , i_t\}$ with $1\le i_1< i_2< \cdots < i_t\le d$, and setting
\begin{equation}\label{fitting matrix} M(J,\varphi)_{ij} := \begin{cases} \varphi_{a}(b_i), &\text{if $j = i_a$ with $1\le a\le t$}\\
                            M_{ij}, &\text{otherwise.}\end{cases}\end{equation}
where, for a natural number $n$, we write $\{b_i\}_{1\le i\le n}$ for the standard basis of the free $\mathcal{A}$-module $\mathcal{A}^n$.

Then the set of all $d\times d$ submatrices of all matrices that are obtained from $M$ by replacing at most $a$ of its columns by arbitrary elements of $\mathcal{A}$ is equal to
\[ \mathfrak{S}^a(M) := \bigcup_{0\le t \le a}\bigcup_{\varphi\in \Hom_{\mathcal{A}}(\mathcal{A}^{d'},\mathcal{A})^t}
{\rm SM}^{d}_{\varphi}(M).\]
%
%\bigcup_{{0\le t \le a}\atop { \varphi\in \Hom_{\mathcal{A}}(\mathcal{A}^{d'},\mathcal{A})^t}}
We note, in particular, that $\mathfrak{S}^0(M)$ is the set of all $d\times d$ submatrices of $M$.

\begin{definition}\label{fit matrix def}{\em For any non-negative integer $a$ the `$a$-th (non-commutative) Fitting invariant of $M$' is the ideal of $\xi(\mathcal{A})$ obtained by setting}
\[ {\rm Fit}_{\mathcal{A}}^a(M) := \xi(\mathcal{A})\cdot \{{\rm Nrd}_{A}(N): N\in \mathfrak{S}^a(M)\}.\]
\end{definition}

\subsubsection{}\label{pres intro}A `free presentation' $\Pi$ of a finitely generated $\mathcal{A}$-module $X$ is an exact sequence of $\mathcal{A}$-modules of the form
\begin{equation}\label{pres} \Pi: \mathcal{A}^{r'_{\Pi}} \xrightarrow{\theta_\Pi} \mathcal{A}^{r_{\Pi}} \xrightarrow{\rho_\Pi} X \to 0 \end{equation}
in which (without loss of generality) one has $r'_{\Pi} \ge r_{\Pi}$. Such a presentation is said to be `quadratic' if $r'_{\Pi} = r_{\Pi}$. 
%
%We shall say that a presentation $\Pi'$ of an $\mathcal{A}$-module $X'$ is `finer' than $h$ if both $r_{h',1} = r_{h,1}$ and $r_{h',2} = r_{h,2}$ and there exists an automorphism of the $\mathcal{A}$-module $\mathcal{A}^{r_{h,2}}$ which induces a well-defined surjective homomorphism of $\mathcal{A}$-modules $X' \to X$.
%
%We say that $\mathcal{A}$-module presentations $h$ and $h'$ are equivalent if both $h$ is finer than $h'$ and $h'$ is finer than $h$ (and we note that in this case the $\mathcal{A}$-modules $X$ and $X'$ are isomorphic).
%
%where in the set $t$ runs over all non-negative integers with $t \le r$ and each $\varphi_i$ belongs to $\Hom_\mathcal{A}(\mathcal{A}^{d_1},\mathcal{A})$.
%
%For any non-negative integer $t$ we write $[n]_t$ for the set of subsets of $\{1, 2, \dots , n\}$ that are of cardinality ${\rm min}\{t,n\}$.
%
In all cases, we write $M_\Pi$ for the matrix of the homomorphism $\theta_\Pi$ with respect to the standard bases of $\mathcal{A}^{r'_{\Pi}}$ and  $\mathcal{A}^{r_{\Pi}}$. %  $\{b_i\}_{1\le i\le r_{\Pi,1}}$ and $\{b_i\}_{1\le i\le r_{\Pi,2}}$ of .

%\begin{definition}{\em For any non-negative integer $a$ the `$a$-th (non-commutative) Fitting invariant of the presentation $\Pi$' is the ideal ${\rm Fit}_{\mathcal{A}}^a(M_{\Pi})$ of  $\xi(\mathcal{A})$.} \end{definition}
%
%The `transpose' $\Pi^{\rm tr}$ of $\Pi$ is the presentation
%
%\[ R[\Gamma]^{d'} \xrightarrow{\Hom_R(\theta,R)} R[\Gamma]^{d} \xrightarrow{} {\rm cok}(\Hom_R(\theta,R)) \to 0,\]
%
%where for $n \in \{d,d'\}$ we use the standard identification $R[\Gamma]^n = \Hom_R(R[\Gamma]^n,R)$.

%\subsection{Fitting invariants of locally-free presentations}

\subsubsection{}\label{lfp def} A `locally-free presentation' $\Pi$ of a finitely generated $\mathcal{A}$-module $X$ is a collection of data of the following form:

\begin{itemize}
\item[$\bullet$] an exact sequence of $\mathcal{A}$-modules
\begin{equation}\label{pres seq} \Pi^{\rm seq}:  P' \xrightarrow{\theta_\Pi} P \xrightarrow{\rho_\Pi} X \to 0\end{equation}
in which $P'$ and $P$ belong to ${\rm Mod}^{\rm lf}(\mathcal{A})$;
\item[$\bullet$] for each prime ideal $\mathfrak{p}$ of $R$ fixed isomorphisms of $\mathcal{A}_{(\mathfrak{p})}$-modules
\[ \iota'_{\Pi,\mathfrak{p}}: P'_{(\mathfrak{p})} \cong \mathcal{A}_{(\mathfrak{p})}^{{\rm rk}_\mathcal{A}(P')}\,\,\text{ and }\,\, \iota_{\Pi,\mathfrak{p}}: P_{(\mathfrak{p})} \cong \mathcal{A}_{(\mathfrak{p})}^{{\rm rk}_\mathcal{A}(P)}.\]
\end{itemize}
Such a presentation will be said to be `locally-quadratic' if ${\rm rk}_\mathcal{A}(P') = {\rm rk}_\mathcal{A}(P)$.

\begin{example}\label{descent}{\em  Let $\varrho: \mathcal{A} \to \mathcal{B}$ be a surjective homomorphism of $R$-orders. Then the induced exact sequence of $\mathcal{B}$-modules
\[ \mathcal{B}\otimes_{\mathcal{A},\varrho}\Pi^{\rm seq}: \,\,  \mathcal{B}\otimes_{\mathcal{A},\varrho}P' \xrightarrow{{\rm id}\otimes \theta_\Pi} \mathcal{B}\otimes_{\mathcal{A},\varrho}P \xrightarrow{} \mathcal{B}\otimes_{\mathcal{A},\varrho}X \to 0\]
and isomorphisms $\mathcal{B}\otimes_{\mathcal{A},\varrho}\iota'_{\Pi,\mathfrak{p}}$ and
 $\mathcal{B}\otimes_{\mathcal{A},\varrho}\iota_{\Pi,\mathfrak{p}}$ together constitute a locally-free presentation $\mathcal{B}\otimes_{\mathcal{A},\varrho}\Pi$ of $\mathcal{B}\otimes_{\mathcal{A},\varrho}X$ that is locally-quadratic if $\Pi$ is locally-quadratic.}
\end{example}

%By localising the sequence $\Pi$ at a prime ideal $\mathfrak{p}$ of $R$ one obtains a free resolution $\Pi_{(\mathfrak{p})}$ of the $\mathcal{A}_{(\mathfrak{p})}$-module $X_{(\mathfrak{p})}$.

\begin{definition}{\em For each non-negative integer $a$, the $a$-th Fitting invariant of the locally-free presentation $\Pi$ is the ideal of $\xi(\mathcal{A})$ obtained by setting
\[ {\rm Fit}_{\mathcal{A}}^a(\Pi) := {\bigcap}_{\mathfrak{p}\in {\rm Spec}(R)}{\rm Fit}_{\mathcal{A}_{(\mathfrak{p})}}^a(M_{\Pi(\mathfrak{p})}).\]
where %the intersection is over all prime ideals $\mathfrak{p}$ of $R$ and 
$\Pi(\mathfrak{p})$ denotes the free resolution of the $\mathcal{A}_{(\mathfrak{p})}$-module $X_{(\mathfrak{p})}$ that is obtained by localising $\Pi^{\rm seq}$ and using the isomorphisms $\iota'_{\Pi,\mathfrak{p}}$ and $\iota_{\Pi,\mathfrak{p}}$.} %In the case that $R=\ZZ$
%we usually abbreviate ${\rm Fit}^a_{R[\Gamma]}(\Pi)$ to ${\rm Fit}^a_{\Gamma}(\Pi)$.
\end{definition}

%In the case that $\mathcal{A} = \ZZ[\Gamma]$ for a finite group $\Gamma$ we will often abbreviate ${\rm Fit}^a_{\mathcal{A}}(\Pi)$ to ${\rm Fit}^a_{\Gamma}(\Pi)$.

The basic properties of these ideals are recorded in the following result.

\begin{lemma}\label{key lfp} Let $\Pi$ be a locally-free presentation of an $\mathcal{A}$-module. %of a finitely generated $\mathcal{A}$-module $X$.
Then the following claims are valid for every non-negative integer $a$.

\begin{itemize}
\item[(i)] ${\rm Fit}_{\mathcal{A}}^a(\Pi)$ is contained in ${\rm Fit}_{\mathcal{A}}^{a+1}(\Pi)$.
%\item[(ii)] ${\rm Fit}^0_\mathcal{A}(\Pi)\subseteq {\rm pAnn}_\mathcal{A}(Z)$.
\item[(ii)] ${\rm Fit}_{\mathcal{A}}^a(\Pi) = \xi(\mathcal{A})$ for all  large enough $a$.
\item[(iii)] For any homomorphism $\varrho$ as in Example \ref{descent} one has %then the induced ring homomorphism
 %$\varrho': \xi(\mathcal{A})\to \xi(\mathcal{B})$ is such that
 $\varrho({\rm Fit}_{\mathcal{A}}^a(\Pi))\subseteq {\rm Fit}_{\mathcal{B}}^{a}(\mathcal{B}\otimes_{\mathcal{A},\varrho}\Pi)$.
\end{itemize}
\end{lemma}

\begin{proof} For each prime $\mathfrak{p}$  we set $M_\mathfrak{p} := M_{\Pi(\mathfrak{p})}$.

Then (i) follows directly from the fact that in the definition of the set of matrices $\mathfrak{S}^a(M_\mathfrak{p})$ that occurs in Definition \ref{fit matrix def} the variable $t$ runs over all integers in the range $0\le t\le a$.

For the module $P$ in (\ref{pres seq}) we set $n := {\rm rk}_{\mathcal{A}}(P)$. Then to prove (ii) it is enough to show that for every $a \ge n$ and every prime $\mathfrak{p}$ one has ${\rm Fit}_{\mathcal{A}_{(\mathfrak{p})}}^a(M_{\mathfrak{p}}) = \xi(\mathcal{A}_{(\mathfrak{p})})$. This is true because for any such $a$ the $n\times n$ identity matrix belongs to $\mathfrak{S}^a(M_{\mathfrak{p}})$. % of matrices that are used to define the $a$-Fitting invariant over $\mathcal{A}_{(\mathfrak{p})}$ of the matrix $M_{\Pi_{(\mathfrak{p})}}$.

In a similar way, (iii) is true since for every prime $\mathfrak{p}$ the induced projection map $\varrho_n: {\rm M}_{n}(\mathcal{A}) \to {\rm M}_{n}(\mathcal{B})$ sends any matrix in $\mathfrak{S}^a(M_{\mathfrak{p}})$ to a matrix in  $\mathfrak{S}^a(\varrho_n(M_{\mathfrak{p}}))$. \end{proof}

\subsubsection{}In the next result we explain the connection between this definition and the notion of non-commutative Fitting invariants of presentations introduced (in the case that $R$ is a discrete valuation ring and $a=0$) by Nickel in \cite{N3} and then subsequently studied by Johnston and Nickel in \cite{JN}.

\begin{proposition}\label{fitt0 invariants lemma} Assume that $R$ is a discrete valuation ring and let $\Pi$ be a free presentation of an $\mathcal{A}$-module $X$. Then all of the following claims are valid.
\begin{itemize}
\item[(i)] One has $\zeta(\mathcal{A})\cdot{\rm Fit}^0_\mathcal{A}(\Pi) = \xi(\mathcal{A})\cdot{\rm Fit}_\mathcal{A}(\Pi)$, where ${\rm Fit}_\mathcal{A}(\Pi)$ is the noncommutative Fitting invariant of Nickel.
\item[(ii)] If $\Pi$ is quadratic, then ${\rm Fit}^0_\mathcal{A}(\Pi)$ is equal to $\xi(\mathcal{A})\cdot {\rm Nrd}_A(M_{\Pi})$ and depends only on the isomorphism class of the $\mathcal{A}$-module $X$.
\item[(iii)] Let $0 \to X_1\to X_2 \to X_3\to 0$ be a short exact sequence of $\mathcal{A}$-modules. Then, if $X_1$ and $X_3$ have quadratic presentations $\Pi_1$ and $\Pi_3$, there exists a quadratic presentation $\Pi_2$ of $X_2$ and one has  
    ${\rm Fit}^0_\mathcal{A}(\Pi_2) = {\rm Fit}^0_\mathcal{A}(\Pi_1)\cdot {\rm Fit}^0_\mathcal{A}(\Pi_3)$.
\end{itemize}
\end{proposition}

\begin{proof} %Since ${\rm Fit}^0_\mathcal{A}(\Pi)$ and ${\rm Fit}_\mathcal{A}(\Pi)$ are both defined via localization, we can (if necessary, after localizing) assume that $\Pi$ is a free presentation of the form (\ref{pres}).

We write $\xi'(\mathcal{A})$ for the $R$-order in $\zeta(A)$ that is generated over $\zeta(\mathcal{A})$ by the elements ${\rm Nrd}_A(M)$ as $M$ runs over matrices in $\bigcup_{n \ge 1}{\rm GL}_n(\mathcal{A})$.

Then, setting $r := r_\Pi$ (in the notation of (\ref{pres})), the invariant ${\rm Fit}_\mathcal{A}(\Pi)$ is defined in \cite[(3.3)]{JN} to be the
 $\xi'(\mathcal{A})$-submodule of $\zeta(A)$ that is generated by the elements ${\rm Nrd}_A(N)$ as $N$ runs over all $r\times r$ submatrices of the matrix $M_\Pi$. Thus, since ${\rm Fit}^0_\mathcal{A}(\Pi)$ is defined to be the ideal of $\xi(\mathcal{A})$ that is generated over $R$ by the same elements ${\rm Nrd}_A(N)$, the required equality $\zeta(\mathcal{A})\cdot{\rm Fit}^0_\mathcal{A}(\Pi) = \xi(\mathcal{A})\cdot{\rm Fit}_\mathcal{A}(\Pi)$ of (i) follows directly from the obvious equality $\zeta(\mathcal{A})\cdot \xi(\mathcal{A}) = \xi(\mathcal{A})\cdot\xi'(\mathcal{A})$.

In the context of (ii) one has $r'_\Pi = r_\Pi$ (in the notation of (\ref{pres})) and so ${\rm Fit}^0_\mathcal{A}(\Pi)$ is, by its very definition, equal to $\xi(\mathcal{A})\cdot {\rm Nrd}_A(M_\Pi)$. Claim (ii) is therefore true provided that the latter ideal depends only on the isomorphism class of the $\mathcal{A}$-module given by the cokernel of $\theta_\Pi$ and this follows from the argument used by Nickel to prove \cite[Th. 3.2ii)]{N3}.

% and so the first assertion of claim (iii) is true if one can show that ${\rm FI}^{0,{\rm tot}}_\mathcal{A}(h)={\rm FI}^0_\mathcal{A}(h)$. To show this it is enough to show that if $h'$ is any $\mathcal{A}$-module presentation that is finer than the quadratic presentation $h$, then ${\rm FI}^0_{\mathcal{A}}(h') \subseteq {\rm FI}^0_{\mathcal{A}}(h)$.
%We set $t:= r_{h,1} = r_{h,2}$ and note that, since $h'$ is finer than $h$, one has $r_{h',1} = r_{h',2} = t$ and there exists a matrix $U$ in ${\rm GL}_t(\mathcal{A})$ and a matrix $V$ in ${\rm M}_t(\mathcal{A})$ with $M(h')U = VM(h)$. It follows that ${\rm Nrd}_A(M(h')) = {\rm Nrd}_A(U^{-1}){\rm Nrd}_A(V){\rm Nrd}_A(M(h))$ and this implies the required inclusion since ${\rm FI}^0_{\mathcal{A}}(h')$ and ${\rm FI}^0_{\mathcal{A}}(h)$ are respectively generated over $\xi(\mathcal{A})$ by ${\rm Nrd}_A(M(h'))$ and ${\rm Nrd}_A(M(h))$ and the product ${\rm Nrd}_A(U^{-1}){\rm Nrd}_A(V)$ belongs to $\xi(\mathcal{A})$.

The key idea in the proof of (iii) is to construct a suitable quadratic presentation $\Pi_2$ of $X_2$ from given quadratic  presentations of $X_1$ and $X_3$ and then to compute the respective zeroth Fitting invariants via the formula in (ii). The precise argument mimics that of Nickel in \cite[Prop. 3.5iii)]{N3} and so will be left to the reader. \end{proof}

\subsection{Fitting invariants of modules}

\subsubsection{}In this section we assume to be given a finitely generated  $\mathcal{A}$-module $Z$.

\begin{definition}\label{def fit mods}{\em  For each non-negative integer $a$, the `$a$-th Fitting invariant' of $Z$ is the ideal of $\xi(\mathcal{A})$ obtained by setting
\[ {\rm Fit}_{\mathcal{A}}^a(Z) := {{\sum}}_\Pi {\rm Fit}^a_{\mathcal{A}}(\Pi),\]
where in the sum $\Pi$ runs over all locally-free presentations of finitely generated $\mathcal{A}$-modules $Z'$ for which there exists a surjective homomorphism of $\mathcal{A}$-modules $Z' \to Z$.}\end{definition} %In the case that $R=\ZZ$ we usually abbreviate ${\rm Fit}_{R[\Gamma]}^a(Z)$ to ${\rm Fit}_{\Gamma}^a(Z)$.

The basic properties of these ideals are described in the next result. Before stating this result we introduce the following useful definition.

\begin{definition}{\em The `central pre-annihilator' of an $\mathcal{A}$-module $Z$ is the $\xi(\mathcal{A})$-submodule of $\zeta(A)$ obtained by setting
\[ {\rm pAnn}_{\mathcal{A}}(Z) := \{x \in \zeta(A): x\cdot \delta(\mathcal{A}) \subseteq {\rm Ann}_{\mathcal{A}}(Z)\},\]
where $\delta(\mathcal{A})$ is the ideal of $\zeta(\mathcal{A})$ from  Definition \ref{def denom}
% \S\ref{denominator ideal section}
and ${\rm Ann}_{\mathcal{A}}(Z)$ denotes the annihilator of $Z$ in $\mathcal{A}$.}
\end{definition}

\begin{remark}{\em  The module  ${\rm pAnn}_{\cA}(Z)$ is finitely generated over $R$, and hence over $\xi(\mathcal{A})$, since the $\cA$-ideal ${\rm Ann}_{\mathcal{A}}(Z)$ is finitely generated over $R$  and 
Lemma \ref{module lemma}(i) implies that $\delta(\cA)$ contains a non-zero integer. In addition, if $\mathcal{A}$ is commutative, then Lemma \ref{module lemma}(v) implies that ${\rm pAnn}_{\cA}(Z)= {\rm Ann}_{\mathcal{A}}(Z)$.} \end{remark}

\begin{theorem}\label{main fit result} The following claims are valid for every finitely generated $\mathcal{A}$-module $Z$ and every non-negative integer $a$.
\begin{itemize}
\item[(i)] ${\rm Fit}_{\mathcal{A}}^a(Z)$ is contained in ${\rm Fit}_{\mathcal{A}}^{a+1}(Z)$.
\item[(ii)] ${\rm Fit}_{\mathcal{A}}^a(Z) = \xi(\mathcal{A})$ if $a$ is large enough.
\item[(iii)] ${\rm Fit}_{\mathcal{A}}^0(Z)$ is contained in ${\rm pAnn}_{\mathcal{A}}(Z)$.
\item[(iv)] Let $e$ be a primitive central idempotent of $A$. Then the ideal  $e\cdot {\rm Fit}_{\mathcal{A}}^a(Z)$ vanishes if ${\rm rr}_{Ae}(e\cdot Z_F)> a \cdot {\rm rr}_{Ae}(Ae)$.
\item[(v)] For any surjective homomorphism of $R$-orders $\varrho: \mathcal{A} \to \mathcal{B}$ there is an inclusion $\varrho({\rm Fit}_{\mathcal{A}}^a(Z)) \subseteq {\rm Fit}_{\mathcal{B}}^{a}(\mathcal{B}\otimes_{\mathcal{A},\varrho}Z)$.
\item[(vi)] For any surjective homomorphism of $\mathcal{A}$-modules $Z \to Z'$ there is an inclusion ${\rm Fit}_{\mathcal{A}}^a(Z)\subseteq {\rm Fit}_{\mathcal{A}}^{a}(Z')$.
\item[(vii)] If $\mathcal{A}$ is commutative, then ${\rm Fit}_{\mathcal{A}}^a(Z)$ is equal to the $a$-th Fitting ideal of the $\mathcal{A}$-module $Z$, as discussed by Northcott in \cite[\S3]{North}.
\end{itemize}
\end{theorem}

\begin{proof} Claims (i), (ii) and (v) follow directly from the corresponding results in Lemma \ref{key lfp} and (vi) from the nature of the sum in Definition \ref{def fit mods}.

In addition, since $R$-modules of the form ${\rm Fit}_{\mathcal{A}}^a(Z)$ and ${\rm pAnn}_{\mathcal{A}}(Z)$ are both finitely generated and torsion-free, the remaining claims can all be proved after localizing at each prime ideal $\mathfrak{p}$ of $R$. In the sequel we shall therefore assume (without explicit comment) that $R$ is local. We also then fix a free presentation $\Pi$ of a finitely generated $\mathcal{A}$-module $Z$ of the form (\ref{pres}) (with $X=Z$) and set $r := r_\Pi$ and $r' := r'_\Pi$.

Claim (iii) is quickly reduced to proving that if $\Pi$ is quadratic, so that $r' = r$, then for any given element $a$ of $\delta(\mathcal{A})$ the product $a\cdot {\rm Nrd}_A(M_\Pi)$ belongs to $\zeta(\mathcal{A})$ and annihilates $Z$.

We set $M :=M_\Pi$. Then claims (i) and (iv) of Lemma \ref{module lemma} combine to imply $a\cdot {\rm Nrd}_A(M)$ belongs to $\zeta(\mathcal{A})$ and so it suffices to prove this element annihilates $Z$. To do this we follow an argument used by Nickel to prove \cite[Th. 4.2]{N3}.

Specifically, it is enough to show that for every element $y$ of $\mathcal{A}^r$ the product
\[ a\cdot{\rm Nrd}_{A}(M)\cdot y = a\cdot MM^*(y) = M(a\cdot M^*(y))\]
belongs to ${\rm im}(\theta_\Pi)$, where the matrix $M^\ast$ is as defined in (\ref{adjoint matrix}). This is in turn a direct consequence of the fact that the definition of $\delta(\mathcal{A})$ ensures $a\cdot M^*$ belongs to ${\rm M}_{r,r}(\mathcal{A})$ and hence that $M(a\cdot M^*(y))$ belongs to ${\rm im}(\theta_\Pi)$. %It is therefore clear that $a{\rm Nrd}_{E[G]}(A)$ annihilates ${\rm cok}(A)$, as required.

To prove (iv) we set $m := {\rm rr}_{Ae}(Ae)$. Then it is enough to prove that for any non-negative integer $a$ one has $e\cdot {\rm Fit}_{\mathcal{A}}^a(\Pi)=0$ whenever ${\rm rr}_{Ae}(e\cdot Z_F) > a\cdot m$.

To show this we let $M$ denote any matrix obtained from $M_\Pi$ by replacing at most $a$ of its columns by arbitrary elements of $\mathcal{A}$. We set $d := r_\Pi$, fix a $d\times d$ submatrix $N$ of $M$ (so that $N$ is a typical matrix of the set $\mathfrak{S}^a(M_\Pi)$) and write $N_\Pi$ for the corresponding submatrix of $M_\Pi$.

We fix a splitting field $E$ for $Ae$ and an isomorphism of algebras of the form $Ae\otimes_{\zeta(A)e}E $ $\cong {\rm M}_{m,m}(E)$. This isomorphism induces a map $\iota: {\rm M}_{d,d}(A) \to {\rm M}_{dm,dm}(E)$ and $e\cdot {\rm Fit}_{\mathcal{A}}^a(\Pi)$ is, by its definition, generated over $\xi(\mathcal{A})$ by the determinants of all matrices of the form $\iota(N)$. In addition, it is clear that
\[ {\rm rank}(\iota(N))\le {\rm rank}(\iota(N_\Pi)) + a\cdot m \le \bigl(d\cdot m - {\rm rr}_{Ae}(e\cdot Z_F)\bigr) + a\cdot m. \]
Hence, if ${\rm rr}_{Ae}(e\cdot Z_F) > a\cdot m$, then ${\rm rank}(\iota(N))< d\cdot m$ and so ${\rm det}(\iota(N)) = 0$, as required to prove (iv).

Finally, to prove (vii), we assume $\mathcal{A}$ is commutative and note that ${\rm Fit}^a_\mathcal{A}(\Pi)$ is generated over $\mathcal{A}$ by elements of the form ${\rm det}(N)$ where $N$ is an $r\times r$ matrix, at least $r-a$ columns of which coincide with the columns of an $r\times r$ submatrix of $M_\Pi$.

The Laplace expansion of ${\rm det}(N)$ therefore shows that it is contained in the ideal of $\mathcal{A}$ generated by the set of $(r-a)\times (r-a)$ minors of $M_\Pi$. Thus, since the latter ideal is, by definition, equal to ${\rm Fit}_\mathcal{A}^a(Z)$ one has ${\rm Fit}^a_\mathcal{A}(\Pi)\subseteq {\rm Fit}_\mathcal{A}^a(Z)$.

To prove the reverse inclusion it suffices to show that for each $(r-a)\times (r-a)$ submatrix $N$ of $M_\Pi$ the term ${\rm det}(N)$ belongs to ${\rm Fit}^a_\mathcal{A}(\Pi)$.

For any natural number $n$ and any non-negative integer $t$ we write $[n]_t$ for the set of subsets of $\{1, 2, \dots , n\}$ that are of cardinality ${\rm min}\{t,n\}$.

Then we assume that $N$ is obtained by first deleting from $M_\Pi$ the columns corresponding to a subset $J = \{i_1,i_2,\ldots , i_a\}$ of $[r]_a$ with $i_1< i_2< \cdots < i_a$, and then taking the rows corresponding to an element $J_1$ of $[r']_{r-a}$. We choose an element $J_1'$ of $[r']_{r}$ that contains $J_1$, label the elements of $J_1'\setminus J_1$ as $k_1< k_2 < \cdots < k_a$ and then define an element $(\varphi_i)_{1\le i\le a}$ of $\Hom_\mathcal{A}(\mathcal{A}^{r'},\mathcal{A})^a$ by setting $\varphi_i(b_j) = \delta_{jk_i}$ for each $j$ with $1\le j\le r'$.

Then an explicit computation shows that, with these choices, the determinant of the matrix $M_\Pi(J,\varphi)$ defined in (\ref{fitting matrix}) is equal to $\pm {\rm det}(N)$ and hence implies that ${\rm det}(N)$ belongs to ${\rm Fit}^a_\mathcal{A}(\Pi)$, as required. \end{proof}

\subsubsection{} In the sequel we write $M^{\rm tr}$ for the transpose of a matrix $M$. 

Then an alternative theory of Fitting invariants is obtained if one replaces the matrices $M(J,\varphi)$ in (\ref{fitting matrix}) by $M^{\rm tr}(J,\varphi)^{\rm tr}$ (or equivalently, one substitutes rows, rather than columns, of the matrix $M$ by elements of $\im(\varphi_a)$). 

The same arguments as above show that the corresponding families of ideals, which we denote by ${\rm Fit}_{\mathcal{A}}^{{\rm tr},a}(\Pi)$ and ${\rm Fit}_{\mathcal{A}}^{{\rm tr},a}(Z)$, validate the natural analogues of Lemma \ref{key lfp}, Proposition \ref{fitt0 invariants lemma} and Theorem \ref{main fit result}. In addition, for certain orders $\mathcal{A}$ a precise connection between the two theories is established in Lemma \ref{transpose lemma} below. %\subsubsection{}We recall from \cite{bs1} that
%if $F\subset \CC$, then for each non-negative integer $a$ one has
%
%\begin{equation}\label{duality eq} {\rm Fit}^a_{R[\Gamma]}(\Pi^{\rm tr}) = {\rm Fit}^a_{R[\Gamma]}(\Pi)^\#\end{equation}
%
%with $x \mapsto x^\#$ the involution of $\zeta(F[G]) = {\prod}_{\widehat G}F^c$ that sends each element $(x_\psi)_\psi$ to $(x_{\check\psi})_\psi$.

%The `pre-annihilator' of $Z$ is the $\xi(R[\Gamma])$-submodule of $\zeta(F[\Gamma])$ obtained by setting
%
%\[ {\rm Ann}^{\rm p}_{\ZZ[\Gamma]}(Z) := \{x \in \zeta(F[\Gamma]): x\cdot \mathcal{H}(R[\Gamma]) \subset {\rm Ann}_{R[\Gamma]}(Z)\}.\]
%
%
%In particular, if $\Gamma$ is abelian, then ${\rm Ann}^{\rm p}_{\zeta(\ZZ[\Gamma])}(Z) = {\rm Ann}_{\zeta(\ZZ[\Gamma])}(Z)$. It is shown in  \cite{bs1} that in all cases one has
%
%\begin{equation}\label{ann} {\rm Fit}^0_G(Z) \subseteq {\rm Ann}^{\rm p}_{\zeta(\ZZ[G])}(Z).\end{equation}

%\subsection{}
%\color{blue}
%{\em By proposition \ref{propformula}, we have
%$$\eta^t_{M,\{n_i\}_i} = (-1)^{t(d-t)}(\varphi_{t+1}\wedge\cdots\wedge\varphi_{d})(n_1\wedge\cdots\wedge n_d),$$
%where $\varphi_j \in \Hom_R(N,R)$ is defined by $\varphi_j(n_i)=m_{ij}$, where $m_{ij}$ is the $(i,j)$-entry of the matrix $M$.}
%\color{black}

\subsubsection{}\label{transpose} Let $\Gamma$ be a finite group. In this section we discuss a construction of presentations for modules over the Gorenstein order $\mathcal{A} = R[\Gamma]$. The observations made here will be useful in later arithmetic applications.

\begin{definition}\label{transpose def}{\em If $\Pi$ is a locally-free presentation of an $R[\Gamma]$-module $X$ (as described in \S\ref{lfp def}), then the `transpose' $\Pi^{\rm tr}$ of $\Pi$ is the locally-free presentation of the $R[\Gamma]$-module ${\rm cok}(\Hom_R(\theta_\Pi,R))$ that is given by the following data:

\begin{itemize}
\item[$\bullet$] the exact sequence of $R[\Gamma]$-modules
\[ (\Pi^{\rm tr})^{\rm seq}: \Hom_R(P,R) \xrightarrow{\Hom_R(\theta_\Pi,R)} \Hom_R(P',R) \xrightarrow{} {\rm cok}(\Hom_R(\theta_\Pi,R)) \to 0,\]
where the linear duals are endowed with the contragredient action of $\Gamma$;
\item[$\bullet$] for each prime $\mathfrak{p}$ the composite isomorphisms
\[\Hom_R(P,R)_{(\mathfrak{p})} \cong \Hom_R(R[\Gamma]^{{\rm rk}_{R[\Gamma]}(P)} ,R)_{(\mathfrak{p})}\cong R_{(\mathfrak{p})}[\Gamma]^{{\rm rk}_{R[\Gamma]}(P)}\]
and
\[\Hom_R(P',R)_{(\mathfrak{p})} \cong \Hom_R(R[\Gamma]^{{\rm rk}_{R[\Gamma]}(P')},R)_{(\mathfrak{p})}\cong R_{(\mathfrak{p})}[\Gamma]^{{\rm rk}_{R[\Gamma]}(P')}\]
where the first maps are respectively induced by the $R_{(\mathfrak{p})}$-linear duals of $\iota_{\Pi,\mathfrak{p}}$ and $\iota'_{\Pi,\mathfrak{p}}$ and the second by the standard isomorphism
 $\Hom_R(R[\Gamma],R) \cong R[\Gamma]$.
\end{itemize}}
\end{definition}

\begin{remark}\label{trtr rem}{\em It is clear $\Pi^{\rm tr}$ is locally-quadratic if and only if $\Pi$ is locally-quadratic. In addition, since $\Hom_R(\Hom_R(\theta_\Pi,R),R)$ identifies naturally with $\theta_\Pi$, there exists a choice of isomorphism of $R[\Gamma]$-modules ${\rm cok}(\theta_\Pi) \cong X$ that induces an identification of $(\Pi^{\rm tr})^{\rm tr}$ with $\Pi$.}\end{remark}

%\begin{definition}{\em For a finitely generated $R[\Gamma]$-module $Z$ and non-negative integer $a$, the `transpose' of the $a$-th Fitting invariant of $Z$ is the ideal
%
%\[   {\rm Fit}_{R[\Gamma]}^{a}(Z)^{\rm tr} :=   {{\sum}}_\Pi {\rm Fit}^a_{R[\Gamma]}(\Pi^{\rm tr}),\]
%
%of $\xi(\mathcal{A})$, where the sum is as in Definition
%\ref{def fit mods}.}
%\end{definition}

Before stating the next result, we recall that the Wedderburn decomposition of $\CC_p[\Gamma]$ induces an identification 
\[ \zeta(\CC_p[\Gamma]) \cong {{\prod}}_{\psi\in {\rm Ir}_p(\Gamma)}\CC_p; \,\,\, x\mapsto (x_\psi)_\psi,\]
where $ {\rm Ir}_p(\Gamma)$ is the set of irreducible $\CC_p$-valued characters of $\Gamma$. We then write $x\mapsto x^\#$ for the $\CC_p$-linear involution of $\zeta(\CC_p[\Gamma])$ with the property that for all $x \in \zeta(\CC_p[\Gamma])$ and $\psi\in  {\rm Ir}_p(\Gamma)$ one has $(x^\#)_{\psi} = x_{\check\psi}$, where $\check\psi$ is the contragredient of $\psi$.

\begin{lemma}\label{transpose lemma} Assume $R$ is contained in  $\CC_p$. Then, if $\Pi$ is a locally-quadratic presentation of $R[\Gamma]$-modules, for every non-negative integer $a$ one has  
\[ {\rm Fit}^{{\rm tr},a}_{R[\Gamma]}(\Pi^{\rm tr}) = {\rm Fit}^a_{R[\Gamma]}(\Pi)^\#.\]
%
%Then if either $a = 0$ or $\Gamma$ is abelian one has
%
%\[ {\rm Fit}^{a}_{R[\Gamma]}(Z)^{\rm tr} = {\rm Fit}^a_{R[\Gamma]}(Z)^\#\]
%
%for every  finitely generated $R[\Gamma]$-module $Z$.
\end{lemma}

\begin{proof} We write $\iota_\#$ for the $\CC_p$-linear anti-involution of $\CC_p[\Gamma]$ that inverts elements of $\Gamma$, and note that $\iota_\#(x) = x^\#$ for all $x \in \zeta(\CC_p[\Gamma])$. 

For $M$ in ${\rm M}_{d}(R[\Gamma])$ we write $\iota_\#(M)$ for the matrix in ${\rm M}_{d}(R[\Gamma])$  obtained by applying $\iota_\#$ to all components of $M$. It is then easily checked that for $M$ in ${\rm M}_{d}(R[\Gamma])$ one has
\begin{equation}\label{transpose rn} {\rm Nrd}_{F[\Gamma]}( \iota_\#(M^{\rm tr})) = {\rm Nrd}_{F[\Gamma]}( \iota_\#(M)) = {\rm Nrd}_{F[\Gamma]}(M)^\#. \end{equation}

Now, after localizing at each prime ideal of $R$, it is enough to prove the claimed equality in the case that $\Pi$ is a quadratic presentation. To consider this case we fix a homomorphism of $R[\Gamma]$-modules  $\theta: R[\Gamma]^{d} \to R[\Gamma]^{d}$  and write $M_\theta$ for its matrix with respect to the standard basis of $R[\Gamma]^{d}$. Then the matrix of $\Hom_{R}(\theta,R)$ with respect to the standard (dual) bases of $\Hom_R(R[\Gamma]^{d},R)$ is equal to $\iota_\#(M_\theta^{\rm tr})$. It follows that ${\rm Fit}^{{\rm tr},a}_{R[\Gamma]}(\Pi^{\rm tr})$ is equal to the $\xi(R[\Gamma])$-ideal that is generated by the elements %
\[ {\rm Nrd}_{F[\Gamma]}\bigl( (\iota_\#(M_\theta^{\rm tr}))^{\rm tr}(J,\varphi)^{\rm tr}\bigr) = {\rm Nrd}_{F[\Gamma]}\bigl( \iota_\#(M_\theta(J,\varphi^\#))^{\rm tr}\bigr) = \bigl({\rm Nrd}_{F[\Gamma]}(M_\theta(J,\varphi^\#)\bigr)^\# \]
as $J$ runs over tuples $\{i_1,i_2,\ldots , i_t\}$ with $t \le a$ and $1\le i_1< i_2< \cdots < i_t\le d$, and $\varphi = (\varphi_i)_{1\le i\le t}$ over 
$\Hom_{R[\Gamma]}(R[\Gamma]^{d},R[\Gamma])^t$. Here we write $\varphi^\#$ for the tuple $(\iota_\#\circ\varphi_i)_{1\le i\le t}$, so the first equality is clear and the second follows from (\ref{transpose rn}). 

In particular, since the second equality in (\ref{transpose rn}) implies that $\xi(R[\Gamma]) = \xi(R[\Gamma])^\#$, to deduce the claimed result from the last displayed equality, it is enough to recall the  explicit definition of ${\rm Fit}^a_{R[\Gamma]}(\Pi)$ and note that $\{\varphi^\#: \varphi \in \Hom_{R[\Gamma]}(R[\Gamma]^{d},R[\Gamma])^t\} = \Hom_{R[\Gamma]}(R[\Gamma]^{d},R[\Gamma])^t.$
%In the case $a=0$, the claimed result therefore follows directly from the equality (\ref{transpose rn}). In the case that $\Gamma$ is abelian (and $a >0$) it follows in a similar way after taking account of Theorem \ref{main fit result}(vii). 
\end{proof}

\begin{remark}\label{hash remark}{\em The above argument has shown that $\xi(R[\Gamma]) = \xi(R[\Gamma])^\#$. In a similar way, since the defining equality (\ref{adjoint matrix}) implies $(\iota_\#(M))^\ast = \iota_\#(M^\ast)$ for every $M$ in ${\rm M}_d(R[\Gamma])$, one has $\delta(R[\Gamma]) = \delta(R[\Gamma])^\#$. In particular, if $Z$ is an $R[\Gamma]$-module, then the last equality combines with the  exactness of Pontryagin duality to imply that, with respect to the contragredient action of $R[\Gamma]$ on $Z^\vee = \Hom_R(Z,F/R)$, one has ${\rm pAnn}_{R[\Gamma]}(Z^\vee) = {\rm pAnn}_{R[\Gamma]}(Z)^\#$.  }\end{remark}

\section{Reduced exterior powers}\label{red ext powers sec}

In this section we discuss the basic properties of an explicit construction of `exterior powers' over semisimple rings.

\subsection{Exterior powers over commutative rings}\label{extc} We first quickly review relevant aspects of the classical theory of exterior powers over commutative rings.

Let $R$ be a commutative ring, and $M$ be an $R$-module. Then an element $f \in \Hom_R(M,R)$ induces, for every natural number $r$, a homomorphism of $R$-modules 
$${\bigwedge}_R^r M \to {\bigwedge}_R^{r-1}M, \quad m_1 \wedge \cdots \wedge m_r \mapsto {\sum}_{i=1}^r (-1)^{i+1}f(m_i) m_1 \wedge \cdots \wedge m_{i-1} \wedge m_{i+1}\wedge \cdots \wedge m_r,$$
which, for convenience, will also often be denoted by $f$. By using this construction, we define for natural numbers $r$ and $s$ with $r \geq s$, a pairing:
$${\bigwedge}_R^r M \times {\bigwedge}_R^s \Hom_R(M,R) \to {\bigwedge}_R^{r-s} M ; \quad (m,\wedge_{i=1}^{i=s}f_i) \mapsto (\wedge_{i=1}^{i=s}f_i)(m) := (f_s \circ \cdots \circ f_1) (m).$$
 (For clarity, we stress that in the composition `$f_s \circ \cdots \circ f_1$' each $f_i$ is regarded in the manner described above as a map ${\bigwedge}_R^{r-i+1}M \to {\bigwedge}_R^{r-i}M$). %Finally, we set
%
%\[ (f_s \circ \cdots \circ f_1)(m).\]
%If $r=s$, one checks that $(f_1 \wedge \cdots \wedge f_r)(m_1 \wedge \cdots \wedge m_r)=\det(f_i(m_j))$.
\

We shall also use the following convenient notation: for any natural numbers $r$ and $s$ with $s \le r$ we write $ {r \atopwithdelims[] s} $ for the subset of $S_r$ comprising permutations $\sigma$ which satisfy both
\[ \sigma(1) < \cdots < \sigma(s)\,\,\,\,\text{ and }\,\,\,\, \sigma(s+1) <\cdots<\sigma(r).\]
(This notation is motivated by the fact that the cardinality of ${r \atopwithdelims[] s}$ is the binomial coefficient $r \choose s$.)

We can now record two results that play an important role in the sequel.

\begin{lemma} \label{propformula} If $s \le r$, then for all subsets $\{f_i\}_{1\le i\le s}$ of $\Hom_R(M,R)$ and $\{m_j\}_{1\le j\le r}$ of $M$ one has
$$(\wedge_{i=1}^{i=s}f_i)(\wedge_{j=1}^{j=r}m_j) ={\sum}_{\sigma \in {r \atopwithdelims[] s} }{\rm{sgn}}(\sigma)\det(f_i(m_{\sigma(j)}))_{1\leq i,j\leq s} m_{\sigma(s+1)}\wedge\cdots\wedge m_{\sigma(r)}.$$
In particular, if $r=s$, then we have
$$(\wedge_{i=1}^{i=r}f_i)(\wedge_{j=1}^{j=r}m_j)=\det(f_i(m_j))_{1\leq i,j \leq r}.$$
\end{lemma}

\begin{proof} This is verified by means of an easy and explicit computation. \end{proof}

\begin{lemma} [{\cite[Lem. 4.2]{bks}}]\label{prope}
Let $E$ be a field and $W$ an $n$-dimensional $E$-vector space. Fix a non-negative integer $m$ with $m \le n$ and a subset $\{\varphi_i\}_{1\le i\le m}$ of $\Hom_E(W,E)$. Then the $E$-linear map
$\Phi={\bigoplus}_{i=1}^m \varphi_i:W\longrightarrow E^{\oplus m}$ is such that
$$\im({\wedge}_{1\leq i\leq m}\varphi_i : {\bigwedge}_E^n W \longrightarrow {\bigwedge}_E^{n-m}W)=
\begin{cases} {\bigwedge}_E^{n-m}\ker (\Phi), &\text{ if $\Phi$ is surjective},\\
   0, &\text{otherwise.}\end{cases} $$
\end{lemma}

%\begin{proof} Suppose first that $\Phi$ is surjective. Then there exists a subspace $W'$ of $W$ such that $W=\ker(\Phi) \oplus W'$ and $\Phi$ maps $W'$ isomorphically onto $E^{\oplus m}$.

%We see that ${\bigwedge}_{1\leq i \leq m}\varphi_i$ induces an isomorphism
%$${\bigwedge}_E^m W' \stackrel{\sim}{\longrightarrow} E.$$
%Hence we have an isomorphism
%$${\bigwedge}_{1\leq i\leq m}\varphi_i:{\bigwedge}_E^nW=\left({\bigwedge}_E^{n-m}\ker (\Phi)\right) \otimes_E \left({\bigwedge}_E^m W'\right) \stackrel{\sim}{\longrightarrow} {\bigwedge}_E^{n-m}\ker(\Phi).$$
%In particular, we have
%$$\im({\bigwedge}_{1\leq i\leq m}\varphi_i : {\bigwedge}_E^n W \longrightarrow {\bigwedge}_E^{n-m}W)={\bigwedge}_E^{n-m}\ker (\Phi).$$
%Next, suppose $\Phi$ is not surjective. Then $\varphi_1,\ldots,\varphi_m \in \Hom_E(W,E)$ are linearly dependent and so ${\bigwedge}_{1\leq i \leq m}\varphi_i=0$, as claimed.
%\end{proof}

\subsection{Reduced exterior powers over semisimple rings}\label{Exterior powers over semisimple rings}In this subsection, we define a notion of exterior powers for finitely generated modules over semsimple rings.

The underlying idea is as follows. If $\Lambda$ is a non-commutative ring for which there exists a functor $\Phi$ from the category of $\Lambda$-modules to the category of modules over some commutative ring $\Omega$ that gives an equivalence of these categories, then the exterior power of a $\Lambda$-module $M$ should be defined via a suitable exterior power of the $\Omega$-module $\Phi(M)$. (In our case, we shall take $\Phi$ to be the Morita functor defined at the end of \S\ref{morita}.)

%If $A$ is a split simple artinian ring, then the Morita functor induces an equivalence between the categories of finitely generated left $A$-modules and of finite-dimensional vector spaces over $\zeta(A)$. This is the key observation of our construction of non-commutative exterior powers.

\subsubsection{}\label{simple art section} Let $K$ be a field of characteristic zero and $A$ a finite-dimensional simple $K$-algebra.
%We first consider the case of simple Artinian rings.

\begin{definition}\label{rth reduced ext def}{\em Take a splitting field $E$ of $A$, set $A_E:=A\otimes_{\zeta(A)}E$ and fix a simple  $A_E$-module $V$. Then for each $A$-module $M$ and each non-negative integer $r$, we define the `$r$-th reduced exterior power' of $M$ over $A$ to be the $E$-vector space
$${\bigwedge}_A^r M:={\bigwedge}_E^{rd}(V^\ast\otimes_{A_E}M_E),$$
where $d:=\dim_E(V)$, $M_E:=M\otimes_{\zeta(A)}E$, and $V^\ast:=\Hom_E(V,E)$.}\end{definition}

\begin{remark}\label{splitting field remark} {\em If $A$ is commutative, and hence a field, then our convention will always be to take $E$, and hence also $V$, to be $A$ so that the above definition coincides with the standard $r$-th exterior power of $M$ as an $A$-module. In the general case, whilst our chosen notation ${\bigwedge}_A^r M$ suppresses the dependence of the definition on the splitting field $E$ and simple $A_E$-module $V$ we feel that this should not lead to confusion. Firstly, for a fixed $E$, all simple $A_E$-modules $V$ are isomorphic and so lead to isomorphic reduced exterior powers. In addition, each $K$-embedding of splitting fields $\sigma: E \to E'$ induces a canonical isomorphism $E'\otimes_{E,\sigma}{\bigwedge}_A^r M \cong {\bigwedge}_A^r M$ where the exterior powers are respectively defined via the pairs $(E,V)$ and $(E',E'\otimes_{E,\sigma}V)$. Aside from this, one can also define ${\bigwedge}_A^r M$ with respect to a `canonical' choice of splitting field $E$ as in Remark \ref{exp splitting} (and see also Remark \ref{group ring rem} below in this regard).}\end{remark}

\begin{remark}\label{func rem}{\em Reduced exterior powers are functorial in the following sense. Any homomorphism of $A$-modules $\theta: M \to M'$ induces for each natural number $r$ a homomorphism of $E$-modules
\[
 {\bigwedge}_A^r M = {\bigwedge}_E^{rd}(V^\ast\otimes_{A_E}M_E) \xrightarrow{{\wedge}_E^{rd}({\rm id}_{V^\ast}\otimes \theta)} {\bigwedge}_E^{rd}(V^\ast\otimes_{A_E}M_E') = {\bigwedge}_A^r M'.\]
}\end{remark}

\subsubsection{}To make analogous constructions for linear duals we use the fact that $\Hom_A(M,A)$ has a natural structure as (left) $A^{\rm op}$-module. In particular, since $V^\ast$ is a simple $A^{\rm op}_E$-module and   $V^{\ast \ast}$ is canonically isomorphic to $V$, the $r$-th reduced exterior power (in the sense of Definition \ref{rth reduced ext def}) of $\Hom_A(M,A)$ is equal to
$${\bigwedge}_{A^{\rm op}}^r \Hom_A(M,A) = {\bigwedge}_E^{rd} (V\otimes_{A_E^{\rm op}}\Hom_{A_E}(M_E,A_E) ).$$

The natural isomorphism
$$V \otimes_{A_E^{\rm op}}\Hom_{A_E}(M_E,A_E) \stackrel{\sim}{\to} \Hom_E(V^\ast\otimes_{A_E}M_E, E) ; \quad v \otimes f \mapsto (v^\ast\otimes m \mapsto v^\ast(f(m)v))$$
therefore induces a composite identification
\[{\bigwedge}_{A^{\rm op}}^r \Hom_A(M,A) = {\bigwedge}_E^{rd} \Hom_E(V^\ast\otimes_{A_E}M_E, E) = {\bigwedge}_E^{rd} \Hom_{\zeta(A)}(V^\ast\otimes_{A_E}M_E, \zeta(A)),\]
with the last identification induced by the trace map $E\to \zeta(A)$. By using this identification, the construction discussed in \S \ref{extc} applies to the $E$-vector space $V^\ast\otimes_{A_E}M_E$ to give a  pairing
\begin{equation}\label{rubin pairing} {\bigwedge}_A^r M \times {\bigwedge}_{A^{\rm op}}^s \Hom_A(M,A) \to {\bigwedge}_A^{r-s} M.\end{equation}
We shall denote the image under this pairing of a pair $(m,\varphi)$  by $\varphi(m)$.

\subsubsection{}\label{rep norm section}We now fix an ordered $E$-basis $\{ v_i\}_{1\le i\le d}$ of $V$ and write $\{ v_i^\ast\}_{1\le i\le d}$ for the corresponding dual basis of $V^\ast$.

For any subsets $\{m_i\}_{1\le i\le r}$ of $M$ and $\{\varphi_i\}_{1\le i\le r}$ of $\Hom_A(M,A)$ we then set
\begin{equation}\label{non-commutative wedge} \wedge_{i=1}^{i=r}m_i :={\wedge}_{1\leq i \leq r}({\wedge}_{1\leq j\leq d}v^\ast_j\otimes m_i) \in {\bigwedge}_E^{rd}(V^\ast\otimes_{A_E}M_E)={\bigwedge}_A^rM \end{equation}
and
\begin{equation} \label{non-commutative wedge 2}\wedge_{i=1}^{i=r}\varphi_i := {\wedge}_{1\leq i \leq r}({\wedge}_{1\leq j \leq d}v_j \otimes \varphi_i) \in {\bigwedge}_{A^{\rm op}}^r \Hom_A(M,A),\end{equation}
where $m_i$ and $\varphi_i$ are regarded as elements of $M_E$ and $\Hom_{A_E}(M_E,A_E)$ in the obvious way.

\begin{lemma}\label{primitive span} Let $M$ be a finitely generated $A$-module. Then the $E$-spaces ${\bigwedge}_A^r M$ and ${\bigwedge}_{A^{\rm op}}^r \Hom_A(M,A)$ are respectively spanned by the sets $\{ \wedge_{i=1}^{i=r}m_i: m_i \in M\}$ and $\{ \wedge_{i=1}^{i=r}\varphi_i: \varphi_i\in \Hom_A(M,A)\}$.\end{lemma}

\begin{proof} %Since the pairing (\ref{rubin pairing}) with $s=r$ is non-degenerate the space $({\bigwedge}_A^rM)^0$ vanishes if and only if the spaces ${\bigwedge}_{A^{\rm op}}^r \Hom_A(M,A)$ and $({\bigwedge}_{A^{\rm op}}^r \Hom_A(M,A))^{\rm prim}$ coincide.
%
%If $A$ is commutative, then it is clear that ${\bigwedge}_{A^{\rm op}}^r \Hom_A(M,A) = ({\bigwedge}_{A^{\rm op}}^r \Hom_A(M,A))^{\rm prim}$ and so claim (i) is verified.

We only prove the claim for $M$ since exactly the same argument works for $\Hom_A(M,A)$ (after replacing $A$ by $A^{\rm op}$). Then, since there exists a surjective homomorphism of $A$-modules of the form $A^t \to M$ (for any large enough $t$) it is enough to prove the claim in the case that $M$ is free of rank $t$.

In this case, if we fix an $A$-basis $\{b_i\}_{1\le i\le t}$ of $M$, then the $E$-space $V^\ast \otimes_{A_E}M_E$ has as a basis the set $\{ x_{ij} := v_j^\ast\otimes b_i\}_{1\le j\le d, 1\le i\le t}$ and so ${\bigwedge}_A^r M$ is generated over $E$ by the exterior powers of all $rd$-tuples of distinct elements in this set. It is thus enough to show that for any $d$ distinct elements $X := \{x_{i_kj_k}\}_{1 \le k\le d}$ of the above set, there exists an element $m_{X}$ of $M_E$ with $\,{\wedge}_{y=1}^{y=d}(v_y^\ast\otimes m_X) = \pm {\wedge}_{k=1}^{k=d}x_{i_kj_k}$.

To do this we set  $X_c := \{x_{i_kj_k}: i_k =c \}$ and $n_c := |X_c|$ for each index $c$ with $1\le c\le t$. We also set $Y:= \{v_i^\ast\}_{1\le i\le d}$. We order the indices $c$ for which $n_c \not= 0$ as $c_1< c_2< \cdots < c_s$ (for a suitable integer $s$) and for each integer $\ell$ with $0\le \ell\le s$ we set $N_\ell := {\sum}_{j=1}^{j=\ell}n_{c_j}$ (so that $N_0 = 0$ and $N_s = d$).

For each index $k$ we then choose $a_{c_k}$ to be an element of $A_E$ whose image under the canonical isomorphism $A_E \cong {\rm End}_E(V^\ast) (\cong {\rm M}_d(E))$ maps the elements $\{v_i^\ast: N_{k-1}< i\le N_k\}$ to the ($N_k-N_{k-1} = n_{c_k}$ distinct) elements $v_j^\ast$ that occur as the first components of the elements in $X_{c_k}$ and maps the remaining $d- n_{c_k}$ elements of $Y$ to zero. It is then straightforward to check the element $m_X := {\sum}_{k=1}^{k=s}a_{c_k}\cdot b_{c_k}$ of $M_E$ is such that
\begin{align*} {\wedge}_{y=1}^{y=d}(v_y^\ast\otimes m_X) =\, &{\wedge}_{y=1}^{y=d}({\sum}_{k=1}^{k=s}v_y^\ast a_{c_k} \otimes b_{c_k})\\
%=\, &{\wedge}_{y=1}^{y=d}({\sum}_{c=1}^{c=t}v_y^\ast a_c\otimes b_c)\\
= \, &\pm {\wedge}_{k=1}^{k=s}\bigl({\wedge}_{N_{k-1}< y \le N_k}(v_y^\ast a_{c_k} \otimes b_{c_k})\bigr)\\
= \, &\pm {\wedge}_{k=1}^{k=d}x_{i_kj_k},\end{align*}
as required.\end{proof}
%{\sum}_{j=1}^{j=k-1}n_{c_j}< i\le {\sum}_{j=1}^{j=k}n_{c_j}\]
%\pm {\wedge}_{k=1}^{k=d}x_{i_kj_k}
%satisfies the required equality (\ref{req eq}).

%first $n_{c_1}$ elements of $Y$ to the ($n_{c_1}$ distinct) vectors $v_j^\ast$ that occur as the first components of the elements in $X'_{c_1}$ and maps all remaining elements of $Y$ to zero.
%If $s > 1$, we next specify $a_{c_2}$ to be an element of $A_E$ whose image under (\ref{rn iso}) maps the vectors $\{v_i^\ast: n_{c_1}< i\le n_{c_1} + n_{c_2}}$
%If $s > 1$, we next specify $a_{c_2}$ to be an element of $A_E$ whose image under (\ref{rn iso})

\begin{remark}\label{comm conventions}{\em If $A$ is commutative, then our convention is that $E = V = A$ (see Remark \ref{splitting field remark}). In particular, in this case $d=1$ and we will always take the basis $\{v_1\}$ fixed above to comprise the identity element of $A$ so that the elements defined in (\ref{non-commutative wedge}) and (\ref{non-commutative wedge 2}) coincide with the classical definition of exterior products. }\end{remark}

\subsubsection{}\label{semisimple exterior power section} We assume now that $A$ is a semisimple ring, with Wedderburn decomposition  (\ref{wedderburn}).

In this case, each finitely generated $A$-module $M$ decomposes as a direct sum $M = {\bigoplus}_{i \in I}M_i$ where each summand $M_i := A_i\otimes_AM$ is a finitely generated $A_i$-module. For any non-negative integer $r$, we then define the $r$-th reduced exterior power of the $A$-module $M$ by setting
\[
%\begin{equation}\label{semisimple ext}
 {\bigwedge}_A^rM:={\bigoplus}_{i\in I} {\bigwedge}_{A_i}^r (A_i \otimes_A M),
% \end{equation}
\]
where each component exterior power in the direct sum is defined with respect to a given choice of splitting field $E_i$ for $A_i$ over $\zeta(A_i)$ and  a given choice of simple $E_i\otimes_{\zeta(A_i)}A_i$-module $V_i$.
 The associated reduced exterior power ${\bigwedge}_{A^{\rm op}}^r \Hom_A(M,A)$ is defined in a similar way.

We refer to the direct product 
\begin{equation}\label{splitting algebra} \tilde E := {\prod}_{i\in I}E_i\end{equation}
as a `splitting algebra' for $A$ and note that ${\bigwedge}_A^rM$ and ${\bigwedge}_{A^{\rm op}}^r \Hom_A(M,A)$ are both modules over $\tilde E$.
 Whilst these constructions clearly depend on the choice of splitting algebra, the following result shows that they behave functorially under scalar extension.

\begin{lemma} For any extension $K'$ of $K$ % in $K^c$
there exists an injective homomorphism from $ {\bigwedge}_A^rM$ to ${\bigwedge}_{A\otimes_KK'}^r(M\otimes_KK')$. %(Note that $A\otimes_K K'$ is not necessarily simple, but Definition \ref{rth reduced ext def} is naturally extended to the case when $A$ is semisimple, by decomposing $A$ into simple components. See (\ref{semisimple ext}) below.)
\end{lemma}

\begin{proof} We can assume, without loss of generality, that $A$ is simple. We then set $F:=\zeta(A)$ and $A':= A\otimes_K K'$ and $M':=M\otimes_KK'$. We assume ${\bigwedge}_A^rM$ is defined by using an algebraic extension $E$ of $F$ in $K^c$  and a simple $A_E$-module $V$, and we set $d=\dim_E(V)$. We also use the notation of Lemma \ref{propsch0}. In particular, we write $\Omega$ for an algebraic closure of $K'$.

%Then, by definition, one has
%$${\bigwedge}_A^rM={\bigwedge}_E^{rd}(V^\ast\otimes_{A_E}M_E),$$
%where

For each $\sigma$ in $\Sigma(F/K,K')$ we fix a $K$-embedding $\widetilde \sigma$ of $E$ into $\Omega$ that extends $\sigma$. % (such a $\widetilde \sigma$ exists since $E/F$ is algebraic).
 Then the field $E_\sigma: = \widetilde \sigma(E)K'$ splits the simple ring $A \otimes_{F}\sigma(F)K'$. In addition,   $V_\sigma:=V \otimes_E E_\sigma$ is a simple module over
  $A_{E_\sigma}:=A \otimes_F E_\sigma$ and so, by Lemma \ref{propsch0} and the definition of reduced exterior powers, one has
$${\bigwedge}_{A'}^rM'={\bigoplus}_{\sigma \in \Sigma(F/K,K')}{\bigwedge}_{E_\sigma}^{rd} ( V_\sigma^\ast \otimes_{A_{E_\sigma}} M_{E_\sigma}),$$
with $M_{E_\sigma}:=M\otimes_F E_\sigma$.

For each $\sigma \in \Sigma(F/K,K')$, there is a canonical embedding
$$V^\ast\otimes_{A_E}M_E \to V_{\sigma}^\ast \otimes_{A_{E_\sigma}}M_{E_\sigma}.$$
This induces an embedding
$$f_\sigma :{\bigwedge}_E^{rd}(V^\ast\otimes_{A_E}M_E) \to {\bigwedge}_{E_\sigma}^{rd} ( V_{\sigma}^\ast \otimes_{A_{E_\sigma}}M_{E_\sigma})$$
and we define the required scalar extension
$${\bigwedge}_A^rM={\bigwedge}_E^{rd}(V^\ast \otimes_{A_E}M_E ) \to  {\bigoplus}_{\sigma \in \Sigma(F/K,K')}{\bigwedge}_{E_\sigma}^{rd} ( V_{\sigma}^\ast \otimes_{A_{E_\sigma}}M_{E_\sigma})={\bigwedge}_{A'}^rM'$$
to be the tuple ${\bigoplus}_{\sigma} f_\sigma$.
\end{proof}

Upon combining the duality pairings (\ref{rubin pairing}) for each simple component $A_i$ of $A$ one obtains a duality pairing
\begin{equation}\label{duality pairing 2} {\bigwedge}_A^r M \times {\bigwedge}_{A^{\rm op}}^s \Hom_A(M,A) \to {\bigwedge}_A^{r-s} M\end{equation}
that we continue to denote by $(m,\varphi)\mapsto \varphi(m)$.

For each subset of elements $\{m_a\}_{1\le a\le r}$ of $M$ and $\{\varphi_a\}_{1\le a\le r}$ of $\Hom_A(M,A)$ we also set
\begin{equation}\label{exterior hom def0}\wedge_{a=1}^{a=r}m_a := ({\wedge}_{1\leq a \leq r}m_{ai})_{i \in I} \in {\bigwedge}_A^{r}M \end{equation}
and
\begin{equation}\label{exterior hom def} \wedge_{a=1}^{a=r}\varphi_a := ({\wedge}_{1\leq a \leq r}\varphi_{ai})_{i \in I} \in {\bigwedge}_{A^{\rm op}}^r \Hom_A(M,A),\end{equation}
%$\varphi_i$
%
where $m_{ai}$ and $\varphi_{ai}$ are the projection of $m_a$ and $\varphi_a$ to
 $M_i$ and $\Hom_{A_i}(M_i,A_i)$ and the component exterior powers are defined
  (via (\ref{non-commutative wedge}) and (\ref{non-commutative wedge 2})) with respect to a fixed ordered $E_i$-basis of $V_i$ (and its dual basis).

%For a left $A$-module $M$ we set $M':=K'\otimes_KM$. Then, assuming $\zeta(A)$ to be \'etale over $K$, we now construct a natural embedding (or `scalar extension')
%$${\bigwedge}_A^rM \hookrightarrow {\bigwedge}_{A'}^rM'$$
%s follows.

%and $\Hom_{A_E}(M_E,A_E)$ in the natural way.
% and, just as above, we define the primitive subspace
%$({\bigwedge}_{A^{\rm op}}^r \Hom_A(M,A))^{\rm prim}$ of ${\bigwedge}_{A^{\rm op}}^r \Hom_A(M,A)$ to be the $E$-linear span of all elements of the form $\wedge_{i=1}^{i=r}\varphi_i$.
%

%We then define a $\zeta(A)$-submodule of ${\bigwedge}_A^rM$ by setting
%
%\[ ({\bigwedge}_A^rM)^{0} := \{ x\in {\bigwedge}_A^rM: \bigl(\wedge_{i=1}^{i=r}\varphi_i\bigr)(x) = 0 \,\,\,\text{ for all } \varphi_i \in \Hom_A(M,A)\}.\]
%

\begin{remark}\label{group ring rem}{\em In each simple component of $A$ that is commutative, we will always fix conventions regarding the bases used in (\ref{exterior hom def0}) and (\ref{exterior hom def}) as in Remark \ref{comm conventions}. In addition, for the non-commutative algebras that arise in the arithmetic settings that are considered in \cite{sbes1-2}, the specification of a splitting field for $A$, of a simple $A_E$-module $V$ and of an ordered $E$-basis of $V$ arises naturally in the following way.

Let $G$ be a finite group of exponent $e$ and write $E$ for the field generated over $\QQ_p$ by a primitive $e$-th root of unity and ${\rm Ir}_p(G)$ for the set of irreducible $\QQ_p^c$-valued characters of $G$. Then, by a classical result of Brauer \cite{brauer}, for each $\chi$ in ${\rm Ir}_p(G)$ there exists a representation
\[ \rho_\chi: G \to {\rm GL}_{\chi(1)}(E)\]
of character $\chi$. The induced homomorphisms of $E$-algebras $\rho_{\chi,\ast}: E[G] \to {\rm M}_{\chi(1)}(E)$ combine to give an
isomorphism
\[ E[G] \xrightarrow{(\rho_{\chi,\ast})_\chi} {\prod}_{\chi\in {\rm Ir}_p(G)}{\rm M}_{\chi(1)}(E).\]
This decomposition shows $E$ is a splitting field for $\QQ_p[G]$, that the spaces $V_\chi := E^{\chi(1)}$, considered as the first columns of the component ${\rm M}_{\chi(1)}(E)$, are a set of representatives of the simple $E[G]$-modules and that one can specify the standard basis of $E^{\chi(1)}$ to be the ordered basis of $V_\chi$. In this way, the specification of a representation $\rho_\chi$ for each $\chi$ in ${\rm Ir}_p(G)$ leads to a canonical choice of the data necessary to define reduced exterior powers.}\end{remark}

\subsection{Basic properties}\label{sec basic}

In this section we record several useful technical properties of the reduced exterior powers defined above. % that will be used in the sequel.

\begin{lemma} \label{proprn1}
Let $A$ be a semisimple ring and $W$ an $A$-module. Then for all subsets $\{w_i\}_{1\le i\le r}$ of $W$ and $\{\varphi_j\}_{1\le j\le r}$ of $\Hom_A(W,A)$ one has
$$(\wedge_{i=1}^{i=r}\varphi_i)(\wedge_{j=1}^{j=r}w_j)=\nr_{{\rm M}_r(A^{\rm op})}((\varphi_i(w_j))_{1\leq i,j \leq r}).$$
\end{lemma}

\begin{proof} We may assume that $A$ is simple (and use the notation of Definition \ref{rth reduced ext def}) so that there is a canonical isomorphism
$A_E:=A\otimes_{\zeta(A)}E \cong \End_E(V)$. In particular, after fixing an ordered $E$-basis $\{ v_i\}_{1\le i\le d}$ of $V$, we can identify $A_E$ with the matrix ring ${\rm M}_d(E)$.

Then the definitions (\ref{non-commutative wedge}) and (\ref{non-commutative wedge 2}) combine to imply
$$(\wedge_{i=1}^{i=r}\varphi_i)(\wedge_{j=1}^{j=r}w_j)=({\wedge}_{1\leq i \leq r}({\wedge}_{1\leq j \leq d}v_j \otimes \varphi_i))({\wedge}_{1\leq i \leq r}({\wedge}_{1\leq j\leq d}v^\ast_j\otimes w_i) ).$$

Next we note that the element $(v_{i'} \otimes \varphi_i)(v_{j'}^\ast \otimes w_j)=v_{j'}^\ast(\varphi_i(w_j)v_{i'})$ of 
$E$ is equal to the $(j',i')$-component of the matrix $\varphi_i(w_j) \in A \subset {\rm M}_d(E)$.
Hence, writing $ ^t\varphi_i(w_j)$ for the transpose of $\varphi_i(w_j)\in {\rm M}_d(E)$ and regarding $( ^t\varphi_i(w_j))_{1\leq i,j \leq r}$ as a matrix in
${\rm M}_{rd}(E)$, Lemma \ref{propformula} implies that
$$(\wedge_{i=1}^{i=r}\varphi_i)(\wedge_{j=1}^{j= r}w_j)=\det( ^t\varphi_i(w_j))_{1\leq i,j \leq r}.$$
It is therefore enough to note that the last expression is equal to $\nr_{{\rm M}_r(A^{\rm op})}( ( \varphi_i(w_j))_{1\leq i,j \leq r})$ by the definition of reduced norm.
\end{proof}

\begin{remark}\label{warning}{\em Lemma \ref{proprn1} implies $(\wedge_{i=1}^{i=r}\varphi_i)(\wedge_{j=1}^{j=r}w_j)$ belongs to $\zeta(A)$. It also shows that this value is independent of the choice of ordered bases of simple modules that are used (in \S\ref{rep norm section}) to normalise the contruction of reduced exterior powers, and hence only depends on the given elements $w_1,\ldots,w_r$ and homomorphisms $\varphi_1,\ldots,\varphi_r$.}% This fact will play an important role in later sections. }%This fact is important to the formulation of our conjectures since the definitions (\ref{non-commutative wedge}) and (\ref{non-commutative wedge 2}) depend on the choice of basis $\{v_j\}_{1\le j\le d}$.}
\end{remark}

\begin{lemma}\label{pairing cor} Let $A$ be a semisimple ring and $W$ a free $A$-module of rank $r$. Then there is a canonical isomorphism of $\zeta(A)$-modules
\[\iota_W: {\bigwedge}_{A^{\rm op}}^r \Hom_A(W,A) \cong \Hom_{\zeta(A)}({\bigwedge}_{A}^r W,\zeta(A)) \]
with the following property: for any $A$-basis $\{b_i\}_{1\le i\le r}$ of $W$ one has
\[ \iota_W(\wedge_{i=1}^{i = r}b_i^\ast)(\wedge_{j=1}^{j = r}b_j) = 1,\]
where for each index $i$ we write $b_i^\ast$ for the element of $\Hom_A(W,A)$ that is dual to $b_i$. \end{lemma}

\begin{proof} If we define reduced exterior powers with respect to the  splitting algebra $\tilde E$ for $A$, then the given hypothesis on $W$ implies that the pairing (\ref{duality pairing 2}) with $s=r$ induces a homomorphism of free rank one $\tilde E$-modules 
\[\iota_W:{\bigwedge}_{A^{\rm op}}^r \Hom_A(W,A) \cong \Hom_{\zeta(A)}({\bigwedge}_{A}^r W,\zeta(A)).\]

Both the bijectivity of this homomorphism and the equality $\iota_W(\wedge_{i=1}^{i = r}b_i^\ast)(\wedge_{j=1}^{j = r}b_j) = 1$ follow directly from Lemma \ref{proprn1}. \end{proof}

\begin{lemma} \label{proprn2}
Let $A$ be a semisimple ring and $W$ a free $A$-module of rank $r$. Fix an $A$-basis $\{b_i\}_{1\le i\le r}$ of $W$. Then for each $\varphi$ in $\End_A(W)$ one has
$$\wedge_{i=1}^{i= r}\varphi(b_i) =\nr_{\End_A(W)}(\varphi)\cdot( \wedge_{i=1}^{i= r}b_i) \in {\bigwedge}_A^rW.$$
\end{lemma}

\begin{proof} The algebra isomorphism
$$\End_A(W)\xrightarrow{\sim} {\rm M}_r(A^{\rm op}); \ \psi \mapsto (b_i^\ast(\psi(b_j)))_{1\leq i,j\leq r}$$
implies that
$$\nr_{\End_A(W)}(\varphi) =\nr_{{\rm M}_r(A^{\rm op})}((b_i^\ast(\varphi(b_j)))_{1\leq i,j\leq r}).$$
By applying Lemma \ref{proprn1}, one therefore has
$$(\wedge_{i=1}^{i=r} b_i^\ast)(\wedge_{j=1}^{j= r}\varphi(b_j) )=\nr_{\End_A(W)}(\varphi).$$
This in turn implies the claimed equality since one also has $(\wedge_{i=1}^{i=r} b_i^\ast)(\wedge_{j=1}^{j=r} b_j)=1$ as a consequence of Lemma \ref{proprn1}.
%This formula follows immediately upon comparing the results of Proposition \ref{proprn1} and Corollary \ref{pairing cor}.
\end{proof}

Finally we establish a useful non-commutative generalization of Lemma \ref{prope}.

\begin{lemma}\label{lemma 4.2 generalisation}
Let $A$ be a semisimple ring and $W$ a free $A$-module of rank $r$. For a natural number $s$ with $s \leq r$ and a subset $\{\varphi_i\}_{1\le i\le s}$ of $\Hom_A(W,A)$  consider the map
$$\Phi:={\bigoplus}_{i=1}^{i=s}\varphi_i: W \to A^{\oplus s}.$$

Then the image of the map
$${\bigwedge}_A^r W \to {\bigwedge}_A^{r-s}W  ; \quad b \mapsto ({\wedge}_{1\leq i\leq s}\varphi_i)(b)$$
is contained in ${\bigwedge}_A^{r-s}\ker(\Phi)$. In addition, if $A$ is simple and $\Phi$ is surjective, respectively not surjective, then the image of this map is equal to ${\bigwedge}_A^{r-s}\ker(\Phi)$, respectively vanishes.
\end{lemma}

\begin{proof} We can assume, without loss of generality, that $A$ is simple (and then use the notation of Definition \ref{rth reduced ext def}). Then, by Morita equivalence, the kernel of the induced $E$-linear map
\[ \Phi_1 := {\bigoplus}_{i=1}^{i=s} ({\rm id}\otimes \varphi_i) : V^\ast\otimes_{A_E}W_E \to (V^\ast)^{\oplus s}.\]
is equal to $V^\ast\otimes_{A_E}\ker(\Phi)_E$ and $\Phi_1$ is surjective if and only if $\Phi$ is surjective.

We write $\{v_j\}_{1\le j\le d}$ for the ordered basis of $V$ with respect to which the exterior product ${\wedge}_{1\le i\le s}\varphi_i$ is defined (in  (\ref{non-commutative wedge 2})) and consider the $E$-linear map
\[ \Phi_2 := {\bigoplus}_{i=1}^{i=s} ({\bigoplus}_{j=1}^{j=d}v_j\otimes \varphi_i) : V^\ast\otimes_{A_E}W_E \to E^{\oplus sd}.\]

Then it is clear that $\ker(\Phi_2) = \ker(\Phi_1) = V^\ast\otimes_{A_E}\ker(\Phi)_E$ and that $\Phi_2$ is surjective if and only if $\Phi_1$ is surjective, and hence if and only if the given map $\Phi$ is surjective.

Given these facts, and the explicit definition of the reduced exterior power ${\bigwedge}_A^{r-s}\ker(\Phi)$, the claimed results follow directly upon applying Lemma \ref{prope} with the data $W,  n, m$ and $\Phi$ respectively replaced by $V^\ast\otimes_{A_E}W_E$, $dr$, $ds$ and the above map $\Phi_2$. \end{proof}

\subsection{Integral structures}\label{reduced rubin latt sec} In this subsection we assume to be given a Dedekind domain $R$ whose fraction field $F$ is a finite extension of either $\QQ$ or $\QQ_p$ for some prime $p$. We also assume to be given an $R$-order $\mathcal{A}$ that spans a (finite-dimensional) semisimple $F$-algebra $A$. We fix a non-negative integer $r$. 

For each finitely-generated $\mathcal{A}$-module $M$, we shall define, and establish the basic properties of, integral $\xi(\mathcal{A})$-structures on the $r$-th reduced exterior power of $F\otimes_R M$ over $A$. These structures will then play a key role in subsequent sections. 

For an $\mathcal{A}$-module $M$ we write $M^\ast$ for the (left) $\mathcal{A}^{\rm op}$-module $\Hom_\mathcal{A}(M,\mathcal{A})$ and $M_F$ for the  $A$-module $F\otimes_R M$.

\subsubsection{}%Just as in \S\ref{app higher fits}, 
We now fix a splitting algebra $\tilde E = {\prod}_{i \in I}E_i$ for $A$ as in (\ref{splitting algebra}) and normalise the definition of reduced exterior products as in (\ref{exterior hom def0}) and (\ref{exterior hom def}). The obvious generalization of the notion of exterior powers over commutative rings is then as follows. 
%For every non-negative integer $r$, we define a $\xi(\mathcal{A})$-module by setting
%This module contains the $\zeta(A)$-module $({\bigwedge}^r_A M_F)^0$ and so is not usually finitely generated (cf.  Lemma \ref{evaluation kernel}(ii)).

\begin{definition}\label{integral reduced exterior power def}{\em  The {\em $r$-th reduced exterior power} ${\bigwedge}_{\mathcal{A}}^r M$ of an $\mathcal{A}$-module $M$ is the $\xi(\mathcal{A})$-submodule  of ${\bigwedge}_{A}^r M_F$ that is generated by the set $\{\wedge_{i=1}^{i= r}m_i: m_i \in M\}$. }
\end{definition}

\begin{lemma}\label{spanning result} Let $M$ be a finitely generated $\mathcal{A}$-module. Then the $\xi(\mathcal{A})$-module ${\bigwedge}_{\mathcal{A}}^r M$ is finitely generated and spans ${\bigwedge}_A^rM_F$ over $\tilde E$. \end{lemma}

\begin{proof} The finite generation of ${\bigwedge}_{\mathcal{A}}^r M$ follows directly from the stronger result of Theorem \ref{exp prop}(ii) below (and so, for brevity, will not be justified here). To prove the second assertion we note that, after choosing a surjective homomorphism of $\mathcal{A}$-modules of the form $\mathcal{A}^d \to M$ (for any suitable natural number $d$), it is enough to prove the claimed result for the module $M = \mathcal{A}^d$. To do this, we write $\tilde A$ for the split central $\tilde E$-algebra $\tilde E\otimes_{\zeta(A)}A$.

Then the result of Lemma \ref{primitive span} (for each simple component of $A$) reduces us to showing that for any subset $\{m'_{j}\}_{1\le j\le r}$ of $\tilde A^d$ there exists a subset $\{m_i\}_{1\le i\le r}$ of $\mathcal{A}^d$ and an element $x$ of $\tilde E$ such that $\wedge_{j=1}^{j=r}m'_j = x\cdot \wedge_{i=1}^{i=r}m_i$.

To prove this we choose, as we may, a free $\tilde A$-submodule $X$ of $\tilde A^d$ that has rank $r$ and contains $\{m'_{j}\}_{1\le j\le r}$. It follows that $X\cap A^d$ is a free $A$-module of rank $r$ and so we can choose a basis $\{m_i\}_{1\le i\le r}$ of it that is contained in $\mathcal{A}^d$. Then $\{m_i\}_{1\le i\le r}$ is an $\tilde A$-basis of $X$  %$\Hom_{A_E}(A_E^d,A_E)$
and, writing $\lambda$ for the $\tilde A$-module endomorphism of $X$ that sends each element $m_j$ to $m'_j$,  Lemma \ref{proprn2} implies  $\wedge_{j=1}^{j= r}m'_j =x\cdot(\wedge_{i=1}^{i= r}m_i)$ with $x = \nr_{\End_{\tilde A}(X)}(\lambda)\in \tilde E$, as required. \end{proof} 

The reduced exterior power ${\bigwedge}_{\mathcal{A}}^r M$ provides an integral $\xi(\mathcal{A})$-structure on ${\bigwedge}_A^rM_F$ and, with the convention of Remark \ref{comm conventions}, coincides with the classical notion of exterior power in the case that $\mathcal{A}$ is commutative. In general, however, it depends on the chosen normalization of the exterior product (\ref{exterior hom def0}) and hence on a choice of bases of simple $A$-modules (though we suppress this fact from the notation). In addition, in arithmetic applications one often needs a slightly larger integral structure on ${\bigwedge}_A^rM_F$. To introduce this, we note that the pairing (\ref{duality pairing 2}) with $s=r$ induces well-defined isomorphisms of $\tilde E$-modules 
\[ \iota_{M_F}^{r,1}: {\bigwedge}_{A}^r M_F \to \Hom_{\tilde E}\bigl( {\bigwedge}_{A^{\rm op}}^r M^\ast_F,\tilde E\bigr)\quad\text{and}\quad \iota_{M_F}^{r,2}: {\bigwedge}_{A^{\rm op}}^r M^\ast_F \to \Hom_{\tilde E}\bigl( {\bigwedge}_{A}^r M_F,\tilde E\bigr).\]

\begin{definition}\label{red rubin def}{\em  The {\em $r$-th reduced Rubin lattice} of the $\mathcal{A}$-module $M$ is the full-preimage 
\[ {\bigcap}_{\mathcal{A}}^r M := (\iota_{M_F}^{r,1})^{-1}\bigl(\Hom_{\xi(\mathcal{A})}\bigl( {\bigwedge}_{\mathcal{A}^{\rm op}}^r M^\ast,\xi(\mathcal{A})\bigr)\bigr)\]
under $\iota_{M_F}^{r,1}$ of the $\xi(\mathcal{A})$-module $\Hom_{\xi(\mathcal{A})}\bigl( {\bigwedge}_{\mathcal{A}^{\rm op}}^r M^\ast,\xi(\mathcal{A})\bigr)$. 
%
%[ 0 \to \left({\bigwedge}^r_A M_F\right)^{\!0} \xrightarrow{\subset}  \left( {\bigcap}_{\mathcal{A}}^r M\right)' \to  {\bigcap}_{\mathcal{A}}^rM \to 0.\]
}
\end{definition}
\vskip 0.05truein

\begin{remark}\label{abelian rs}{\em More explicitly, the above definition implies that    
\begin{equation}\label{explicit rrl} {\bigcap}_{\mathcal{A}}^r M :=\{ a \in {\bigwedge}_{A}^r M_F : (\wedge_{i=1}^{i= r}\varphi_i)(a) \in\xi(\mathcal{A}) \mbox{ for all } \varphi_1,\ldots,\varphi_r \in M^\ast    \}.\end{equation}
There is also a natural isomorphism of $\xi(\mathcal{A})$-modules
%If $A$ is a split semisimple $F$-algebra, then the results of Lemmas \ref{primitive span} and \ref{pairing cor} combine to %imply that the map
$${\bigcap}_\cA^r M \to \Hom_{\xi(\cA)} \left( \iota_{M_F}^{r,2}\bigl({\bigwedge}_{\cA^{\rm op}}^r M^\ast\bigr), \xi(\cA) \right); \ a \mapsto \left(\Phi \mapsto  \Phi(a)\right). $$
If $A$ is commutative, then $\xi(\mathcal{A}) = \mathcal{A}$ (by Lemma \ref{xi lemma}(iii)) and, with respect to the conventions fixed in Remark \ref{comm conventions}, this map induces a canonical  isomorphism of $\mathcal{A}$-modules
\begin{equation}\label{rubin lattice iso}{\bigcap}_\cA^r M  \cong \Hom_\cA\left({\bigwedge}_\cA^rM^\ast, \cA\right)\end{equation}
(see \cite[Prop. A.7]{sbA}). Modules of the latter form were first considered (in the setting of group rings of abelian groups) by Rubin in \cite{R} in order to formulate refined versions of Stark's Conjecture. These lattices are in turn a special case of the formalism of `exterior power biduals' that has subsequently played a key role in the theory of higher rank Euler, Kolyvagin and Stark systems that is developed by Sakamoto and the present authors in \cite{bss}.}
\end{remark}

%\begin{remark}
%{\em
%When $A$ is a split semisimple $F$-algebra, one checks by using Lemmas \ref{primitive span} and \ref{pairing cor} that
%$${\bigcap}_\cA^r M \xrightarrow{\sim} \Hom_{\xi(\cA)} \left(\xi(\cA)\cdot \{\wedge_{i=1}^{i=r} \varphi_i : \varphi_i \in \Hom_\cA(M,\cA)\}, \xi(\cA) \right); \ a \mapsto \left(\Phi \mapsto  \Phi(a)\right) $$
%is an isomorphism. If $A$ is commutative, then the right hand side coincides with
%$$\Hom_\cA\left({\bigwedge}_\cA^r \Hom_\cA(M,\cA), \cA\right),$$
%which is exactly the definition of the exterior power bidual. }
%\end{remark}

The basic properties of reduced Rubin lattices in the general case are recorded in the following result. We note, in particular, that this result verifies  ${\bigcap}_{\mathcal{A}}^rM$ provides an integral $\xi(\mathcal{A})$-structure on  ${\bigwedge}_A^rM_F$ that is both functorially well-behaved and varies naturally with the chosen normalisation of reduced exterior products. 

\begin{theorem}\label{exp prop} For each finitely generated $\mathcal{A}$-module $M$ and  non-negative integer $r$ the following claims are valid. %Let $M$ and $r$ be as in Definition \ref{rs def}.
\begin{itemize}
%\item[(i)] ${\bigcap}_{\mathcal{A}}^rM$ contains $({\bigwedge}^r_A M_F)^0$, is stable under multiplication by $\xi(\mathcal{A})$ and spans ${\bigwedge}_{A}^r M_F$.
%
\item[(i)] If $r = 0$, then ${\bigcap}_{\mathcal{A}}^r M = \xi(\mathcal{A})$.

%\item[(ii)] The $\xi(\mathcal{A})$-module that is generated by the elements ${\wedge}_{i=1}^{i=r}m_i$ as $\{m_i\}_{1\le i%\le r}$ ranges over subsets of $M$ is contained in ${\bigcap}_{\mathcal{A}}^r M$ contains ${\bigwedge}_{\mathcal{A}}^r M$. 
%
\item[(ii)] The $\xi(\mathcal{A})$-module ${\bigcap}_{\mathcal{A}}^rM$ contains ${\bigwedge}_{\mathcal{A}}^r M$, is finitely generated and torsion-free over $R$ and spans ${\bigwedge}_A^rM_F$ over $\tilde E$.  It also varies naturally with the choice of bases (of the simple modules $V$) that occur in the definition (\ref{non-commutative wedge 2}) of exterior powers over each simple component of $A$.
\item[(iii)] For every prime ideal $\mathfrak{p}$ of $R$ one has $\left({\bigcap}_{\mathcal{A}}^rM\right)_{(\mathfrak{p})} = {\bigcap}_{\mathcal{A}_{(\mathfrak{p})}}^rM_{(\mathfrak{p})}$. Hence one has
\[ {\bigcap}_{\mathcal{A}}^r M = {\bigcap}_{\frp \in {\rm Spec}(R)}\left( {\bigcap}_{\mathcal{A}_{(\mathfrak{p})}}^rM_{(\mathfrak{p})}\right).\]
%
%where the intersection runs over all $\mathfrak{p}$ and takes place in $\left({\bigcap}_{\mathcal{A}}^r M\right)_F$.
%
\item[(iv)] For each map of finitely generated $\mathcal{A}$-modules $\iota: M \to M'$, there exists an induced map of $\zeta(A)$-modules $\iota_{*,F}^r: {\bigwedge}^r_A M_F \to {\bigwedge}^r_A M_F'$ that restricts to give a map of $\xi(\mathcal{A})$-modules
 $\iota^r_{*}: {\bigcap}_{\mathcal{A}}^rM \to  {\bigcap}_{\mathcal{A}}^rM'$. If $\iota$ is injective, then so are $\iota_{*,F}^r$ and $\iota_{*}^r$. If $\iota$ is injective and, in addition, the group ${\rm Ext}^1_{\mathcal{A}}({\rm cok}(\iota),\mathcal{A})$ vanishes, then one has 
\[  \iota_*^r \left({\bigcap}_{\mathcal{A}}^rM\right) = \iota_{*,F}^r\left({\bigwedge}^r_A M_F\right) \cap {\bigcap}_{\mathcal{A}}^rM'. \]
\item[(v)] Let $s$ be a natural number with $s \le r$ and $\{\varphi_i\}_{ 1\le i\le s}$ a subset of $\Hom_\mathcal{A}(M,\mathcal{A})$. Then the map
$${\bigwedge}_A^r M_F \to {\bigwedge}_A^{r-s}M_F  ; \quad x \mapsto ({\wedge}_{1\leq i\leq s}\varphi_i)(x)$$
sends ${\bigcap}_{\mathcal{A}}^rM$ into ${\bigcap}_{\mathcal{A}}^{r-s}M$.
\item[(vi)] If $M$ is a free $\mathcal{A}$-module of rank $d$ with $d \ge r$, then for any choice of basis $b=\{ b_j\}_{1\le j\le d}$ of $M$ there is a natural split surjective homomorphism of $\xi(\mathcal{A})$-modules
$$ \theta_b: {\bigcap}_\mathcal{A}^r M \to {\bigoplus}_{\sigma\in{d \atopwithdelims[] r} }\xi(\mathcal{A}).$$
This homomorphism is bijective if and only if either  $\mathcal{A}$ is commutative or $r=d$.
\item[(vii)] Assume that $\mathcal{A} = {\rm M}_n(\mathcal{B})$ for a natural number $n$ and commutative $R$-order $\mathcal{B}$. Then $\xi(\mathcal{A})= \mathcal{B}$ and,  for any  finitely generated $\mathcal{A}$-module $M$,  one has both
\[ {\bigwedge}_A^rM_F = {\bigwedge}_{\mathcal{B}_F}^{nr} ( \mathcal{B}^n\otimes_{\mathcal{A}}M_F) \quad \text{and} \quad {\bigcap}_\cA^r M = {\bigcap}_{\mathcal{B}}^{nr}  (\mathcal{B}^n\otimes_{\mathcal{A}}M),\]
%\left({\bigwedge}_{\mathfrak{A}}^{rn}(\right)_{\rm tf},\]
%
where  $\mathcal{B}^n$ denotes the right $\mathcal{A}$-module comprising row vectors of length $n$ over  $\mathcal{B}$.   
%and the subscript `tf' indicates that we have taken the quotient of the exterior product by its $\mathfrak{A}$-torsion submodule.
\end{itemize}
\end{theorem}

\begin{proof} We fix a Wedderburn decomposition (\ref{wedderburn}) of $A$ and assume throughout this argument that reduced exterior powers are defined with respect to the splitting algebra $\tilde E = {\prod}_{i \in I}E_i$ fixed above. 

Then, to prove (i), we note that $\iota_{M_F}^{0,1}$ is the identify function on the algebra $\tilde E = {\bigwedge}_{A}^0M_F = {\bigwedge}_{A^{\rm op}}^0M^\ast_F$. We also note that, by convention, the exterior power of the empty subset of $M^\ast$ is the identity element of $\tilde E$ and hence that ${\bigwedge}_{\mathcal{A}^{\rm op}}^0 M^\ast = \xi(\mathcal{A})$. Claim (i) is therefore true since, in $\Hom_{\tilde E}(\tilde E, \tilde E) = \tilde E$, one has  
\[ \Hom_{\xi(\mathcal{A})}\bigl( {\bigwedge}_{\mathcal{A}^{\rm op}}^0 M^\ast,\xi(\mathcal{A})\bigr) = \Hom_{\xi(\mathcal{A})}\bigl(\xi(\mathcal{A}),\xi(\mathcal{A})\bigr) = \xi(\mathcal{A}). \]
% 
%note that ${\bigwedge}_{A}^0 M_F = {\bigoplus}_{i \in I}E_i$ and, by convention, the exterior power of the empty subset of $\Hom_A(M,A)$ is the identity element of the algebra ${\bigwedge}_{A^{\rm op}}^0 \Hom_A(M_F,A) = {\bigoplus}_{i \in I}E_i$. In view of these descriptions, claim (i) follows directly from the explicit definition of ${\bigcap}_\mathcal{A}^rM$ in the case $r=0$. 
Given this result, we can assume in the rest of the proof that $r > 0$.

It is convenient to prove (iv) next. To do this we note the existence of a homomorphism of $\zeta(A)$-modules $\iota_{\ast,F}^r$ of the stated form is a consequence of the fact that for every simple $A_E$-module $V$ the given map $\iota$ induces a homomorphism of $E$-vector spaces $\iota^r_V: {\bigwedge}_E^{rd}(V^\ast\otimes_{A_E}M_E) \to {\bigwedge}_E^{rd}(V^\ast\otimes_{A_E}M'_E)$. We note further that if $\iota$ is injective, then each map $\iota^r_V$ is injective (as the algebra $A_E$ is semisimple) and so $\iota_{\ast,F}^r$ is also injective, as claimed.  

We write $\iota^\ast: (M')^\ast \to M^\ast$ for the homomorphism of $\mathcal{A}^{\rm op}$-modules that is induced by $\iota$. Then $\iota^\ast(\varphi')(m) = \varphi'(\iota(m))$ for every $\varphi'\in (M')^\ast$ and $m\in M$, and so Lemmas \ref{proprn1} and \ref{spanning result} combine to imply that, for every $x$ in ${\bigcap}_{\mathcal{A}}^rM$ and subset $\{\varphi_i'\}_{1\le i\le r}$ of $(M')^\ast$, one has
\begin{equation}\label{func equal} ({\wedge}_{i=1}^{i=r}\varphi'_i)(\iota_{\ast,F}^r(x)) = {\wedge}_{i=1}^{i=r}(\iota^\ast(\varphi'_i))(x).\end{equation}
These equalities imply $\iota_{\ast,F}^r(x)\in {\bigcap}_{\mathcal{A}}^rM'$ for all $x \in {\bigcap}_{\mathcal{A}}^rM$, and hence that $\iota_{\ast,F}^r$ restricts to give a map of $\xi(\mathcal{A})$-modules $\iota_\ast^r: {\bigcap}_{\mathcal{A}}^rM \to {\bigcap}_{\mathcal{A}}^rM'$. In addition, if $\iota$ is injective, then  the functor $\Hom_\mathcal{A}(-,\mathcal{A})$ applies to the tautological exact sequence $0 \to M\xrightarrow{\iota} M' \to {\rm cok}(\iota) \to 0$ to give an exact sequence
\[ (M')^\ast \xrightarrow{\iota^\ast} M^\ast \to {\rm Ext}^1_{\mathcal{A}}({\rm cok}(\iota),\mathcal{A}).\]
In particular, if  ${\rm Ext}^1_{\mathcal{A}}({\rm cok}(\iota),\mathcal{A})$ also vanishes, then  $\iota^\ast$ is surjective and so the final assertion of (iv) is a consequence of the equalities (\ref{func equal}). 

Turning now to (ii), the inclusion ${\bigwedge}_{\mathcal{A}}^rM \subseteq {\bigcap}_{\mathcal{A}}^rM$ follows directly upon comparing Definition \ref{integral reduced exterior power def} with the description (\ref{explicit rrl}) of ${\bigcap}_{\mathcal{A}}^rM$ and the explicit formula in Lemma \ref{proprn1}. In view of Lemma \ref{spanning result}, this inclusion proves that ${\bigcap}_{\mathcal{A}}^rM$ spans ${\bigwedge}_A^rM_F$ over $\tilde E$. It  is also clear that ${\bigcap}_{\mathcal{A}}^rM$ is $R$-torsion-free and to prove it is finitely generated we observe  there exists a natural number $d$ and an injective homomorphism of $\mathcal{A}$-modules from the quotient  $M_{\rm tf}$ of $M$ by its $R$-torsion submodule to $\mathcal{A}^d$. To justify this we note that the semisimplicity of $A$ implies the existence, for any large enough $d$, of an injective homomorphism of $A$-modules $\iota': M_F \to A^d$. Hence, since $M_{\rm tf}$ is a finitely generated $R$-submodule of $M_F$, there exists a non-zero element $x$ of $R$ such that the composite homomorphism $M_{\rm tf}\subset M_F \xrightarrow{\iota'} A^d \xrightarrow{\times x} A^d$ factors through the inclusion $\mathcal{A}^d \subseteq A^d$ and so gives an injective homomorphism $\iota:M_{\rm tf} \to \mathcal{A}^d$ of the required form.

Upon applying (iv) to $\iota$ one sees that it is therefore enough to prove the finite generation of ${\bigcap}_{\mathcal{A}}^rM = {\bigcap}_{\mathcal{A}}^rM_{\rm tf}$ in the case $M = \mathcal{A}^d$. %To deal with this case we fix a splitting field $E$ for $A$ that has finite degree over $F$ and set $A_E := E\otimes_{\zeta(A)}A$. % and write $\mathcal{O}$ for the integral closure of $R$ in $E$.
 In this case, Lemma \ref{pairing cor} reduces us to showing that  ${\bigwedge}_{A^{\rm op}}^r \Hom_A(A^d,A)$ is the $\tilde E$-linear span of elements of the form $\wedge_{i=1}^{i= r}\varphi_i$ as $\varphi_i$ ranges over $(\mathcal{A}^d)^\ast$ and this follows from Lemma \ref{spanning result} (with $\mathcal{A}$ and $M$ taken to be $\mathcal{A}^{\rm op}$ and $(\mathcal{A}^d)^\ast$ respectively).

To make precise and prove the second assertion of (ii), we may assume that $A$ is simple, with splitting field $E$, that $M$ is an $\mathcal{A}$-lattice (regarded as a subset of $M_E$ in the obvious way) and that  $\{\tilde v_j\}_{1\le j\le d}$ is an alternative to the $E$-basis $\{ v_j\}_{1\le j \le d}$ of the simple $A_E$-module $V$ that is used in (\ref{non-commutative wedge 2}). We then 
write $\tau$ for the (unique) element of $(A_E^{\rm op})^\times \cong {\rm Aut}_E(V)^{\rm op}$ such that $v_j\cdot \tau = \tilde  v_j$ for all $j$, and $\mathcal{A}^\tau$ for the $R$-order $\tau^{-1}\mathcal{A}\tau$ in $\tau^{-1} A\tau$. Then one has $\xi(\mathcal{A}^\tau) = \xi(\mathcal{A})$ and the map $\varphi\mapsto (m \mapsto \varphi(m)\tau)$ induces a bijection $\Hom_{\mathcal{A}}(M,\mathcal{A}) \cong \Hom_{\mathcal{A}^\tau}(\tau^{-1}(M),\mathcal{A}^\tau)$. This in turn implies an equality ${\bigwedge}^{r,\sim}_{\mathcal{A}^{\rm op}}\Hom_{\mathcal{A}}(M,\mathcal{A}) = {\bigwedge}^{r}_{(\mathcal{A}^\tau)^{\rm op}}\Hom_{\mathcal{A}^\tau}(\tau^{-1}(M),\mathcal{A}^\tau)$ of $\xi(\mathcal{A})$-lattices, and hence that ${\bigcap}_{\mathcal{A}}^{r,\sim}M = {\bigcap}_{\mathcal{A}^\tau}^{r}\tau^{-1}(M)$, where the superscripts `$\sim$' indicate the reduced exterior power and Rubin lattice are defined relative to $\{\tilde v_j\}_{1\le j\le d}$ rather than $\{ v_j\}_{1\le j\le d}$. This completes the proof of (ii).

We note next that $M^\ast$ is contained in
$M^\ast_{(\mathfrak{p})} = \Hom_{\mathcal{A}_{(\mathfrak{p})}}(M_{(\mathfrak{p})},
\mathcal{A}_{(\mathfrak{p})})$ for each prime ideal $\mathfrak{p}$ of $R$.
In particular, since $\xi(\mathcal{A})_{(\mathfrak{p})} = \xi(\mathcal{A}_{(\mathfrak{p})})$ (by Lemma \ref{xi lemma}(ii)) it is easily seen that $\left({\bigcap}_{\mathcal{A}}^rM\right)_{(\mathfrak{p})}$ contains ${\bigcap}_{\mathcal{A}_{(\mathfrak{p})}}^rM_{(\mathfrak{p})}$.

To show the reverse inclusion we fix an element $a$ of $\left({\bigcap}_{\mathcal{A}}^rM\right)_{(\mathfrak{p})}$, a subset $\{\varphi_j\}_{1\le j\le r}$ of maps in  $\Hom_{\mathcal{A}_{(\mathfrak{p})}}(M_{(\mathfrak{p})},\mathcal{A}_{(\mathfrak{p})}) = \Hom_{\mathcal{A}}(M,\mathcal{A})_{(\mathfrak{p})}$ and an element $x$ of $R\cap R_{(\mathfrak{p})}^\times$ such that each $x\cdot \varphi_j$ belongs to $\Hom_{\mathcal{A}}(M,\mathcal{A})$. Then one has 
\[ (\wedge_{j=1}^{j= r}\varphi_j)(a) = {\rm Nrd}_A(x^{-r})\cdot (\wedge_{j=1}^{j= r}(x\cdot\varphi_j))(a) \in {\rm Nrd}_A(x)^{-r}\cdot \xi(\mathcal{A})_{(\mathfrak{p})} = \xi(\mathcal{A})_{(\mathfrak{p})}.\]
Here the first equality follows directly from the explicit definition (\ref{exterior hom def}) of reduced exterior powers, the containment is valid since $a\in \left({\bigcap}_{\mathcal{A}}^rM\right)_{(\mathfrak{p})}$ and the last equality since ${\rm Nrd}_A(x)$ is a unit of $\xi(\mathcal{A}_{(\mathfrak{p})}) = \xi(\mathcal{A})_{(\mathfrak{p})}.$ This shows  ${\bigcap}_{\mathcal{A}_{(\mathfrak{p})}}^rM_{(\mathfrak{p})}$ contains   $\left({\bigcap}_{\mathcal{A}}^rM\right)_{(\mathfrak{p})}$ and hence completes the proof that ${\bigcap}_{\mathcal{A}_{(\mathfrak{p})}}^rM_{(\mathfrak{p})} = \left({\bigcap}_{\mathcal{A}}^rM\right)_{(\mathfrak{p})}$. Given this, the displayed equality in (iii) then follows directly from the general result of \cite[Prop. (4.21)(vi)]{curtisr}.

Next we note that (v) is true because the definition of the lattice ${\bigcap}_{\mathcal{A}}^rM$ ensures that for any subset $\{\vartheta_j\}_{ 1\le j\le r-s}$ of $M^\ast$ and any $x$ in ${\bigcap}_{\mathcal{A}}^rM$ one has
\[ ({\wedge}_{1\leq j\leq r-s}\vartheta_j)\bigl(({\wedge}_{1\leq i\leq s}\varphi_i)(x)\bigr) =  \bigl(({\wedge}_{1\leq i\leq s}\varphi_i)\wedge ({\wedge}_{1\leq j\leq r-s}\vartheta_j)\bigr)(x)\in \xi(\mathcal{A}).\]
To prove (vi) we define $\theta_b$ to be the map of $\xi(\mathcal{A})$-modules that satisfies
\[ \theta_b(x) = ((\wedge_{i=1}^{i = r}b^\ast_{\sigma(i)})(x))_{\sigma\in{d \atopwithdelims[] r}}\]
for all $x$ in ${\bigcap}_\mathcal{A}^rM$. % and then set $D_b({\bigcap}_\mathcal{A}^r M):=\ker(\theta_b).$
 We write $\theta'_b$ for the map of $\xi(\mathcal{A})$-modules ${\bigoplus}_{\sigma\in{d \atopwithdelims[] r} }\xi(\mathcal{A}) \to {\bigcap}_\mathcal{A}^r M$ which satisfies

\[ \theta'_b((c_\sigma)_\sigma) = {\sum}_{\sigma\in{d \atopwithdelims[] r}}c_\sigma \cdot \wedge_{i=1}^{i = r}b_{\sigma(i)}\]
for all $(c_\sigma)_\sigma$ in ${\bigoplus}_{\sigma\in{d \atopwithdelims[] r} }\xi(\mathcal{A})$. Then Lemma \ref{proprn1} implies that $({\wedge}_{j=1}^{j=r}b_{\sigma(j)}^\ast)({\wedge}_{i=1}^{i=r}b_{\tau(i)})=\delta_{\sigma\tau}$ for all $\sigma$ and $\tau$ in ${d \atopwithdelims[] r}$ and so the composite $\theta_b\circ \theta'_{b}$ is the identity on ${\bigoplus}_{\sigma\in{d \atopwithdelims[] r} }\xi(\mathcal{A})$. This shows that $\theta'_b$ is a section to $\theta_b$, as required.

%Now claim (ii) implies that the $R$-module $D_b({\bigcap}_\mathcal{A}^r M)$ is finitely generated (being a submodule of ${\bigcap}_\mathcal{A}^r M$). The module $D_b({\bigcap}_\mathcal{A}^r M)$ is therefore finite if it is $R$-torsion and this also follows from claim (iii) since the set $\{\wedge_{i=1}^{i = r}b_{\sigma(i)}: \sigma\in{d \atopwithdelims[] r}\}$ spans $F\cdot{\bigcap}_\mathcal{A}^r M = {\bigwedge}_A^r M_F$ over $\zeta(A)$ and so the map $F\otimes_R\theta_b$ is bijective.

Next we note that if $\mathcal{A}$ is commutative, then ${\bigcap}_\mathcal{A}^rM={\bigwedge}_\mathcal{A}^rM$ (as $M$ is free) and $\xi(\mathcal{A}) = \mathcal{A}$ and using these equalities it is easily seen that $\theta_b$ is an isomorphism.

To complete the proof of (vi) it is therefore enough to fix a primitive central idempotent $e$ of $A$ for which $eA$ is not commutative and to show that $e(F\otimes_R\ker(\theta_b))$ vanishes if and only if $r = d$.

We fix a splitting field for $Ae$ and a simple $E\otimes_{\zeta(A)e}Ae$ module $V$. We set $n := {\rm dim}(V)$ so that $n > 1$. Then $V^\ast\otimes_{A_E}M_E$ is an $E$-vector space of dimension $nd$ and so
 $e\bigl({\bigcap}_\mathcal{A}^r M\bigr)$ spans an $E$-vector space
 of dimension ${nd \atopwithdelims() nr}$. Since $e\bigl({\bigoplus}_{\sigma\in{d \atopwithdelims[] r} }\xi(\mathcal{A})\bigr)$ spans an $E$-vector space of dimension ${d\atopwithdelims()r}$ it is therefore enough to show that ${nd \atopwithdelims() nr}$ is equal to ${d \atopwithdelims() r}$ if and only if $r=d$ and we leave this as an exercise for the reader. 
 
Turning to (vii), we note that the algebra $B:= \mathcal{B}_F$ is a finite direct product ${\prod}_{i\in I}F_i$ of finite degree field extensions $F_i$ of $F$. Then, for any natural number $m$, the induced algebra decomposition
\[ {\rm M}_m(A) \xrightarrow[\theta]{\sim} {\rm M}_{ mn}(B) \xrightarrow{\sim} {\prod}_{i \in I}{\rm M}_{ mn}(F_i)\]
implies that, for every $N$ in ${\rm M}_m(A)$, one has ${\rm Nrd}_A(N) = {\rm det}\bigl(\theta(N)\bigr)\in B$. In particular, since $\theta \bigl({\rm M}_n(\mathcal{A})\bigr) = {\rm M}_{mn}(\mathcal{B})$, one has $\xi(\mathcal{A}) = \mathcal{B}$.
 
In a similar way, for any finitely generated $\mathcal{A}$-module $M$ the claimed description of ${\bigwedge}_A^rM_F$ follows directly from the explicit definition of reduced exterior powers (via Definition \ref{rth reduced ext def} for each simple component $A_i$ of $A$ and with the respective splitting fields $E$ taken to be $F_i$) and the fact that, for each index $i$, $F_i^{n}$ is a simple right ${\rm M}_{n}(F_i)$-module. 

Finally, we note that Morita equivalence gives an isomorphism of $\mathcal{B}$-modules
\[ \Hom_\mathcal{A}(M,\mathcal{A}) \cong \Hom_{\mathcal{B}}(\mathcal{B}^n\otimes_{\mathcal{A}}M,\mathcal{B}^n) \cong \Hom_{\mathcal{B}}(\mathcal{B}^n\otimes_{\mathcal{A}}M,\mathcal{B})^n\]
in which the second map is induced by the standard basis of the $\mathcal{B}$-module $\mathcal{B}^n$. Given this isomorphism, the claimed description of ${\bigcap}^r_{\mathcal{A}}M$ follows directly from the given description of ${\bigwedge}_{A}^rM_F$ and the explicit definition (via (\ref{non-commutative wedge 2}) and (\ref{exterior hom def})) of the reduced exterior products ${\wedge}_{i=1}^{i=r}\varphi_i$ of elements $\varphi_i$ of $M^\ast$. This completes the proof of  (vii). 
\end{proof}

\begin{remark}{\em Orders of the form discussed in Theorem \ref{exp prop}(vii) arise naturally in the setting of the group rings discussed in Example \ref{dj examples}(i). For such orders, the second displayed equality in Theorem \ref{exp prop}(vii) combines with the isomorphism (\ref{rubin lattice iso}) (with $\cA$, $r$ and $M$ taken to be $\mathcal{B}$, $nr$ and $\mathcal{B}^n\otimes_{\mathcal{A}}M$) to give an explicit description of the lattice ${\bigcap}_\cA^r M $.  
%and the subscript `tf' indicates that we have taken the quotient of the exterior product by its $\mathfrak{A}$-torsion submodule.
}\end{remark}

\subsubsection{}\label{useful cons section}Lemma \ref{lemma 4.2 generalisation} gives rise to a useful construction of elements in reduced Rubin lattices. To describe  this we identify each matrix $M$ in ${\rm M}_{d',d}(\mathcal{A})$ with the homomorphism of $\mathcal{A}$-modules
\[ \theta_M:\mathcal{A}^{d'} \to \mathcal{A}^d\]
that sends each (row) vector $x$ to $x\cdot M$.

For each primitive central idempotent $e$ of $A$ we fix a non-zero simple (left) $A$-module $V(e)$ upon which $e$ acts as the identity and we write $D(e)$ for the associated division ring ${\rm End}_A(V(e))$. (We recall that such a module $V(e)$ is unique up to isomorphism.)

%Remark \ref{reduced rank rem} implies that for any non-negative integer $r$ and any finitely generated $A$-module $W$, one has ${\rm rr}_{Ae}\bigl(eW)\bigr)={\rm rr}_{Ae}\bigl( (Ae)^r\bigr)$ if and only if $V(e)$ occurs with exact multiplicity $r\cdot {\rm dim}_{D(e)}(V(e))$ in the Wedderburn decomposition of $W$.

\begin{proposition}\label{matrix construction} Fix natural numbers $d'$ and $d$ with $d' >  d$ and set $r:=d'-d >0$. Then for each matrix $M$ in ${\rm M}_{d',d}(\mathcal{A})$ for which ${\rm Ext}_{\mathcal{A}}^1(\im(\theta_M),\mathcal{A})$ vanishes, there exists a canonical element $\varepsilon_M$ of ${\bigcap}^{r}_{\mathcal{A}}\ker(\theta_M)$ that has both of the following properties. %\begin{itemize}
%\item[(i)] For all $\varphi_1,\ldots,\varphi_r$ in $\Hom_\mathcal{A}(\ker(M),\mathcal{A})$ one has $(\bigwedge_{i=1}^{i= r}\varphi_i)(a) \in\xi(\mathcal{A})$.
%
\begin{itemize}
\item[(i)] For a primitive central idempotent $e$ of $A$ the following conditions are equivalent.
\begin{itemize}
\item[(a)] $e(\varepsilon_M)\not= 0.$
\item[(b)] ${\rm rr}_{Ae}\bigl(e\cdot \ker(\theta_M)_F\bigr)={\rm rr}_{Ae}\bigl( (Ae)^r\bigr).$
\item[(c)] $V(e)$ occurs with exact multiplicity $r\cdot {\rm dim}_{D(e)}(V(e))$ in the Wedderburn decomposition of $\ker(\theta_M)_F$.
    \item[(d)] $\zeta(A)\cdot e(\varepsilon_M) = e\cdot ({\bigcap}^{r}_{\mathcal{A}}\ker(\theta_M))_F$.
\end{itemize}
\item[(ii)] Write $0_{d',r}$ for the $d'\times r$ zero matrix. Then for the block matrix $(0_{d',r} \mid M)$ in ${\rm M}_{d'}(\mathcal{A})$ one has
     %Write $\iota: \mathcal{A}^d \hookrightarrow \mathcal{A}^{d'}$ for the injective homomorphism of $\mathcal{A}$-modules that sends the $i$-th element, for each $i$ with $1\le i\le d$, of the standard basis of $\mathcal{A}^d$ to the $(r+i)$-th element of the standard basis of $\mathcal{A}^{d'}$. Then the matrix in ${\rm M}_{d'}(\cA)$ that represents (with respect to the standard basis) the endomorphism $\iota\circ \theta_M$ of $\mathcal{A}^{d'}$ is the block matrix $(0_{d',r}\mid M)$ and  one has
 \[ {\rm Fit}_\cA^r((0_{d',r} \mid M)) = \xi(\mathcal{A})\cdot \{ ({\wedge}_{i=1}^{i=r}\varphi_i)(\varepsilon_M) : \varphi_i \in \Hom_\mathcal{A}(\ker (\theta_M), \mathcal{A})\}.\]
 In particular, for each subset $\{\varphi_i\}_{1\le i\le r}$ of $\Hom_\mathcal{A}(\ker (\theta_M), \mathcal{A})$ one has
 \[ ({\wedge}_{i=1}^{i=r}\varphi_i)(\varepsilon_M) \in {\rm Fit}_\mathcal{A}^0 ({\rm cok}( \theta_M)).\]
  %\subseteq {\rm Fit}_\mathcal{A}^0 ({\rm cok}( \theta_M)) \subseteq {\rm pAnn}_\mathcal{A} ({\rm cok}( \theta_M)).
 %\end{multline*}
%
%In particular, one has
% $$\xi(\mathcal{A})\cdot \{ ({\wedge}_{i=1}^{i=r}\varphi_i)(\varepsilon_M) : \varphi_i \in \Hom_\mathcal{A}(\ker (\theta_M), \mathcal{A})\}\subseteq {\rm pAnn}_\mathcal{A} ({\rm cok}( \theta_M)).$$
\end{itemize}
%  if and only if $V(e)$ occurs with exact multiplicity $(d'-d)\cdot {\rm dim}_{D(e)}(V(e))$ in the Wedderburn decomposition of $\ker(\theta_M)_F$.
\end{proposition}

\begin{proof} For each integer $i$ with $1\le i\le d$ we write $\theta_M^i$ for the element of $\Hom_\mathcal{A}(\mathcal{A}^{d'},\mathcal{A})$ that sends $x$ to the $i$-th component of the element $\theta_M(x)$ of $\mathcal{A}^d$. %We also set $r := d'-d$.

Then Lemma \ref{lemma 4.2 generalisation} implies that the image of the homomorphism
$$\Theta_M: {\bigcap}_\mathcal{A}^{d'} \mathcal{A}^{d'} \to {\bigcap}_\mathcal{A}^{r}\mathcal{A}^{d'}  ; \quad b \mapsto ({\wedge}_{1\leq i\leq d}\,\theta_M^i)(b)$$
coming from Theorem \ref{exp prop}(v) is contained in the submodule ${\bigwedge}_A^{r}\ker(\theta_M)$ of ${\bigwedge}_A^{r}A^{d'}$.

Thus, since ${\rm Ext}_{\mathcal{A}}^1(\im(\theta_M),\mathcal{A})$ is assumed to vanish, the final assertion of Theorem \ref{exp prop}(iv) implies $\im(\Theta_M)$ is contained in ${\bigcap}_\mathcal{A}^{r}\ker(\theta_M)$.

In particular, if we write $\{b_i\}_{1\le i\le d'}$ for the canonical basis of $\mathcal{A}^{d'}$, then ${\wedge}_{i=1}^{i=d'}b_i$ belongs to ${\bigcap}_\mathcal{A}^{d'}\mathcal{A}^{d'}$ (by Lemma \ref{proprn1}) and so we can define
$$\varepsilon_M:=\Theta_M\bigl({\wedge}_{i=1}^{i=d'}b_i\bigr) \in {\bigcap}_\mathcal{A}^{r}\ker(\theta_M).$$

We now fix a primitive central idempotent $e$ of $A$, a splitting field $E$ for the simple algebra $Ae$, a simple $E\otimes_{\zeta(A)}A$-module $V$ and an $E$-basis $\{v_j^\ast\}_{1\le j\le t}$ of $V^*$ (so $t = {\rm dim}_E(V)$). Then, since ${\rm rr}_{Ae}\bigl(e(\ker(\theta_M)_F)\bigr)\ge {\rm rr}_{Ae}\bigl( ( Ae)^r\bigr)$, the $E$-space $e\cdot {\bigwedge}^r_A\ker(\theta_M)_F$ $= {\bigwedge}_{Ae}^re(\ker(\theta_M)_F)$ does not vanish and has dimension one if and only if ${\rm rr}_{Ae}\bigl(e(\ker(\theta_M)_F)\bigr) = {\rm rr}_{Ae}\bigl( ( Ae)^r\bigr)$. In addition, since $\{v_j^*\otimes b_i\}_{1\le j\le t, 1\le i\le d'}$ is an $E$-basis of $V^*\otimes_{A_E}A_E^{d'}$, the $E$-space  $e\cdot {\bigwedge}_A^{d'}A^{d'}$ has dimension one, with basis $e({\wedge}_{i=1}^{i=d'}b_i)$. Upon combining these observations with the  final assertion of Lemma  \ref{lemma 4.2 generalisation}, it therefore follows that
\begin{align*} {\rm rr}_{Ae}\bigl(e(\ker(\theta_M)_F)\bigr) = {\rm rr}_{Ae}\bigl( ( Ae)^r\bigr) \Longrightarrow&\, e(\varepsilon_M)\not= 0\\
                             \Longrightarrow&\, E\cdot e(\varepsilon_M) = e\cdot{\bigwedge}_{A}^r\ker(\theta_M)_F\\
                             \Longrightarrow&\,{\rm rr}_{Ae}\bigl(e(\ker(\theta_M)_F)\bigr) = {\rm rr}_{Ae}\bigl( ( Ae)^r\bigr)\end{align*}
and hence that the conditions (a) and (b) stated in (i) are each equivalent to the equality $\,E\cdot e(\varepsilon_M) = e\cdot{\bigwedge}_{A}^r\ker(\theta_M)_F$.

Further, since Lemma \ref{spanning result} implies that $ e\cdot{\bigwedge}_{A}^r\ker(\theta_M)_F$ is equal to the $E$-linear span of $e\cdot({\bigcap}_{\mathcal{A}}^r\ker(\theta_M))_F$, one has $E\cdot e(\varepsilon_M) = e\cdot{\bigwedge}_{A}^r\ker(\theta_M)_F$ if and only if the inclusion $ \zeta(A)\cdot e(\varepsilon_M)\subseteq e\cdot ({\bigcap}^{r}_{\mathcal{A}}\ker(\theta_M))_F$ is an equality (as stated in condition (d)). To complete the proof of (i) it is therefore enough to note that the equivalence of conditions (b) and (c) follows directly from  Remark \ref{reduced rank rem}.

To prove (ii), we note first that the assumed vanishing of ${\rm Ext}_{\mathcal{A}}^1(\im(\theta_M),\mathcal{A})$ implies the restriction map $\Hom_\mathcal{A}(\mathcal{A}^{d'},\mathcal{A})\to \Hom_\mathcal{A}(\ker(\theta_M), \mathcal{A})$ is surjective. This fact combines with the definition of $\varepsilon_M$ and the result of Lemma  \ref{proprn1} to imply that
\begin{multline}\label{first eq claim (ii)}\xi(\mathcal{A})\cdot \{ ({\wedge}_{i=1}^{i=r}\varphi_i)(\varepsilon_M) : \varphi_i \in \Hom_\mathcal{A}(\ker (\theta_M), \mathcal{A})\}\\ =\xi(\mathcal{A})\cdot \{\nr_A((M' \mid M)) : M' \in {\rm M}_{d',r}(\mathcal{A})\},\end{multline}
and Definition \ref{fit matrix def} implies directly that the latter ideal is equal to ${\rm Fit}_\cA^r((0_{d',r} \mid M))$.

For each $M'$ in ${\rm M}_{d',r}(\mathcal{A})$ we write $\theta_{M',M}$ for the endomorphism of $\mathcal{A}^{d'}$ represented, with respect to the standard basis, by the block matrix $(M' \mid M)$. Then $\nr_A((M' \mid M))$ belongs to ${\rm Fit}_{\mathcal{A}}^0({\rm cok}(\theta_{M',M}))$ (see Definition \ref{def fit mods}) and so Theorem \ref{main fit result}(vi) implies that the final assertion of  (ii) will follow as a consequence of (\ref{first eq claim (ii)}) if there exists a surjective homomorphism of $\mathcal{A}$-modules from ${\rm cok}(\theta_{M',M})$ to ${\rm cok}(\theta_{M})$. The existence of such a homomorphism is in turn a consequence of the commutative diagram of $\mathcal{A}$-modules

\[\begin{CD}
 \mathcal{A}^{d'} @> \theta_{M',M} >> \mathcal{A}^{d'} \\
 @\vert  @V \varrho VV \\
 \mathcal{A}^{d'} @> \theta_M >> \mathcal{A}^{d} \end{CD}\]
in which $\varrho$ is the (surjective) map that sends $b_i$ for each $i$ with $1\le i \le r$, respectively $r<  i\le d'$, to zero, respectively to the $(i-r)$-th element of the standard basis of $\mathcal{A}^{d}$. \end{proof}

\section{Reduced determinant functors}\label{non-comm dets}

The theory of determinant functors for complexes of modules over commutative noetherian rings was established by Knudsen and Mumford in \cite{knumum}, with later clarifications provided by Knudsen in \cite{knudsen}, in both cases following suggestions of Grothendieck.

It was subsequently shown by Deligne in \cite{delignedet} that there exists a universal determinant functor for any exact category, with values in an associated commutative Picard category of `virtual  objects' (cf. Remark \ref{virtual objects rem}).
 For the category of projective modules over certain non-commutative rings, Deligne's construction has played a key role in the formulation of refined `special value conjectures' in arithmetic, such as the equivariant Tamagawa number conjecture from \cite{bf}. The latter conjecture takes the form of an equality in a relative algebraic $K$-group, and an alternative approach to the formulation of conjectures in such groups was later described by Fukaya and Kato in \cite{fukaya-kato} via a theory of `localized  $K_1$-groups'.

In this section we shall use the theory of reduced Rubin lattices to prove the existence of a `reduced determinant functor' on the derived category of bounded complexes of locally-free modules over a non-commutative order. This approach is more explicit than those of Deligne or Fukaya and Kato, but will depend on the same sort of auxiliary data as was fixed in our construction of reduced exterior powers in \S\ref{red ext powers sec}.

In \cite{sbes1-2} this determinant functor plays a key role in the construction of families of non-commutative Euler systems. In addition, in a further article, it will be used to define a natural  non-commutative generalization of the notion of `zeta element' that originates with Kato in \cite{K1} and thereby to shed new light on the content of the equivariant Tamagawa number conjecture relative to non-abelian Galois extensions.

Throughout this section we fix data $R, F, \mathcal{A}$ and $A$ as in \S\ref{app higher fits}.

\subsection{Statement of the main result}

\subsubsection{}For a commutative noetherian ring $\Lambda$ we write $\mathcal{P}(\Lambda)$ for the category of graded invertible $\Lambda$-modules. This is a commutative Picard category (in the sense of \cite{maclane63}) and, for the reader's convenience, we first quickly review its basic properties.

An object of $\mathcal{P}(\Lambda)$ comprises a pair
$(L,\alpha)$ where $L$ is an  invertible $\Lambda$-module and $\alpha$ is a continuous function from ${\rm Spec}(\Lambda)$ to $\ZZ$. (Here we recall that a $\Lambda$-module $L$ is said to be `invertible' if it is finitely generated and for every prime ideal $\wp$ of $\Lambda$ the $\Lambda_{(\wp)}$-module $L_{(\wp)}$ is free of rank one.)

A homomorphism $\theta: (L,\alpha) \to (M,\beta)$ in $\mathcal{P}(\Lambda)$ is a homomorphism of $\Lambda$-modules such that $\theta_{(\wp)} = 0$ whenever
$\alpha(\wp) \not= \beta(\wp)$.

The tensor product of two objects $ (L,\alpha)$ and $(M,\beta)$ in $\mathcal{P}(\Lambda)$
is given by
\[ (L,\alpha)\otimes (M,\beta) = (L\otimes_\Lambda M,\alpha + \beta)\]
and for each such pair there is an isomorphism
\[ \psi_{(L,\alpha),(M,\beta)}: (L,\alpha)\otimes (M,\beta) \xrightarrow{\sim}
  (M,\beta)\otimes (L,\alpha)\]
in $\mathcal{P}(\Lambda)$ such that for every $\wp$ and every $\ell$ in $L_{(\wp)}$ and $m$ in $M_{(\wp)}$ one has
\[ \psi_{(L,\alpha),(M,\beta)}(\ell\otimes m) = (-1)^{\alpha(\wp)\cdot
 \beta(\wp)} \cdot (m\otimes \ell).\]

 The unit object ${\bf 1}_{\mathcal{P}(\Lambda)}$ is the pair $(\Lambda,0)$ and the natural `evaluation map' isomorphism $L\otimes_\Lambda \Hom_\Lambda(L,\Lambda) \cong \Lambda$
 induces an isomorphism in $\mathcal{P}(\Lambda)$
\[ (L,\alpha)\otimes (\Hom_\Lambda(L,\Lambda),-\alpha) \cong {\bf 1}_{\mathcal{P}(\Lambda)}.\]

This isomorphism is used to regard $(\Hom_\Lambda(L,\Lambda),-\alpha)$ as a right inverse to $(L,\alpha)$ and it is then also regarded as a left inverse by means of the isomorphism $\psi_{(\Hom_\Lambda(L,\Lambda),-\alpha),(L,\alpha)}$.

%With respect to these constructions, $\mathcal{P}(R)$ forms a Picard category.

%\subsubsection{}
%\begin{lemma} \end{lemma}
%\begin{proof} \end{proof}

\subsubsection{}\label{ordered basis sec}In the sequel we write the Wedderburn decomposition of $A$ as 
\[ A = {\prod}_{i \in I}A_i\]
so that each algebra $A_i$ is of the form ${\rm M}_{n_i}(D_i)$ for a division ring $D_i$ (so that $\zeta(A_i) = \zeta(D_i)$).

For each index $i$ we choose a splitting field $E_i$ for $D_i$ so that $D_i\otimes_{\zeta(A_i)}E_i %$D_i\otimes_{\zeta(D_i)}E_i 
= {\rm M}_{m_i}(E_i)$. We then fix an indecomposable idempotent $f_i$ of ${\rm M}_{m_i}(E_i)$ and an $E_i$-basis $\{w_a\}_{1\le a\le m_i}$ of the left ideal 
$W_i$ of ${\rm M}_{m_i}(E_i)$ that is generated by $f_i$.
%$W_i := f_i\cdot {\rm M}_{m_i}(E_i)$. 
(When making such a choice we always follow the convention of Remark \ref{comm conventions} on each simple component $A_i$ that is commutative.)

Then the direct sum $V_i := W_i^{n_i}$ of $n_i$-copies of $W_i$ is a simple $A_i\otimes_{\zeta(A_i)}E_i$-module and has as an $E_i$-basis  the set $\underline{\varpi}_{i} = \{\varpi_{aj}\}_{1\le a\le n_i, 1 \le j \le m_i}$ where $\varpi_{aj}$ denotes the element of $V_i$ that is equal to $w_j$ in its $a$-th component and is zero in all other components.

We order each set $\underline{\varpi}_{i}$ lexicographically and will apply the constructions of \S\ref{red ext powers sec} with respect to the collection of ordered bases 
\[  \varpi := \{ \underline{\varpi}_{i}: i \in I\}.\]

The following straightforward observation will also be useful.

\begin{lemma}\label{rr=gi} Let $\mathcal{R}$ be an $R$-order in $\zeta(A)$. Then the reduced rank ${\rm rr}_A(Z)$ of a finitely generated $A$-module $Z$ determines a locally-constant function on ${\rm Spec}(\mathcal{R})$.\end{lemma}

\begin{proof} The maximal $R$-order $\mathcal{M}$ in $\zeta(A)$ is  ${\prod}_{i \in I}\mathcal{O}_i$, where $\mathcal{O}_i$ is the integral closure of $R$ in the field $\zeta(A_i)$.

Since the inclusion $\mathcal{R}\to \mathcal{M}$ is an integral ring extension, the going-up theorem implies that every prime ideal of $\mathcal{R}$ has the form $\wp = \mathcal{R}\cap \wp'$ for a prime ideal $\wp'$ of $\mathcal{M}$, and then $\mathcal{R}_{(\wp)}$ is a finite index subgroup of $\mathcal{M}_{(\wp')}$. The key point now is that there exists a unique index $i(\wp)$ in $I$ such that the kernel of the projection $\mathcal{M} \to \mathcal{O}_{i(\wp)}$ is contained in $\wp'$ and hence that $(\mathcal{R}_{(\wp)})_F = (\mathcal{M}_{(\wp')})_F$ identifies with $A_{i(\wp)}$.

One then obtains a well-defined locally-constant function on ${\rm Spec}(\mathcal{R})$ by sending each $\wp$ to the $i(\wp)$-th component ${\rm rr}_{A_{i(\wp)}}(A_{i(\wp)}\otimes_A Z)$ of ${\rm rr}_A(Z)$.\end{proof}

\subsubsection{} The category ${\rm Mod}^{\rm lf}(\mathcal{A})$ of finitely generated, locally-free $\mathcal{A}$-modules (as discussed in \S\ref{enoch section}) is a full additive subcategory of the abelian category of $\mathcal{A}$-modules. In addition, if $M'$ and $M''$ belong to ${\rm Mod}^{\rm lf}(\mathcal{A})$, then any short exact sequence of $\mathcal{A}$-modules of the form 
\[ 0 \to M' \to M \to M'' \to 0\]
is split (since $M''$ is a projective $\mathcal{A}$-module), and so $M$ also belongs to ${\rm Mod}^{\rm lf}(\mathcal{A})$.  

These observations imply that  ${\rm Mod}^{\rm lf}(\mathcal{A})$ is an exact category in the sense of Quillen \cite[p. 91]{quillen} and, for each non-negative integer $i$, we denote the associated algebraic $K$-group in degree $i$ by 
\[ K_i^{\rm lf}(\mathcal{A}) := K_i({\rm Mod}^{\rm lf}(\mathcal{A})).\] 

\subsubsection{}\label{lf,0 def}For any noetherian ring $\Lambda$ we write $\Der(\Lambda)$ for the derived category of (left) $\Lambda$-modules and $\Der^{\rm perf}(\Lambda)$ for the full triangulated subcategory of $\Der(\Lambda)$ comprising complexes that are `perfect' (that is, isomorphic in $\Der(\Lambda)$ to a bounded complex of finitely generated projective 
$\Lambda$-modules). 

We also write $\DerC^{\rm lf}(\mathcal{A})$ for the category of bounded complexes of objects of ${\rm Mod}^{\rm lf}(\mathcal{A})$, and $\Der^{{\rm lf}}(\mathcal{A})$ for the full triangulated subcategory of $\Der^{\rm perf}(\mathcal{A})$ comprising complexes that are isomorphic (in $\Der(\mathcal{A})$) to an object of $\DerC^{\rm lf}(\mathcal{A})$. %We then write $D^{\rm lf}(\mathcal{A})_{\rm is}$ for the subcategory of $D^{\rm lf}(\mathcal{A})$ in which morphisms are restricted to be isomorphisms.

For each object $M$ of ${\rm Mod}^{\rm lf}(\mathcal{A})$ we write $[M]$ for the associated element of the Grothendieck group $K_0^{\rm lf}(\mathcal{A})$. We  note that if $C$ belongs to $\Der^{{\rm lf}}(\mathcal{A})$ and $P^\bullet$ is any object of $\DerC^{\rm lf}(\mathcal{A})$ that is isomorphic in $\Der(\mathcal{A})$ to $C$, then the `Euler characteristic' element  
\begin{equation}\label{ec def} \chi_\mathcal{A}(C) := {\sum}_{i\in \mathbb{Z}} (-1)^i[P^i] \in K^{\rm lf}_0(\mathcal{A})\end{equation}
can be checked to independent of the choice of $P^\bullet$. 

We define the `reduced locally-free classgroup' of $\mathcal{A}$ by setting 
\[ {\rm SK}_0^{\rm lf}(\mathcal{A}) := \ker\bigl( K^{\rm lf}_0(\mathcal{A}) \to \ZZ\bigr),\]
where the arrow denotes the homomorphism that is induced by sending each object $M$ of ${\rm Mod}^{\rm lf}(\mathcal{A})$ to ${\rm rk}_\mathcal{A}(M)$. 

We then write $\DerC^{{\rm lf},0}(\mathcal{A})$ for the subcategory of $\DerC^{\rm lf}(\mathcal{A})$ comprising complexes $P^\bullet$ for which $\chi_\mathcal{A}(P^\bullet)$ belongs to ${\rm SK}_0^{\rm lf}(\mathcal{A})$ and $\Der^{{\rm lf},0}(\mathcal{A})$ for the full triangulated subcategory of $\Der^{\rm lf}(\mathcal{A})$ comprising complexes $C$ for which $\chi_\mathcal{A}(C)$ belongs to ${\rm SK}_0^{\rm lf}(\mathcal{A})$. (The latter condition is equivalent to requiring that $C$ be isomorphic in $\Der(\mathcal{A})$ to an object of $\DerC^{{\rm lf},0}(\mathcal{A})$.)

In the next result we show that, under certain natural hypotheses, the category $\Der^{{\rm lf},0}(\mathcal{A})$ has a more direct interpretation. 

To state this result we recall that, for any Dedekind domain $\Lambda$ with field of fractions $E$, and any $\Lambda$-order $\mathcal{B}$, the `reduced projective class group' ${\rm SK}_0(\mathcal{B})$ of $\mathcal{B}$ is defined to be the kernel of the natural scalar extension homomorphism $K_0(\mathcal{B}) \to K_0(E\otimes_\Lambda \mathcal{B})$. (For more details of such groups see, for example, \cite[Rem. (49.11)]{curtisr}). 

We also note that, for any of the general classes of order $\mathcal{A}$ discussed in Example \ref{loc free exam}, and for every prime ideal $\mathfrak{p}$ of $R$, the group 
${\rm SK}_0(\mathcal{A}_{(\mathfrak{p})})$ can be checked to vanish. 

%To state the result we write $D^{{\rm perf},0}(\mathcal{A})$ for the full triangulated subcategory of $D^{\rm perf}(\mathcal{A})$ comprising complexes whose Euler characteristic in $K_0(\mathcal{A})$ belongs to the subgroup %${\rm SK}_0(\mathcal{A})$.  

\begin{lemma}\label{sk0 vanishes lemma} If ${\rm SK}_0(\mathcal{A}_{(\mathfrak{p})})$ vanishes for every prime ideal $\mathfrak{p}$ of $R$, then $\Der^{{\rm lf},0}(\mathcal{A})$ is naturally equivalent to the full triangulated subcategory $\Der^{{\rm perf},0}(\mathcal{A})$ of $\Der^{\rm perf}(\mathcal{A})$ comprising complexes whose Euler characteristic in $K_0(\mathcal{A})$ belongs to ${\rm SK}_0(\mathcal{A})$. 
\end{lemma} 

\begin{proof} We write $\chi^{\rm proj}_\mathcal{A}(C)$ for the Euler characteristic in $K_0(\mathcal{A})$ associated to an object $C$ of $\Der^{\rm perf}(\mathcal{A})$ (this is defined via a choice of resolution of $C$, just as in (\ref{ec def})). We then also write $\DerC^{{\rm p},0}(\mathcal{A})$ for the category of bounded complexes of finitely generated projective $\mathcal{A}$-modules $Q^\bullet$ for which $\chi^{\rm proj}_\mathcal{A}(Q^\bullet)$ belongs to ${\rm SK}_0(\mathcal{A})$. 

It is clear that $\DerC^{{\rm lf},0}(\mathcal{A})$ identifies with a subcategory of $\DerC^{{\rm p},0}(\mathcal{A})$ and also that $\Der^{{\rm perf},0}(\mathcal{A})$ coincides with the subcategory of $\Der(\mathcal{A})$ comprising complexes that are isomorphic  to an object of $\DerC^{{\rm p},0}(\mathcal{A})$. To prove the stated claim it is therefore enough to show, under the stated hypothesis, that every complex in $\DerC^{{\rm p},0}(\mathcal{A})$ is isomorphic (in $\Der(\mathcal{A})$) to a complex in $\DerC^{{\rm lf},0}(\mathcal{A})$. 

To do this we fix a complex $Q^\bullet$ in $\DerC^{{\rm p},0}(\mathcal{A})$ and then use a standard construction in homological algebra (as, for example, in \cite[Rapport, Lem. 4.7]{del}) to fix a quasi-isomorphism of complexes of $\mathcal{A}$-modules $\theta: P^\bullet \to Q^\bullet$ in which $P^\bullet$ is a bounded complex of finitely generated $\mathcal{A}$-modules in which all terms, except possibly the first non-zero term $P^a$, are free. It is then enough for us to show that for every $\mathfrak{p}$ the $\mathcal{A}_{(\mathfrak{p})}$-module $P_{(\mathfrak{p})}^a$ is free. 

Now, the quasi-isomorphism $\theta$ implies that 
\[ \chi^{\rm proj}_\mathcal{A}(P^\bullet) = 
\chi^{\rm proj}_\mathcal{A}(Q^\bullet)\in {\rm SK}_0(\mathcal{A})\]
and so, for every $\mathfrak{p}$, the Euler characteristic in $K_0(A)$ of the complex $F\otimes_{\mathcal{A}_{(\mathfrak{p})}}P^\bullet_{(p)}$ vanishes. Since ${\rm SK}_0(\mathcal{A}_{(\mathfrak{p})})$ is assumed to vanish, the Euler characteristic of $P^\bullet_{(\mathfrak{p})}$ in $K_0(\mathcal{A}_{(\mathfrak{p})})$ is therefore also zero, and so (by \cite[Prop. (38.22)]{curtisr}) there exists an isomorphism of 
$\mathcal{A}_{(\mathfrak{p})}$-modules 
\[ P' \oplus P_{(\mathfrak{p})}^a \oplus {\bigoplus}_{b}P_{(\mathfrak{p})}^b \cong P' \oplus {\bigoplus}_{c}P_{(\mathfrak{p})}^c\]
in which $P'$ is finitely generated and projective and, in the direct sums, $b$ and $c$ respectively run over all integers with $b \equiv a$ (mod $2$) and $c \not\equiv a$ (mod $2$). In particular, since each of the $\mathcal{A}_{(\mathfrak{p})}$-modules $P_{(\mathfrak{p})}^b$ and $P_{(\mathfrak{p})}^c$ are free, the $\mathfrak{p}$-completion of the displayed isomorphism combines with the Krull-Schmidt-Azumaya theorem (in the form of \cite[Cor. (6.15)]{curtisr} with $R = R_\mathfrak{p}$ and $A = \mathcal{A}_{\mathfrak{p}})$ to imply $P_{\mathfrak{p}}^a$ is a free $\mathcal{A}_{\mathfrak{p}}$-module. It therefore follows (from Maranda's Theorem) that $P_{(\mathfrak{p})}^a$ is a free $\mathcal{A}_{(\mathfrak{p})}$-module, as required. \end{proof} 
%
%Next we note that for any  $\mathcal{A}$-module $M$ one has %Since $R_\mathfrak{p}/\mathfrak{p}R_\mathfrak{p} = R/\mathfrak{p} = R_{(\mathfrak{p})}/(\mathfrak{p}R_{(\mathfrak{p})})$, for  any
%
%\[ M_{(\mathfrak{p})}/\mathfrak{p}M_{(\mathfrak{p})} = \bigl(R_{(\mathfrak{p})}/(\mathfrak{p}R_{(\mathfrak{p})})\bigr) \otimes_R M =  \bigl(R_{\mathfrak{p}}/(\mathfrak{p}R_{\mathfrak{p}})\bigr)\otimes_R M =  
%M_\mathfrak{p}/\mathfrak{p}M_{\mathfrak{p}}.\]
%
%Thus, setting $\mathcal{B} := \mathcal{A}_{(\mathfrak{p})}$ and $X := P_{(\mathfrak{p})}^a$, the above argument implies that %$X/(\mathfrak{p}X)$ is a free 
%$\mathcal{B}/(\mathfrak{p}\mathcal{B})$-module. Since $\mathfrak{p}\mathcal{B}$ is contained in ${\rm rad}(\mathcal{B})$ (by %\cite[Prop. (5.22)(i)]{curtisr}), Nakayama's Lemma (in the form of \cite[Prop. (5.7)]{curtisr} with $R = R_{(\mathfrak{p})}$ %and $A = \mathcal{B}$) therefore implies that the $\mathcal{B}$-module $X$ is free, as required. 

\begin{remark}\label{sk0 vanishes rem} {\em Let $C$ be an object of $\Der^{\rm perf}(\mathcal{A})$ whose Euler characteristic in $K_0(\mathcal{A})$ vanishes. Then, without any hypothesis on reduced projective class groups, the argument of Lemma \ref{sk0 vanishes lemma} shows that $C$ belongs to the subcategory $\Der^{{\rm lf},0}(\mathcal{A})$ of $\Der^{\rm perf}(\mathcal{A})$.}\end{remark}

\subsubsection{}We can now state the main result of \S\ref{non-comm dets}. 

In this result we write $\Der^{\rm lf}(\mathcal{A})_{\rm is}$ for the subcategory of $\Der^{\rm lf}(\mathcal{A})$ in which morphisms are restricted to be isomorphisms and we use the concept of `extended determinant functor' that is made precise in Definition \ref{ext det functor def} below.

\begin{theorem}\label{ext det fun thm} For each set of ordered bases $\varpi$ as in \S\ref{ordered basis sec}, there exists a canonical extended determinant functor
\begin{equation*} {\rm d}_{\mathcal{A},\varpi}: \Der^{\rm lf}(\mathcal{A})_{\rm is} \to \mathcal{P}(\xi(\mathcal{A}))\end{equation*}
that has all of the following properties.

\begin{itemize}
\item[(i)] For each exact triangle
\[ C' \xrightarrow{u} C \xrightarrow{v} C'' \xrightarrow{w} C'[1]\]
in $\Der^{\rm lf}(\mathcal{A})$ there exists a canonical isomorphism
\[ {\rm d}_{\mathcal{A},\varpi}(C') \otimes {\rm d}_{\mathcal{A},\varpi}(C'') \xrightarrow{\sim} {\rm d}_{\mathcal{A},\varpi}(C)\]
in $\mathcal{P}(\xi(\mathcal{A}))$ that is functorial with respect to isomorphisms of triangles.
\item[(ii)] Let $\varrho: \mathcal{A} \to \mathcal{B}$ be a surjective homomorphism of $R$-orders and $\varrho_\ast:\xi(\mathcal{A}) \to \xi(\mathcal{B})$ the induced homomorphism. Then for each $C$ in $\Der^{\rm lf}(\mathcal{A})$ the complex $\mathcal{B}\otimes^{\mathbb{L}}_{\mathcal{A},\varrho}C$ belongs to $\Der^{\rm lf}(\mathcal{B})$ and there exists a canonical isomorphism
    \[ \xi(\mathcal{B})\otimes_{\xi(\mathcal{A}),\varrho_\ast}{\rm d}_{\mathcal{A},\varpi}(C) \cong {\rm d}_{\mathcal{B},\varpi'}(\mathcal{B}\otimes^{\mathbb{L}}_{\mathcal{A},\varrho}C)\]
in $\mathcal{P}(\xi(\mathcal{B}))$, where $\varpi'$ is the ordered subset of $\varpi$ that corresponds to all simple modules $V_i$ that factor through the scalar extension of $\varrho$.
\item[(iii)] If $P$ belongs to ${\rm Mod}^{\rm lf}(\mathcal{A})$, then ${\bigcap}_{\mathcal{A}}^{{\rm rk}_{\mathcal{A}}(P)}P$ is an invertible $\xi(\mathcal{A})$-module and one has
\[ {\rm d}_{\mathcal{A},\varpi}(P[0]) = \bigl({\bigcap}_{\mathcal{A}}^{{\rm rk}_{\mathcal{A}}(P)}P,\, 
{\rm rr}_{A}(P_F)\bigr).\]
Here the reduced Rubin lattice ${\bigcap}_{\mathcal{A}}^{{\rm rk}_{\mathcal{A}}(P)}P$ is defined with respect to $\varpi$ and ${\rm rr}_{A}(P_F)$ is regarded as a locally-constant function on ${\rm Spec}(\xi(\mathcal{A}))$ via Lemma \ref{rr=gi}.
\item[(iv)] The restriction of ${\rm d}_{\mathcal{A},\varpi}$ to $\Der^{{\rm lf},0}(\mathcal{A})_{\rm is}$ is independent of the choice of $\varpi$.
\end{itemize}
\end{theorem}

\begin{remark}\label{comparison of categories} {\em The approach of Deligne in \cite[\S4]{delignedet} constructs a `universal determinant functor' for the exact category ${\rm Mod}^{\rm lf}(\mathcal{A})$, with values in an associated commutative Picard category $V^{\rm lf}(\mathcal{A})$ of `virtual objects' (for more details see Remark \ref{virtual objects rem} below). In particular, in this way each determinant functor $ {\rm d}_{\mathcal{A},\varpi}$ constructed as in Theorem \ref{ext det fun thm} naturally induces a functor 
\[ \phi^{\rm lf}_{\mathcal{A},\varpi}: V^{\rm lf}(\mathcal{A}) \to \mathcal{P}(\xi(\mathcal{A}))\]
that is strongly monoidal (in the sense defined in \cite[Ch.
XI.2]{maclane}). For a
commutative Picard category $\mathcal{P}$ we write $\pi_0(\mathcal{P})$ for the (abelian) group of
isomorphism classes of its objects, with group structure induced by
the given bifunctor `product' of $\mathcal{P}$, and $\pi_1(\mathcal{P})$ for the (abelian) group of automorphisms of the unit object of $\mathcal{P}$. Then, by a general result on Picard categories, the functor $\phi^{\rm lf}_{\mathcal{A},\varpi}$ gives an equivalence of commutative Picard categories if and only if, for both $i = 0$ and $i=1$, the induced homomorphism of groups 
\[ \pi_i(\phi^{\rm lf}_{\mathcal{A},\varpi}): \pi_i(V^{\rm lf}(\mathcal{A})) \to \pi_i(\mathcal{P}(\xi(\mathcal{A})))\]
is bijective. In addition, the topological model of $ V^{\rm lf}(\mathcal{A})$ that is constructed in \cite[\S4.2-4.5]{delignedet} implies the existence of canonical isomorphisms $\pi_i(V^{\rm lf}(\mathcal{A})) \cong K_i^{\rm lf}(\mathcal{A})$ for both $i = 0$ and $i = 1$. It is also clear that there are canonical isomorphisms 
\[ \pi_i(\mathcal{P}(\xi(\mathcal{A}))) \cong \begin{cases} {\rm Pic}(\xi(\mathcal{A}))\times H^0({\rm Spec}(\xi(\mathcal{A})),\ZZ), &\text{if $i=0$,}\\
\xi(\mathcal{A})^\times, &\text{if $i=1$},\end{cases}
\]
where ${\rm Pic}(\xi(\mathcal{A}))$ denotes the Picard group of the commutative ring $\xi(\mathcal{A})$. These respective identifications can be used to show that, in most cases, the map $ \pi_0(\phi^{\rm lf}_{\mathcal{A},\varpi})$ is neither injective nor surjective and that, if $\mathcal{A}$ is not commutative, then the same is true of the map $ \pi_1(\phi^{\rm lf}_{\mathcal{A},\varpi})$ (which coincides with the composite of the natural scalar extension map $K_1^{\rm lf}(\mathcal{A}) \to K_1(A)$ and the reduced norm map $K_1(A) \to \zeta(A)^\times$ of the semisimple $F$-algebra $A$). In particular, one knows that, in most cases,  % (for example, the map $ \pi_1(\phi^{\rm lf}_{\mathcal{A},\varpi})$ can be computed via the composite is induced by the reduced norm of the semisimple algebra $A$) 
the functor $\phi^{\rm lf}_{\mathcal{A},\varpi}$ is not an equivalence of commutative Picard categories. }%This issue is considered further in \cite{bdms}. }
\end{remark} 

\begin{remark}\label{ext det fun rem}{\em Let $\mathcal{R}$ be any $R$-order in $\zeta(A)$ with the property that for each prime ideal $\mathfrak{p}$ of $R$ the localization $\mathcal{R}_{(\mathfrak{p})}$ contains the reduced norms of all matrices in $\bigcup_{n \ge 1}{\rm GL}_n(\mathcal{A}_{(\mathfrak{p})})$. (Note that $\xi(\mathcal{A})$ is, by its very definition, an example of such an order $\mathcal{R}$ but that there can also, in principle, exist such orders that are strictly contained in  $\xi(\mathcal{A})$.) Then a closer analysis of our proof of Theorem \ref{ext det fun thm} will show that there exists a determinant functor
\[ {\rm d}_{\mathcal{A},\varpi,\mathcal{R}}: \Der^{\rm lf}(\mathcal{A})_{\rm is} \to \mathcal{P}(\mathcal{R})\]
that satisfies analogues of all of the properties of ${\rm d}_{\mathcal{A},\varpi}$ listed above and is, in addition, such that if $\mathcal{R}'$ is any order in $\zeta(A)$ that contains $\mathcal{R}$, then
 one has
\[ {\rm d}_{\mathcal{A},\varpi,\mathcal{R}'} = \iota_{\mathcal{R}',\mathcal{R}}\circ {\rm d}_{\mathcal{A},\varpi,\mathcal{R}}\]
where $\iota_{\mathcal{R}',\mathcal{R}}$ is the natural scalar extension functor $\mathcal{P}(\mathcal{R}) \to \mathcal{P}(\mathcal{R}')$. However, we shall make no use of this additional generality in the sequel and so, for simplicity, only consider  $\xi(\mathcal{A})$. }\end{remark}

The proof of Theorem \ref{ext det fun thm} will occupy the remainder of \S\ref{non-comm dets}. Our basic approach is to adapt an argument that is used by Flach and the first author in \cite[\S2]{bf} and so is closely modelled on the original constructions that are made by Knudsen and Mumford in \cite{knumum}.

\subsection{Determinant functors}\label{det functor prel section}

\subsubsection{}Let $\ce$
be an exact category and write $\ce_{\rm is}$ for the
subcategory of $\ce$ in which morphisms are restricted to isomorphisms.  
%By an `admissible subobject' $E'$ of an object $E$ in $\ce$ one means an object that lies in an exact sequence in $\ce$ of the form $E' \to E\to E'' .$ We note in particular that, since $\ce$ is exact, if $E'$ is an admissible subobject of $E$ then every admissible subobject of $E'$ is itself an admissible subobject of $E$.
Then the following definition is equivalent to that given in \cite[\S2.3]{bf}.

\begin{definition}\label{det functor def}{\em
A {\em determinant functor} on $\ce$ is a Picard category $\mathcal{P}$, with unit object ${\bf 1}_{\mathcal{P}}$ and product $\boxtimes$,
together with the following data.
\smallskip

\begin{itemize}
\item[(a)] A covariant functor ${\rm d} : \ce_{\rm is}\rightarrow \mathcal{P}$.
\item[(b)] For each short exact sequence $E'\xrightarrow{\alpha}E\xrightarrow{\beta}
E''$ in $\ce$ a morphism
\[
{\rm i}(\alpha,\beta):\vir{E}\xrightarrow{\sim}\vir{E'}\boxtimes\vir{E''}\]
in $\mathcal{P}$ that is functorial for isomorphisms of short exact sequences.
\item[(c)] For each zero object $0$ in $\ce$ an isomorphism
\begin{equation*}
\zeta(0):\vir{0} \xrightarrow{\sim}\eins_{\mathcal{P}}.
\end{equation*}
\end{itemize}
These data are subject to the following axioms.
\begin{itemize}
\item[(d)] For each isomorphism $\phi:E\rightarrow E'$ in $\ce$, the induced exact sequences
\[ 0\rightarrow E\xrightarrow{\phi} E'\,\, \text{ and }\,\, E\xrightarrow{\phi} E'\rightarrow 0\]
are such that $\vir{\phi}$ and $\vir{\phi^{-1}}$ respectively coincide with the composite maps
\[ \vir{E} \xrightarrow{{\rm i}(0,\phi)}
\vir{0}\boxtimes\vir{E'}\xrightarrow{\zeta(0)\boxtimes\id}
\vir{E'}\]
and
\[ \vir{E'} \xrightarrow{{\rm i}(\phi,0)}
\vir{E}\boxtimes\vir{0}\xrightarrow{\id\boxtimes\zeta(0)}
\vir{E}.\]
\item[(e)] Given a commutative diagram of objects in $\mathcal{E}$
\[\begin{CD}
 E_1' @> \alpha'>> E_2' @> \beta' >> E_3'\\
@V\gamma' VV @V\gamma VV @V \gamma'' VV \\
E_1 @> \alpha >> E_2 @> \beta >> E_3\\
@V\delta' VV @V\delta VV @V \delta''VV\\
E_1'' @> \alpha'' >> E_2'' @> \beta'' >> E_3''
\end{CD}\]
in which each row and column is a short exact sequence, the diagram
\begin{equation*}
\hskip-0.4truein
\begin{CD}
\vir{E_2} @> {\rm i}(\gamma,\delta)>> \vir{E_2'}\boxtimes \vir{E_2''}\\
@V {\rm i}(\alpha,\beta)VV @VV (1\boxtimes \psi_{{\rm d}(E_3'),{\rm d}(E_1'')}\boxtimes 1)\cdot ({\rm i}(\alpha',\beta')\boxtimes {\rm i}(\alpha'',\beta'')) V\\
\vir{E_1}\boxtimes \vir{E_3}
@> {\rm i}(\gamma',\delta')\boxtimes {\rm i}(\gamma'',\delta'')>> \vir{E_1'}\boxtimes \vir{E_1''}\boxtimes \vir{E_3'}\boxtimes \vir{E_3''}
\end{CD}
\end{equation*}

%\[\begin{CD}
 %\vir{E_1'}\boxtimes \vir{E_1''}\boxtimes \vir{E_3'}\boxtimes \vir{E_3''} @>\vir{\gamma',\delta'}\boxtimes \vir{\gamma'',\delta''} >> &  \vir{E_1}\boxtimes \vir{E_3} \\
 %@VV\vir{\alpha',\beta'}\boxtimes \vir{\alpha'',\beta''}\cdot (1 \boxtimes \psi_{\vir{E_1''},\vir{E_3'}} \boxtimes 1)V & @VV\vir{\alpha,\beta}V\\
 %\vir{E_2'}\boxtimes \vir{E_2''} @> \vir{\gamma,\delta} >> & \vir{E_2}
 %\end{CD}\]
 %\vir{E'}\boxtimes\vir{E/E'} @>>> \vir{E''}\boxtimes\vir{E'/E''}\boxtimes
%\vir{E/E'}
%For admissible subobjects $E'$ of $E$ and $E''$ of $E'$ in  $\ce$ the diagram
%\[\begin{CD} \vir{E} @>>>\vir{E''}\boxtimes\vir{E/E''}\\ @VVV @VVV\\
%\vir{E'}\boxtimes\vir{E/E'} @>>> \vir{E''}\boxtimes\vir{E'/E''}\boxtimes
%\vir{E/E'}
%\end{CD}\]
commutes.
\end{itemize}
}
\end{definition}
%\smallskip

\begin{remark}\label{virtual objects rem}{\em The terminology of `determinant functor' used above is borrowed from the key example in which
 $\ce$ is the category of vector bundles on a scheme, $\mathcal{P}$ is the
category of line bundles and the functor is taking the highest
exterior power. However, as was shown by Deligne in \cite[\S4]{delignedet}, there exists a universal
determinant functor for any given exact category $\ce$. More
precisely, there exists a commutative Picard category $V(\ce)$, called the
`category of virtual objects' of $\ce$, together with data (a)-(c)
 which in addition to (d) and (e) also satisfies the following
universal property.
\begin{itemize}
\item[(f)] For any Picard category $\mathcal{P}$ the category of strongly monoidal functors $\Hom^{\boxtimes}(V(\ce),\mathcal{P})$ is
naturally equivalent to the category of determinant functors $\ce_{\rm is}\rightarrow \mathcal{P}$.
\end{itemize}
Although comparatively inexplicit, this construction has played a key role in the formulation of special value conjectures for motives with non-commutative coefficients.}
\end{remark}%We recall that the category $V(\ce)$ has a commutativity
%constraint defined as follows. Let
%\[\tau_{E',E''}:\vir{E'}\boxtimes\vir{E''}\xleftarrow{\vir{\Sigma_1}}
%\vir{E'\oplus
%E''}\xrightarrow{\vir{\Sigma_2}}
%\vir{E''}\boxtimes\vir{E'}\]
%be the isomorphism induced by the short exact sequences
%\[\Sigma_1: 0\rightarrow E'\rightarrow E'\oplus E''\rightarrow
%E''\rightarrow 0\]
%\[\Sigma_2: 0\rightarrow E''\rightarrow E'\oplus E''\rightarrow
%E'\rightarrow 0.\]
%Replacing $\vir{\Sigma}$ by $\tau_{E',E''}\circ\vir{\Sigma}$
%yields a datum a), b), c) with values in
%$V(\ce)^{\boxtimes-op}$, the Picard category with product
%$(L,M)\mapsto M\boxtimes L$, and satisfying d),e). By the
%universal property f) of $V(\ce)$, this corresponds to a
%monoidal functor $F:V(\ce)\rightarrow V(\ce)^{\boxtimes-op}$.
%Since we have only changed the value of $\vir{\,\,}$
%on short exact sequences, $F$ is the identity on objects and
%morphisms and so the monoidality of $F$ gives a commutativity
%constraint on $V(\ce)$.

Recalling that the category ${\rm Mod}^{\rm lf}(\mathcal{A})$ is exact, our aim in the remainder of \S\ref{det functor prel section} will be to construct, for each set of ordered bases $\varpi$ as in \S\ref{ordered basis sec}, a canonical determinant functor
\begin{equation*}\label{det functor claim} {\rm d}^\diamond_{\mathcal{A},\varpi}: {\rm Mod}^{\rm lf}(\mathcal{A})_{\rm is} \to \mathcal{P}(\xi(\mathcal{A})).\end{equation*}

\subsubsection{}We start by establishing several useful technical properties of the reduced Rubin lattices of modules in ${\rm Mod}^{\rm lf}(\mathcal{A})$.

In the following result we assume that all reduced exterior power constructions are made with respect to the fixed bases $\varpi$ but will not usually indicate this dependence explicitly.

%At the outset we assume to be a given a finitely generated locally-free $\mathcal{A}$-module $P$ and set $r := {\rm rk}_\mathcal{A}(P)$ and $P_F := F\otimes_R P$.
%
%We fix an $A$-basis $\underline{b}_{0} = \{b_{0,j}\}_{1\le j\le r}$ of $P_F$ and define a free rank one $\zeta(A)$-submodule of ${\bigwedge}^r_{A}P_F$ by setting
%
%\[ \bigl({\bigwedge}^r_{A}P_F\bigr)_0 := \zeta(A)\cdot\wedge_{j=1}^{j=r}b_{0,j},\]
%
%where all exterior powers are defined with respect to the given set of ordered bases $\varpi$.
%
%Then Corollary \ref{proprn2} implies this module is independent of the choice of basis $\underline{b}_{0}$. In particular, if for any prime ideal $\frp$ of $R$ one fixes an $\mathcal{A}_{(\mathfrak{p})}$-basis $\underline{b}_{\frp} = \{b_{\frp,j}\}_{1\le j\le r}$ of the localization $P_{(\mathfrak{p})}$, then $\underline{b}_{\frp}$ is also an $A$-basis of $P_F$ and so one obtains a free rank one $\xi(\mathcal{A}_{(\mathfrak{p})})$-submodule of $\bigl({\bigwedge}^r_AP_F\bigr)_0$ by setting
%
%\[ {\bigwedge}_{\mathcal{A}_{(\mathfrak{p})}}^r P_{({\mathfrak{p}})} := \xi(\mathcal{A}_{(\mathfrak{p})})\cdot \wedge_{j=1}^{j=r}b_{\frp,j}.\]

%We then obtain a $\xi(\mathcal{A})$-submodule of $\bigl({\bigwedge}^r_AP_F\bigr)_0$ by setting
%
%\[ {\bigwedge}_{\mathcal{A}}^r P := \bigcap_\frp{\bigwedge}_{\mathcal{A}_{(\mathfrak{p})}}^r P_{({\mathfrak{p}})}\]
%
%where the intersection is taken over all prime ideals $\frp$ of $R$.

\begin{proposition}\label{first independence check} Fix an object $P$ of ${\rm Mod}^{\rm lf}(\mathcal{A})$ and set $r := {\rm rk}_\mathcal{A}(P)$. Then the following claims are valid.
\begin{itemize}
\item[(i)] If $P$ is a free $\mathcal{A}$-module, with basis $\{b_{j}\}_{1\le j\le r}$, then ${\bigcap}_{\mathcal{A}}^r P$ is a free rank one $\xi(\mathcal{A})$-module with basis $\wedge_{j=1}^{j=r}b_{j}$.
\item[(ii)] For each prime ideal $\mathfrak{p}$ of $R$ fix an  $\mathcal{A}_{(\mathfrak{p})}$-basis $\{b_{\frp,j}\}_{1\le j\le r}$ of $P_{(\mathfrak{p})}$. Then the $\xi(\mathcal{A})_{(\mathfrak{p})}$-module $\left({\bigcap}_{\mathcal{A}}^r P\right)_{(\mathfrak{p})}$ is free of rank one, with basis $\wedge_{j=1}^{j=r}b_{\frp,j}$. Hence one has
\[ {\bigcap}_{\mathcal{A}}^r P = {\bigcap}_{\frp\in {\rm Spec}(R)} \bigl(\xi(\mathcal{A})_{(\mathfrak{p})}\cdot \wedge_{j=1}^{j=r}b_{\frp,j}\bigr).\]
%
%where the intersection runs over all $\mathfrak{p}$ and takes place in $\left({\bigcap}_{\mathcal{A}}^r P\right)_F$.
%
\item[(iii)] ${\bigcap}_{\mathcal{A}}^r P$ is an invertible $\xi(\mathcal{A})$-module. %, with $({\bigwedge}_{\mathcal{A}}^r P)_{(\mathfrak{p})} = {\bigwedge}_{\mathcal{A}_{(\mathfrak{p})}}^r P_{({\mathfrak{p}})}$ for all prime ideals $\mathfrak{p}$ of $R$.
\item[(iv)] Let $\varrho: \mathcal{A} \to \mathcal{B}$ be a surjective homomorphism of $R$-orders. Write $B$ for the $F$-algebra spanned by $\mathcal{B}$ and  $\varrho_1:A\to B$, $\varrho_2: \zeta(A) \to \zeta(B)$ and $\varrho_3:\xi(\mathcal{A}) \to \xi(\mathcal{B})$ for the surjective ring homomorphisms induced by $\varrho$. Then $\mathcal{B}\otimes_{\mathcal{A},\varrho}P$ is a locally-free $\mathcal{B}$-module and the natural isomorphism of $\zeta(B)$-modules $\zeta(B)\otimes_{\zeta(A),\varrho_2}{\bigwedge}_{A}^r P_F \cong {\bigwedge}_{B}^r (B\otimes_{A,\varrho_1} P_F)$ restricts to give an isomorphism of invertible $\xi(\mathcal{B})$-modules $\xi(\mathcal{B})\otimes_{\xi(\mathcal{A}),\varrho_3}
    {\bigcap}_{\mathcal{A}}^r P \cong {\bigcap}_{\mathcal{B}}^r (\mathcal{B}\otimes_{\mathcal{A},\varrho}P)$, where the exterior powers in the latter module are defined with respect to the same ordered $E$-bases of those simple $A_E$-modules that factor through the scalar extension of $\varrho_1$.
\item[(v)]  If $P_1 \xrightarrow{\theta} P_2 \xrightarrow{\phi}P_3$ is a (split) short exact sequence in ${\rm Mod}^{\rm lf}(\mathcal{A})$, and we set $r_i := {\rm rk}_\mathcal{A}(P_i)$ for $i = 1,2,3$, then there exists an isomorphism of $\xi(\mathcal{A})$-modules
\[ {\rm i}^\diamond_\varpi(\theta,\phi): {\bigcap}_{\mathcal{A}}^{r_2}P_2 \cong {\bigcap}_{\mathcal{A}}^{r_1}P_1\otimes_{\xi(\mathcal{A})}
{\bigcap}_{\mathcal{A}}^{r_3}P_3\]
that has the following properties:
\begin{itemize}
\item[(a)] ${\rm i}^\diamond_\varpi(\theta,\phi)$ is functorial with respect to isomorphisms of short exact sequences;
\item[(b)] If $P_3 \xrightarrow{\phi'} P_2 \xrightarrow{\theta'} P_1$ is any exact sequence of $\mathcal{A}$-modules obtained by choosing a splitting of the given sequence, then the following diagram commutes
\[ \begin{CD}
{\bigcap}_{\mathcal{A}}^{r_2}P_2 @> {\rm i}^\diamond_\varpi(\theta,\phi) >>  &{\bigcap}_{\mathcal{A}}^{r_1}P_1\otimes_{\xi(\mathcal{A})}{\bigcap}_{\mathcal{A}}^{r_3}P_3\\
@\vert  & @VV x\otimes y\mapsto \alpha\cdot  y\otimes x V \\
{\bigcap}_{\mathcal{A}}^{r_2}P_2 @> {\rm i}^\diamond_\varpi(\phi',\theta') >> &{\bigcap}_{\mathcal{A}}^{r_3}P_3
\otimes_{\xi(\mathcal{A})}{\bigcap}_{\mathcal{A}}^{r_1}P_1.\end{CD}\]
Here $\alpha$ is the element $((-1)^{\rho_{1,i}\cdot \rho_{3,i}})_{i \in I}$ of ${\prod}_{i \in I}\zeta(A_i) = \zeta(A)$, where  $\rho_{j,i}$ denotes the $i$-th component of the reduced rank ${\rm rr}_A(P_{j,F})$.%$\tau$ is the isomorphism
%
%\[ {\bigwedge}_{\mathcal{A}}^{r_1}P_1\otimes_{\xi(\mathcal{A})}{\bigwedge}_{\mathcal{A}}^{r_3}P_3 \cong {\bigwedge}_{\mathcal{A}}^{r_3}P_3\otimes_{\xi(\mathcal{A})}{\bigwedge}_{\mathcal{A}}^{r_1}P_1\]
% $x\otimes y$ to $ (-1)^{r_1r_3}\cdot  y\otimes x$.
\end{itemize}
\end{itemize}
\end{proposition}

\begin{proof} Claim (i) follows directly from the proof of Theorem \ref{exp prop}(vi). (We note also that Lemma \ref{proprn2} implies the $\xi(\mathcal{A})$-module generated by $\wedge_{j=1}^{j=r}b_{j}$ is independent of the choice of basis $\{b_j\}_{1\le j\le r}$.)

After replacing $\mathcal{A}$ and $P$ by $\mathcal{A}_{(\mathfrak{p})}$ and $P_{(\mathfrak{p})}$ for a prime ideal $\mathfrak{p}$, the same argument implies that the $\xi(\mathcal{A}_{(\mathfrak{p})})$-module ${\bigcap}_{\mathcal{A}_{(\mathfrak{p})}}^r P_{(\mathfrak{p})}$ is free of rank one, with basis $\wedge_{j=1}^{j=r}b_{\frp,j}$ where the elements $b_{\frp,j}$ are chosen as in  (ii). Given this fact, (ii) follows directly from the result of Theorem \ref{exp prop}(iii).

To prove (iii) we fix a prime ideal $\wp$ of $\xi(\mathcal{A})$. Then $\mathfrak{p} := R \cap \wp$ is a prime ideal of $R$ and by Roiter's Lemma (cf. \cite[Lem. (31.6)]{curtisr}) there exists a free $\mathcal{A}$-submodule $P'$ of $P$ such that $P_{(\mathfrak{p})} = P'_{(\mathfrak{p})}$. This equality combines with Theorem \ref{exp prop}(iii) to imply that
\[ \bigl({\bigcap}_{\mathcal{A}}^r P\bigr)_{(\wp)} = \bigl(\bigl({\bigcap}_{\mathcal{A}}^r P\bigr)_{(\mathfrak{p})}\bigr)_{(\wp)} = \bigl(\bigl({\bigcap}_{\mathcal{A}}^r P'\bigr)_{(\mathfrak{p})}\bigr)_{(\wp)} = \bigl({\bigcap}_{\mathcal{A}}^r P'\bigr)_{(\wp)}.\]
%
%\[ \bigl({\bigcap}_{\mathcal{A}}^r P\bigr)_{(\wp)} = \bigl(\bigl({\bigcap}_{\mathcal{A}}^r P\bigr)_{(\mathfrak{p})}\bigr)_{(\wp)} =
 %\bigl(\bigl({\bigcap}_{\mathcal{A}_{(\mathfrak{p})}}^r P_{(\mathfrak{p})}\bigr)_{(\wp)} = \bigl(\bigl({\bigcap}_{\mathcal{A}_{(\mathfrak{p})}}^r P'_{(\mathfrak{p})}\bigr)_{(\wp)} = \bigl(\bigl({\bigcap}_{\mathcal{A}}^r P'\bigr)_{(\mathfrak{p})}\bigr)_{(\wp)} = \bigl({\bigcap}_{\mathcal{A}}^r P'\bigr)_{(\wp)}.\]
 %
 %where the second equality follows from Proposition \ref{exp prop}(iii).
In particular, since (i) implies ${\bigcap}_{\mathcal{A}}^r P'$ is a free $\xi(\mathcal{A})$-module of rank one, the $\xi(\mathcal{A})_{(\wp)}$-module $\bigl({\bigcap}_{\mathcal{A}}^r P\bigr)_{(\wp)}$ is also free of rank one, as required to prove  (iii). % that The result of claim (iii)  is a free
 % $\xi(\mathcal{A})_\mathfrak{m}$-module of rank one for every maximal ideal $\mathfrak{m}$ of $\xi(\mathcal{A})$.
%But, since the field $\xi(\mathcal{A})/\mathfrak{m}$ is finitely generated over $R$, it must be finite and so one can choose a prime ideal $\mathfrak{q}$ of $R$ that is coprime to the order of this field. Then $\bigl({\bigwedge}_{\mathcal{A}}^r P\bigr)_\mathfrak{m}$ is equal to the localisation at $\mathfrak{m}$ of ${\bigwedge}_{\mathcal{A}}^r P(\mathfrak{q})$ and so is free of rank one, as required.

Claim (iv) is verified by a straightforward exercise and, for brevity, we leave this to the reader.

Turning to (v) we fix an $\mathcal{A}$-module section $\sigma$ to $\phi$. We note that the given exact sequence implies $r_{2} = r_{1}+ r_{3}$ and also that for any given $A$-bases $\underline{b}_{j} := \{b_{j,a}\}_{1\le a\le r_{j}}$ of $P_{j,F}$ for $j=1,3$ we obtain an $A$-basis $\underline{b}_{2} := \{b_{2,a}\}_{1\le a\le r_{2}}$ of $P_{2,F}$ by setting $b_{2,a} = \theta(b_{1,a})$ if $1\le a\le r_{1}$ and $b_{2,a} = \sigma_i(b_{3,a-r_{1}})$ if $r_{1}< a\le r_{3}$.

Write $E$ for the algebra ${\prod}_{i \in I}E_i$. Then for each $j\in \{1,2,3\}$, the $E$-module ${\bigwedge}^{r_j}_AP_{j,F}$ is free of rank one, with basis $\wedge_{a=1}^{a=r_j}b_{j,a}$, and so there is a unique isomorphism of $E$-modules %We then write $P_{i,F}'$ for the free $\zeta(A)$-module of rank $r_i$ that is generated by the elements  $\{ \wedge^1_Ab_{i,a}\}_{1\le a\le r_i}$ and note that
%  $\bigl({\bigwedge}_{A}^{r_i}P_{i,F}\bigr)_0$ coincides with the $r_i$-th exterior power of the $\zeta(A)$-module $P_{i,F}'$.  In particular, the short exact sequence of free $\zeta(A)$-modules
%
%\[  0 \to P_{1,F}' \xrightarrow{\theta'} P_{2,F}' \xrightarrow{\phi'} P_{3,F}' \to 0\]
%
%in which $\theta'(\wedge^1_Ab_{1,a}) = \wedge^1_Ab_{2,a}$ for $1\le a\le r_1$ and $\phi'(\wedge^1_Ab_{2,a}) = \wedge^1_Ab_{3,a}$ for $r_{1}< a\le r_{3}$ induces an isomorphism of $\zeta(A)$-modules
%
\begin{align*} \Delta: {\bigwedge}_{A}^{r_2}P_{2,F} = \bigl({\bigwedge}_{A}^{r_1}P_{1,F}\bigr)\otimes_{E}
\bigl({\bigwedge}_{A}^{r_3}P_{3,F}\bigr)\end{align*}
that sends $\wedge_{j=1}^{j=r_{2}}b_{2,j}$ to $(\wedge_{s=1}^{s=r_{1}}b_{1,s})\otimes(\wedge_{t=1}^{t=r_{3}}b_{3,t})$.
%
%\begin{equation*}\label{rational iso} \Delta(\wedge_{j=1}^{j=r_{2}}(\wedge^1_Ab_{2,j})) = (\wedge_{s=1}^{s=r_{1}}(\wedge^1_Ab_{1,s}))\otimes_{\zeta(A)}(\wedge_{t=1}^{t=r_{3}}(\wedge_A^1b_{3,t})).\end{equation*}
%
In addition, by using Lemma \ref{proprn2}, one checks this isomorphism is independent both of the choices of bases $\underline{b}_{1}$ and $\underline{b}_{3}$ and of the choice of section $\sigma$.

In particular, if one fixes a prime ideal $\mathfrak{p}$ of $R$ and then chooses the elements $\{b_{1,s}\}_{1\le s\le r_{1}}$ and $\{b_{3,t}\}_{1\le t\le r_{3}}$ to be  $\mathcal{A}_{(\mathfrak{p})}$-bases of $P_{1,(\mathfrak{p})}$ and $P_{3,(\mathfrak{p})}$, then the choice of $\sigma$ implies that the set $\{b_{2,j}\}_{1\le j\le r_{2}}$ defined above is an $\mathcal{A}_{(\mathfrak{p})}$-basis of $P_{2,(\mathfrak{p})}$ and so the explicit descriptions in (ii) imply that
\begin{align*} \left(\Delta\bigl({\bigcap}_{\mathcal{A}}^{r_2} P_2\bigr)\right)_{(\mathfrak{p})} = &\,  \Delta\bigl( \left({\bigcap}_{\mathcal{A}}^{r_2} P_2\right)_{(\mathfrak{p})}\bigr)\\
= &\, \left({\bigcap}_{\mathcal{A}}^{r_1} P_1\right)_{(\mathfrak{p})} \otimes _{\xi(\mathcal{A})_{(\mathfrak{p})}} \left({\bigcap}_{\mathcal{A}}^{r_3} P_3\right)_{(\mathfrak{p})}\\ 
= &\, \left({\bigcap}_{\mathcal{A}}^{r_1} P_1 \otimes _{\xi(\mathcal{A})} {\bigcap}_{\mathcal{A}}^{r_3} P_3\right)_{(\mathfrak{p})}.\end{align*}
Since this is true for all primes $\mathfrak{p}$, one therefore has
\[ \Delta\bigl({\bigcap}_{\mathcal{A}}^{r_2} P_2\bigr)= {\bigcap}_{\mathcal{A}}^{r_1} P_1 \otimes _{\xi(\mathcal{A})} {\bigcap}_{\mathcal{A}}^{r_3} P_3\]
% the equality in claim (iii) implies
%
%\[ \Delta\bigl( {\bigwedge}_{\mathcal{A}}^{r_2}P_{2}\bigr) = \bigl({\bigwedge}_{\mathcal{A}}^{r_1}P_{1}\bigr)\otimes _{\xi(\mathcal{A})}
%\bigl( {\bigwedge}_{\mathcal{A}}^{r_3}P_{3}\bigr)\]
%
and so we can define ${\rm i}^\diamond_\varpi(\theta,\phi)$ to be the isomorphism induced by
 restricting $\Delta$. It is then straightforward to see that this isomorphism is functorial with respect to isomorphisms of short exact sequences, as required by (v)(a).

To justify (v)(b) we set $\rho_{2,i} := {\rm rr}_{A_i}(A_i\otimes_A P_{2,F})$ for each index $i$ in $I$. Then the definition (\ref{rr def}) of reduced rank combines with the given exact sequence to imply $\rho_{2,i} = \rho_{1,i} + \rho_{3,i}$. In addition, it combines with the explicit definitions of reduced exterior powers to imply that 
\[ {\bigwedge}_{A_i}^{r_2}(A_i\otimes _A P_{2,F}) = {\bigwedge}_{E_i}^{\rho_{2,i}}W_i,\]
with $W_i$ the $E_i$-space $V_i^\ast\otimes_{A_{i}\otimes_{\zeta(A_i)}E_i}P_{2,E_i}$ of dimension $\rho_{2,i}$, and that the $i$-th components $(\wedge_{s=1}^{s=r_{1}}b_{2,s})_i$ and $(\wedge_{t=r_1+1}^{t=r_{3}}b_{2,t})_i$ of the elements $\wedge_{s=1}^{s=r_{1}}b_{2,s}$ and $\wedge_{t=r_1+1}^{t=r_{3}}b_{2,t}$ are respectively the exterior products of $\rho_{1,i}$ and $\rho_{3,i}$ distinct elements of $W_i$. In the space ${\bigwedge}_{A_i}^{r_2}(A_i\otimes _A P_{2,F})$ one therefore has
\[   (\wedge_{s=1}^{s=r_{1}}b_{2,s})_i\wedge (\wedge_{t=r_1+1}^{t=r_{3}}b_{2,t})_i = (-1)^{\rho_{1,i}\cdot \rho_{3,i}}\cdot (\wedge_{t=r_1+1}^{t=r_{3}}b_{2,t})_i\wedge (\wedge_{s=1}^{s=r_{1}}b_{2,s})_i.\]
Taken together, these equalities imply that the diagram in (v)(b) commutes, as required to complete the proof.
\end{proof}

\begin{remark}{\em Let $P$ be a free $\mathcal{A}$-module of rank one. If $\mathcal{A}$ is commutative, then there is a natural identification
${\bigcap}_{\mathcal{A}}^1P  = {\bigwedge}_{\mathcal{A}}^1 P \cong P$. In general, however, ${\bigcap}_{\mathcal{A}}^1 P$ is a module over $\xi(\mathcal{A})$ and hence different from $P$.}\end{remark}

\subsubsection{}The results of Lemma \ref{rr=gi} (with $\mathcal{R} = \xi(\mathcal{A})$) and Proposition \ref{first independence check}(iii) combine to imply that for each $P$ in ${\rm Mod}^{\rm lf}(\mathcal{A})$ one obtains a well-defined object of $\mathcal{P}(\xi(\mathcal{A}))$ by setting
\[ {\rm d}_{\mathcal{A},\varpi}^\diamond(P) := \bigl({\bigcap}_{\mathcal{A}}^{{\rm rk}_\mathcal{A}(P)}P,\,{\rm rr}_A(P_F)\bigr).\]
%
%where in the second component $r$ is regarded as a constant function on ${\rm Spec}(\xi(\mathcal{A}))$.

For each short exact sequence $0\to  P_1 \xrightarrow{\theta} P_2 \xrightarrow{\phi} P_3 \to 0$ in ${\rm Mod}^{\rm lf}(\mathcal{A})$ the construction in Proposition \ref{first independence check}(v) also gives rise to a commutative diagram of isomorphisms in $\mathcal{P}(\xi(\mathcal{A}))$ of the form
\[\begin{CD}
  {\rm d}^\diamond(P_2) @> {\rm i}_\varpi^\diamond(\theta,\phi) >>
  {\rm d}^\diamond(P_1)\otimes {\rm d}^\diamond(P_3)\\
  @\vert @VV \psi_{{\rm d}^\diamond(P_1),{\rm d}^\diamond(P_3)} V\\
  {\rm d}^\diamond(P_2) @> {\rm i}_\varpi^\diamond(\phi',\theta') >> {\rm d}^\diamond(P_3)\otimes {\rm d}^\diamond(P_1).\end{CD}\]
in which we abbreviate ${\rm d}_{\mathcal{A},\varpi}^\diamond$ to ${\rm d}^\diamond$.

\begin{proposition}\label{functor construct} The associations $P \mapsto {\rm d}_{\mathcal{A},\varpi}^\diamond(P)$ and
 $(\theta,\phi) \mapsto {\rm i}_\varpi^\diamond(\theta,\phi)$ give a
 well-defined determinant functor ${\rm d}^\diamond_{\mathcal{A},\varpi}: {\rm Mod}^{\rm lf}(\mathcal{A})_{\rm is} \to \mathcal{P}(\xi(\mathcal{A}))$.%of the form (\ref{det functor claim}).

In addition, for any homomorphism $\varrho: \mathcal{A} \to \mathcal{B}$ as in Theorem \ref{ext det fun thm}(ii), and any module $P$ in ${\rm Mod}^{\rm lf}(\mathcal{A})$,  there exists a canonical isomorphism in $\mathcal{P}(\xi(\mathcal{B}))$ of the form
\[ \xi(\mathcal{B})\otimes_{\xi(\mathcal{A}),\varrho_\ast}{\rm d}_{\mathcal{A},\varpi}^\diamond(P) \cong {\rm d}_{\mathcal{B},\varpi'}^\diamond(\mathcal{B}\otimes_{\mathcal{A},\varrho}P).\]
\end{proposition}

\begin{proof} The above associations combine with the result of Theorem   \ref{exp prop}(i) to give data as in (a), (b) and (c) of Definition \ref{det functor def}.

It is clear that these data satisfy condition (d) in the latter definition and also straightforward to check that they satisfy condition (e) by using the general result of Lemma \ref{compatible} below (with $\Lambda = \mathcal{A}$) to make a compatible choice of sections when computing each of the maps ${\rm i}_\varpi^\diamond(\gamma',\delta'), {\rm i}^\diamond_\varpi(\gamma'',\delta''), {\rm i}_\varpi^\diamond(\alpha',\beta'),$ ${\rm i}_\varpi^\diamond(\alpha'',\beta''), {\rm i}_\varpi^\diamond(\alpha,\beta)$ and ${\rm i}_\varpi^\diamond(\gamma,\delta)$.

Finally, the existence of the displayed isomorphism follows directly from the result of Proposition \ref{first independence check}(iv).
\end{proof}

\begin{lemma}\label{compatible} We assume to be given a ring $\Lambda$ and a commutative diagram of short exact sequences of finitely generated projective $\Lambda$-modules of the form
\[\begin{CD}
 M_1 @> d_1'>> N_1 @> d_1>> P_1\\
@V\epsilon_1 VV @V\phi_1 VV @V \kappa_1VV \\
M_2 @> d_2'>> N_2 @> d_2>> P_2\\
 @V\epsilon_2 VV @V\phi_2 VV @V \kappa_2VV\\
M_3 @> d_3'>> N_3 @> d_3>> P_3.\\
\end{CD}\]
Then there exist $\Lambda$-equivariant sections $\sigma_i: P_i \rightarrow N_i$ to $d_i$ for $i=1,2$ and $3$ such that there are commutative diagrams of $\Lambda$-modules
\begin{equation}\label{diagrams} \begin{CD} N_1 @< \sigma_1 << P_1 \\
              @V \phi_1 VV @VV \kappa_1V\\
              N_2 @<< \sigma_2 < P_2\end{CD} \,\,\,\,\,\,\,\text{   and   }\,\,\,\,\,\,\,\, \begin{CD} N_2 @< \sigma_2 << P_2 \\
              @V \phi_2 VV @VV \kappa_2V\\
              N_3 @<< \sigma_3 < P_3.\end{CD}\end{equation}
%Maybe noetherian is unnecessary and only the projectivity of $P_1,P_2,P_3$ is used in the proof?}
\end{lemma}

\begin{proof} First choose any $\Lambda$-equivariant
 section $\sigma$ to $d_2$ and write $\theta$ for the composite homomorphism
$\phi_2\circ \sigma\circ \kappa_1: P_1\to N_3$.

The commutativity of the given diagram implies that there exists a unique homomorphism $\theta_1$ in ${\rm Hom}_\Lambda(P_1,M_3)$ such that $\theta = d_3'\circ\theta_1$. Since $P_1$ is a projective $\Lambda$-module
 we can then choose a homomorphism $\theta_2$ in ${\rm Hom}_\Lambda(P_1,M_2)$ with $\theta_1 = \epsilon_2\circ \theta_2$.

Next we note that, since $P_3$ is a projective $\Lambda$-module, the group ${\rm Ext}_\Lambda^1(P_3,M_2)$ vanishes and so there exists a homomorphism $\theta_3$ in $\Hom_\Lambda(P_2,M_2)$ with $\theta_2 = \theta_3\circ \kappa_1$.

We now set $\sigma_2 := \sigma - d_2'\circ\theta_3 \in {\rm Hom}_\Lambda(P_2,N_2)$. Then $\sigma_2$ is a section to $d_2$ since
\[ d_2\circ \sigma_2 = d_2\circ \sigma  - (d_2\circ d_2')\circ\theta_3 = d_2\circ \sigma.\]
In addition, for $x$ in $P_1$ one has
\begin{align*} \phi_2(\sigma_2(\kappa_1(x))) &= \phi_2(\sigma(\kappa_1 (x))) - \phi_2(d_2'\circ\theta_3(\kappa_1(x))) \\
&= \theta(x) - d_3'(\epsilon_2(\theta_3\circ\kappa_1)(x)) \\
&= \theta(x) - d_3'((\epsilon_2\circ\theta_2)(x))
\\
&= \theta(x) - (d_3'\circ\theta_1)(x)\\
&= \theta (w') - \theta(w') =0.\end{align*}
Since $P_1$ is a projective $\Lambda$-module, this implies there exists a unique homomorphism $\sigma_1$ in $\Hom_\Lambda(P_1,N_1)$ which makes the first diagram in (\ref{diagrams}) commute (with respect to our fixed map $\sigma_2$) and hence that $\kappa_1(d_1\circ \sigma_1) = (d_2\circ \sigma_2) \circ \kappa_1= \kappa_1$ 
so that $\sigma_1$ is a section to $d_1$.

Finally we note that the commutativity of the first diagram in (\ref{diagrams}) implies there exists a (unique) homomorphism $\sigma_3$ in $\Hom_{\Lambda}(P_3,N_3)$ which makes the second diagram in (\ref{diagrams}) commute and one checks easily that this homomorphism is a
section to $d_3$, as required. \end{proof}

\subsection{Extended determinant functors}\label{ext det section}

\subsubsection{}Let $\Lambda$ be a noetherian ring. We write ${\rm Mod}(\Lambda)$ for the category of finitely generated (left) $\Lambda$-modules.

We assume to be given an abelian subcategory ${\rm Mod}^\dagger(\Lambda)$ of ${\rm Mod}(\Lambda)$ that is exact in the sense of Quillen \cite{quillen} and a determinant functor on ${\rm Mod}^\dagger(\Lambda)$ in the sense of Definition \ref{det functor def}.

We write $\mathcal{P}$ for the target category of this determinant functor and
 ${\rm d}^\diamond$, ${\rm i}^\diamond$ (and $\zeta$) for the associated data as in Definition \ref{det functor def} (a), (b) (and (c)).

We write $\Der^\dagger(\Lambda)$ for the full triangulated subcategory of $\Der(\Lambda)$ comprising complexes that are isomorphic to a bounded complex of modules in ${\rm Mod}^\dagger(\Lambda)$.  We also write $D^\dagger(\Lambda)_{\rm is}$ for the subcategory of $D^\dagger(\Lambda)$ in which morphisms are restricted to isomorphisms. %be quasi-isomorphisms of complexes.

We regard ${\rm Mod}^\dagger(\Lambda)_{\rm is}$ as a subcategory of $D^\dagger(\Lambda)_{\rm is}$ by identifying each object $M$ of ${\rm Mod}^\dagger(\Lambda)$ with the complex that comprises $M$ in degree zero and is zero in all other degrees.

In what follows we use the term `true triangle' as synonymous for
`short exact sequence of complexes'.

\begin{definition}\label{ext det functor def}
{\em An `extension' of the determinant functor comprising ${\rm d}^\diamond$ and ${\rm i}^\diamond$ to the category $D^\dagger(\Lambda)$ comprises data of the following form.
\begin{itemize}
\item[(a)] A covariant functor $ {\rm d} : D^\dagger(\Lambda)_{\rm is} \rightarrow \mathcal{P}$.

\item[(b)] For each true triangle $X\xrightarrow{u}
Y\xrightarrow{v} Z$ in which $X,Y$ and $Z$ are objects of $D^\dagger(\Lambda)$, an isomorphism ${\rm i}(u,v):{\rm d}(Y)\xrightarrow{\sim}{\rm d}(X)\boxtimes{\rm d}(Z)$ in $ \mathcal{P}$.
\end{itemize}

This data is subject to the following axioms:

\begin{itemize}
\item[(i)] If
\[\begin{CD}
X @>u>> Y @>v>> Z \\  @VfVV    @VgVV   @VhVV \\ X' @>u'>> Y'
@>v'>> Z'
\end{CD}\]
is a commutative diagram of true triangles and $f,g,h$ are 
quasi-isomorphisms, then
\[ (\vir{f}\boxtimes\vir{h})\circ {\rm i}(u,v)\circ\vir{g}^{-1}= {\rm i}(u',v').\]
\item[(ii)] If $u$, respectively $v$, is a quasi-isomorphism,
then ${\rm i}(u,v) =\vir{u}^{-1}$, respectively ${\rm i}(u,v)=\vir{v}$.
%\item[c)] $\vir{\,\,}$ commutes with the functors
%induced by any ring extension $R\rightarrow R'$ and for any true triangle
%$E$ we have $R'\otimes_R \vir{E}=\vir{R'\otimes_R E}$.
\item[(iii)] For any commutative diagram of complexes

\begin{equation*}\label{tnt}
\mbox{$\minCDarrowwidth1em
\begin{CD}
 X @>u>> Y @>v>> Z \\  @VfVV    @VgVV   @VhVV    \\
X' @>u'>> Y' @>v'>> Z' \\ @Vf'VV    @Vg'VV   @Vh'VV \\ X'' @>u''>>
Y'' @>v''>> Z''
\end{CD}$}
\end{equation*}
in which all of the rows and columns are true triangles and all terms are objects of $D^\dagger(\Lambda)$, the following
diagram in $ \mathcal{P}$ commutes
\begin{equation*}
%\hskip-0.6truein
\begin{CD}
\vir{Y'} @> {\rm i}(g,g')>> \vir{Y}\boxtimes \vir{Y''}\\
@V{\rm i}(u',v')VV @VV  ({\rm i}(u,v)\boxtimes\, {\rm i}(u'',v''))V\\
\vir{X'}\boxtimes \vir{Z'} @> (1\boxtimes \psi_{\vir{X''},\vir{Z}}\boxtimes 1)\cdot ({\rm i}(f,f')\boxtimes\, {\rm i}(h,h'))>> \vir{X}\boxtimes
\vir{Z}\boxtimes \vir{X''}\boxtimes \vir{Z''}.
\end{CD}
\end{equation*}
\item[(iv)] On the subcategory ${\rm Mod}^\dagger(\Lambda)_{\rm is}$ one has
 ${\rm d}= {\rm d}^\diamond$ and ${\rm i} = {\rm i}^\diamond$.
\end{itemize}
A collection of data as in (a) and (b) that satisfies the conditions (i), (ii), (iii) and (iv) with respect to any choice of determinant functor ${\rm d}^\diamond$ and $i^\diamond$ on ${\rm Mod}^\dagger(\Lambda)$ will be referred to as an `extended determinant functor' on $D^\dagger(\Lambda)$.
}
\end{definition}

We can now state the main result of this section.

\begin{proposition}\label{intermediate prop} The determinant functor constructed in Proposition \ref{functor construct} has a canonical extension to the category $\Der^{\rm lf}(\mathcal{A}).$

We write ${\rm d}_{\mathcal{A},\varpi}$ and ${\rm i}_{\mathcal{A},\varpi}$ for the data associated to this extension as in Definition \ref{ext det functor def}(a) and (b).  Then the extended determinant functor has the following additional properties.
\begin{itemize}
%(Note that we have suppressed any explicit
%reference to commutativity constraints in the above diagram).
%\item[e)] For any object $X$ of
%$D^{{\rm lf},p}(\mathcal{A})$ there exists a canonical isomorphism
%\begin{equation} \vir{X}\xrightarrow{\sim}\underset{i\in\bz}{\bigotimes}
%\vir{H^i(X)}^{(-1)^{i}}
%\label{kmd}
%\end{equation}
%that is functorial with respect to quasi-isomorphisms.
%\end{itemize}
%\label{derive}
\item[(v)] Fix a homomorphism $\varrho: \mathcal{A} \to \mathcal{B}$ as in Theorem \ref{ext det fun thm}(ii). Then for any complex $X$ in $\Der^{\rm lf}(\mathcal{A})$ the complex $\mathcal{B}\otimes_{\mathcal{A},\varrho}^\mathbb{L}X$ belongs to $\Der^{\rm lf}(\mathcal{B})$ and there exists a canonical isomorphism in $\mathcal{P}(\xi(\mathcal{B}))$ of the form
    \[ \xi(\mathcal{B})\otimes_{\xi(\mathcal{A}),\varrho_\ast}{\rm d}_{\mathcal{A},\varpi}(X) \cong {\rm d}_{\mathcal{B},\varpi'}(\mathcal{B}\otimes_{\mathcal{A},\varrho}^\mathbb{L}X).\]
\item[(vi)] The restriction of ${\rm d}_{\mathcal{A},\varpi}$ to $\Der^{{\rm lf},0}(\mathcal{A})_{\rm is}$ is independent of the choice of bases $\varpi$.
\end{itemize}
\end{proposition}

\begin{proof} This argument follows the approach used by Flach and the first author to prove the same result in the setting of virtual objects (see \cite[Prop. 2.1]{bf}).

The essential point is therefore that, excluding the assertions (v) and (vi), the claimed result follows directly from the general result  \cite[Prop. 4]{knumum} of Knudsen and Mumford and the formal constructions that are used to prove \cite[Th. 2]{knumum}, via which the same statements are proved for the determinant functor over a commutative ring.

To be more precise, if one removes the condition (iv) (regarding compatibility with scalar extensions) from the definition of extended determinant functor that is given in \cite[Def. 4]{knumum}, then the only properties of the determinant functor that are used in the constructions  that underlie the proof of \cite[Th. 2]{knumum} are those listed in
 [loc. cit., Prop. 1 (excluding (iii))] and Proposition \ref{functor construct} implies that ${\rm d}^\diamond_{\mathcal{A},\varpi}$ and
 ${\rm i}^\diamond_{\mathcal{A},\varpi}$ have all of these necessary properties.

In addition, since the restriction of ${\rm d}_{\mathcal{A},\varpi}$ to
 ${\rm Mod}^{\rm lf}(\mathcal{A})_{\rm is}$ is equal to ${\rm d}^\diamond_{\mathcal{A},\varpi}$ the property in (v) follows from the final assertion of Proposition \ref{functor construct}.

Lastly, we note that the property in (vi) follows directly from the result of Lemma \ref{technical complex}(ii)  below.\end{proof}

\begin{remark}\label{shift sign}{\em Proposition \ref{intermediate prop} implies that for any $X$ in $\Der^{\rm lf}(\mathcal{A})$ there is a natural isomorphism 
\[ {\rm d}_{\mathcal{A},\varpi}(X[1]) \xrightarrow{\sim} {\rm d}_{\mathcal{A},\varpi}(X)^{-1}\]
in $\mathcal{P}(\xi(\mathcal{A}))$. This is because if we write ${\rm Cone}_{X}$ for the mapping cone of the identity endomorphism of $X$, then the associated true triangle $X \xrightarrow{u} {\rm Cone}_X \xrightarrow{v} X[1]$ induces a composite isomorphism
\[ {\rm d}_{\mathcal{A},\varpi}(X)\otimes {\rm d}_{\mathcal{A},\varpi}(X[1]) \xrightarrow{{\rm i}_\varpi(u,v)} {\rm d}_{\mathcal{A},\varpi}({\rm Cone}_X ) \xrightarrow{\sim} {\bf 1}_{\mathcal{P}(\xi(\mathcal{A}))},\]
where the second isomorphism is induced by the acyclicity of ${\rm Cone}_X$.}\end{remark}

\subsubsection{}For the reader's convenience, we give a more explicit description of the construction in Proposition \ref{intermediate prop}.

%For any invertible $\xi(\mathcal{A})$-module $\mathcal{L}$ we set $\mathcal{L}^{-1}:= \Hom_{\xi(\mathcal{A})}(\mathcal{L},\xi(\mathcal{A}))$, regarded as an (invertible) $\xi(\mathcal{A})$-module via the natural composition action of $\xi(\mathcal{A})$.

\begin{lemma}\label{technical complex}\

\begin{itemize}
\item[(i)] Let $P^\bullet$ be a complex in $\DerC^{\rm lf}(\mathcal{A})$. Then for each quasi-isomorphism $u: P^\bullet \to X$ of complexes of $\mathcal{A}$-modules, the map ${\rm d}_{\mathcal{A},\varpi}(u)$ induces an isomorphism
\[ \bigotimes_{i \in \ZZ}\bigl({\bigcap}_{\mathcal{A}}^{{\rm rk}_\mathcal{A}(P^i)}P^i,\,{\rm rr}_A(P^i_F)\bigr)^{(-1)^i}  \xrightarrow{\sim} {\rm d}_{\mathcal{A},\varpi}(X)\]
in $\mathcal{P}(\xi(\mathcal{A}))$, where each lattice  ${\bigcap}_{\mathcal{A}}^{{\rm rk}_\mathcal{A}(P^i)}P^i$ is defined with respect to the bases $\varpi$.

\item[(ii)] If $X$ belongs to $\Der^{{\rm lf},0}(\mathcal{A})$, then ${\rm d}_{\mathcal{A},\varpi}(X)$ is independent of the choice of bases $\varpi$.
\end{itemize}
\end{lemma}

\begin{proof} To prove (i) it is enough to show that there is a canonical isomorphism in $\mathcal{P}(\xi(\mathcal{A}))$ of the form
\begin{equation}\label{tauto isom} \bigotimes_{i \in \ZZ}
\bigl({\rm d}^\diamond_{\mathcal{A},\varpi}(P^i)\bigr)^{(-1)^i} \xrightarrow{\sim} {\rm d}_{\mathcal{A},\varpi}(P^\bullet).\end{equation}

To show this we can clearly assume $P^\bullet$ is non-zero. We then let $d$ denote the largest integer for which $P^d$ is non-zero and write $P^\bullet_{<d}$ for the complex obtained from $P^\bullet$ by replacing $P^d$ by $0$ and leaving all other terms unchanged.

Then there is a natural true triangle in $\Der^{\rm lf}(\mathcal{A})$
\[ P^d[-d] \xrightarrow{u} P^\bullet \xrightarrow{v} P^\bullet_{<d}\]
and hence an isomorphism in $\mathcal{P}(\xi(\mathcal{A}))$
\[ {\rm i}_\varpi(u,v) : {\rm d}_{\mathcal{A},\varpi}(P^d[-d])\otimes {\rm d}_{\mathcal{A},\varpi}(P^\bullet_{<d}) \xrightarrow{\sim}
 {\rm d}_{\mathcal{A},\varpi}(P^\bullet).\]

By an induction on the number of non-zero terms of $P^\bullet$ we are therefore reduced to proving that ${\rm d}_{\mathcal{A},\varpi}(P^d[-d])$ identifies with $\bigl({\rm d}^\diamond_{\mathcal{A},\varpi}(P^d)\bigr)^{(-1)^d}$. This follows directly from (repeated application of) Remark \ref{shift sign} and the fact that ${\rm d}_{\mathcal{A},\varpi}(P^d[0]) = {\rm d}^\diamond_{\mathcal{A},\varpi}(P^d)$.

To prove (ii) it is enough to fix a quasi-isomorphism $P^\bullet \to X$ as in (i) and show that the graded module $\bigotimes_{i \in \ZZ}
\bigl({\rm d}^\diamond_{\mathcal{A},\varpi}(P^i)\bigr)^{(-1)^i}$ is independent of the choice of ordered bases $\varpi$. Since it is enough to prove this after localising at each prime ideal of $R$ we can also assume, without loss of generality, that each $\mathcal{A}$-module $P^i$ is free. In each degree $i$ we then set
 $r_i := {\rm rk}_\mathcal{A}(P^i)$, fix an ordered $\mathcal{A}$-basis $\{b_{i,j}\}_{1\le j\le r_i}$ of $P^i$ and write $\{b^\ast_{i,j}\}_{1\le j\le r_i}$ for the $\mathcal{A}$-basis of $\Hom_\mathcal{A}(P^i,\mathcal{A})$ that is dual to $\{b_{i,j}\}_{1\le j\le r_i}$.

Then Lemma \ref{pairing cor} combines with Proposition \ref{first independence check}(i) to imply the natural isomorphism
\[ \Hom_{\zeta(A)}({\bigwedge}_{A}^{r_i}P^i_F,\zeta(A))\cong {\bigwedge}_{A^{\rm op}}^{r_i}\Hom_A(P^i_F,A)\]
identifies the module  $\Hom_{\xi(\mathcal{A})}({\bigcap}_{\mathcal{A}}^{r_i}P^i,\xi(\mathcal{A}))$ with $\xi(\mathcal{A})\cdot \wedge_{j=1}^{j=r_i}b^\ast_{i,j}$. 

Setting $b^1_{i,j} := b_{i,j}$ and $b^{-1}_{i,j} := b^\ast_{i,j}$ for each $i$ and $j$, it is therefore enough to prove that if the image of $\chi_\mathcal{A}(P^\bullet)$ in  $K_0(A)$ vanishes, then the element
\begin{equation}\label{tensor element} \bigotimes_{i \in \ZZ} {\wedge}_{j=1}^{j=r_i} b^{(-1)^i}_{i,j}\end{equation}
is independent of the choice of bases $\varpi$ used in the definition of exterior products.

It suffices to verify this after projecting to each simple component of $A$ and so we shall assume $A$ is simple (and use the notation of Definition \ref{rth reduced ext def}). We then fix bases $\{v_j\}_{1\le j\le d}$ and $\{\tilde v_j\}_{1\le j\le d}$ of the $E$-space $V$ and write $M = (M_{st})_{1\le s,t\le d}$ for the matrix in ${\rm GL}_d(E)$ that satisfies $\tilde v_s = {\sum}_{t =1}^{t=d}M_{st}v_t$ for each integer $s$. Then, writing $N$ for the matrix in ${\rm GL}_d(E)$ that is equal to the inverse of the transpose of $M$, one has  $\tilde v^*_s = {\sum}_{t =1}^{t=d}N_{st}v_t^*$ for each integer $s$ and by using these equalities one computes that in each even degree $i$ there is an equality
\[ {\wedge}_{1\leq j \leq r_i}({\wedge}_{1\leq s\leq d}\tilde v^\ast_s\otimes b_{i,j}) = {\rm det}(N)^{r_i}\cdot {\wedge}_{1\leq j \leq r_i}({\wedge}_{1\leq s\leq d}v^\ast_s\otimes b_{i,j})\]
and in each odd degree $i$ an equality
\[ {\wedge}_{1\leq j \leq r_i}({\wedge}_{1\leq s\leq d}\tilde v_s\otimes b^\ast_{i,j}) = {\rm det}(M)^{r_i}\cdot{\wedge}_{1\leq j \leq r_i}({\wedge}_{1\leq s\leq d}v_s\otimes b^\ast_{i,j}). \]
Since ${\rm det}(M) = {\rm det}(N)^{-1}$ this implies that the difference between the elements (\ref{tensor element}) when the exterior products are computing using the basis $\{v_j\}_{1\le j\le d}$, respectively  $\{\tilde v_j\}_{1\le j\le d}$, of $V$ is the factor ${\rm det}(M)^{{\sum}_{i \in \ZZ}(-1)^ir_i}$. To complete the proof it is therefore enough to note that  if the image of $\chi_\mathcal{A}(P^\bullet)$ in  $K_0(A)$ vanishes, then the sum ${\sum}_{i \in \ZZ}(-1)^ir_i$ is equal to $0$.
\end{proof}

\subsection{The proof of Theorem \ref{ext det fun thm}}\label{derived ext section} To complete the proof of Theorem \ref{ext det fun thm} we must show that the construction of Proposition \ref{intermediate prop} retains the properties (i), (ii) and (iii) in Definition \ref{ext det functor def} (with ${\rm Mod}^\dagger(\Lambda) = {\rm Mod}^{\rm lf}(\mathcal{A})$) after one replaces true triangles and quasi-isomorphisms by arbitrary exact triangles and isomorphisms in $\Der^{\rm lf}(\mathcal{A})$. In particular, whilst in this case the diagrams in Definition \ref{ext det functor def}(i) and (iii)  are commutative in the category of complexes of $\mathcal{A}$-modules, we need to establish analogous results for diagrams that commute only up to homotopy.

\subsubsection{}%In the following result we write $\iota_F$ and $\iota_F'$ for the natural scalar extension functors from $V$ denote the natural scalar extension maps.
As a key preliminary step, we consider similar results for the semisimple algebra $A$. To do this we recall that, for any choice of ordered bases $\varpi$ as in \S\ref{ordered basis sec} the argument of Flach and the first author in \cite[Lem. 2]{bf} constructs a determinant functor on ${\rm Mod}(A)$.

To recall the explicit construction we use the notation of \S\ref{ordered basis sec}. In particular, we assume first that $A=A_1$ is simple and hence equal to ${\rm M}_{n_1}(D_1)$ for a division ring $D_1$ with Schur index $m_1$ and we fix a splitting field $E_1$ for $D_1$ that is a maximal subfield of $D_1$. We set $n := n_1$, $m := m_1$, $D := D_1$, $E := E_1$, $W := W_1$, $V := V_1 = W_1^n$ and use the ordered $E$-bases $\{w_a\}_{1\le a\le m}$ of $W$ and $\varpi = \underline{\varpi}_1$ of $V$.

Then for each $M$ in  ${\rm Mod}(A)$ the (left) $D$-module $D^{ n}\otimes_AM$ is free of rank $r := {\rm rk}_D(M)/n$ and so, since ${\rm rr}_D(D^n\otimes_AM) = {\rm rr}_A(M)$, one has ${\rm rr}_A(M) = r\cdot m$.   %If $E$ is an extension of $F'$ such that $E\otimes_{F'}D \cong {\rm M}_{d}(E)$ we can also fix an ordered $E$-basis $\{v_i\}_{1\le i\le d}$ of a simple ${\rm M}_{d}(E)$-module $V$.
If we fix an ordered $D$-basis $\{b_a\}_{1\le a\le r}$ of $D^n\otimes_AM$, then Lemma \ref{proprn2} implies that the dimension one $\zeta(D)$-space spanned by $\wedge_{a=1}^{a=r}(\wedge_{s=1}^{s=m}(w_s^\ast \otimes b_a))$ is independent of the choice of $\{ b_a\}_{1\le a\le r}$ and we set
\begin{equation}\label{ss graded def} {\rm d}^\diamond_{A,\varpi}(M) := \bigl(\zeta(D)\cdot \wedge_{a=1}^{a=r}(\wedge_{s=1}^{s=m}(w_s^\ast \otimes b_a)),{\rm rr}_A(M)\bigr).\end{equation}

We next assume to be given a short exact sequence $M_1 \to M_2 \to M_3$ in ${\rm Mod}(A)$ and set $r_j := {\rm rk}_D(M_j)/n$ for each $j = 1,2, 3$. Then, by following the same approach as the proof of Proposition \ref{first independence check}(v), we can use a choice of splitting $\sigma$ of the induced short exact sequence of free $D$-modules 
\[ 0 \to D^{n}\otimes_{A}M_1 \to D^{n}\otimes_{A}M_2 \to D^{n}\otimes_{A}M_3\to 0\]
to construct a basis $\{b_{2a}\}_{1\le a\le r_2}$ of $D^n\otimes_{A}M_2$
from given bases $\{b_{ja}\}_{1\le a\le r_j}$ of $D^{ n}\otimes_{A}M_j$ for $j = 1,3$. We then define
\[ {\rm i}^\diamond_{A,\varpi}: {\rm d}^\diamond_{A,\varpi}(M_2) \to {\rm d}^\diamond_{A,\varpi}(M_1)\otimes {\rm d}^\diamond_{A,\varpi}(M_3)\]
to be the unique isomorphism of graded $\zeta(D)$-spaces with  %
\begin{multline*} {\rm i}^\diamond_{A,\varpi}(\wedge_{a=1}^{a=r_2}(\wedge_{s=1}^{s=m}(w_s^\ast \otimes b_{2a})),{\rm rr}_A(M_2))
\\ = (\wedge_{a=1}^{a=r_1}(\wedge_{s=1}^{s=m}(w_s^\ast \otimes b_{1a})),{\rm rr}_A(M_1))\otimes (\wedge_{a=1}^{a=r_3}(\wedge_{s=1}^{s=m}(w_s^\ast \otimes b_{3a})),{\rm rr}_A(M_3)).\end{multline*}
(Lemma \ref{proprn2} implies that this map is independent of the splitting $\sigma$ and bases $\{b_{ja}\}_{1\le a\le r_j}$ for $j=1,3$ that are used.)

In the more general case that $A$ is not simple, one uses its Wedderburn decomposition $A = {\prod}_{i \in I}A_i$ to define ${\rm d}^\diamond_{A,\varpi}$ and ${\rm i}^\diamond_{A,\varpi}$ componentwise. With this construction, the image under ${\rm d}^\diamond_{A,\varpi}$ of a module $M$ in ${\rm Mod}(A)$ has grading equal to ${\rm rr}_A(M)$, regarded as a function on ${\rm Spec}(\zeta(A)) = \bigcup_{i \in I}{\rm Spec}(\zeta(D_i))$ in the obvious way.

Before stating the next result we note that, since $A$ is semisimple, the category $\Der^{\rm perf}(A)$ identifies with the full triangulated subcategory of $\Der(A)$ comprising complexes that are isomorphic to a bounded complex of finitely generated $A$-modules.

\begin{proposition}\label{semisimple det} There exists a canonical extension to $\Der^{\rm perf}(A)$ of the determinant functor given by ${\rm d}^\diamond_{A,\varpi}$ and ${\rm i}^\diamond_{A,\varpi}$. The associated functor
\[ {\rm d}_{A,\varpi}: \Der^{\rm perf}(A)_{\rm is} \to \mathcal{P}(\zeta(A))\]
has the following properties.

\begin{itemize}
\item[(i)] For any object $X$ of $\Der^{\rm perf}(A)$ there exists a canonical isomorphism
\begin{equation*} {\rm d}_{A,\varpi}(X) \xrightarrow{\sim}\underset{i\in\bz}{\bigotimes}\,
 {\rm d}_{A,\varpi}^\diamond(H^i(X))^{(-1)^{i}}
\label{kmd}
\end{equation*}
in $\mathcal{P}(\zeta(A))$ that is functorial with respect to quasi-isomorphisms.

\item[(ii)] The following diagram
\[\begin{CD}
\Der^{\rm perf}(A)_{\rm is} @> {\rm d}_{A,\varpi} >> \mathcal{P}(\zeta(A))\\
@A AA @A AA\\
\Der^{\rm lf}(\mathcal{A})_{\rm is} @> {\rm d}_{\mathcal{A},\varpi} >> \mathcal{P}(\xi(\mathcal{A}))\end{CD}
\]
commutes, where the vertical maps are the natural scalar extension functors.
\end{itemize}
\end{proposition}

\begin{proof} The first claim is proved by using exactly the same formal argument  that establishes the first claim of Proposition \ref{intermediate prop}. %, can be used in just the same way to show that there exists a canonical extension $D^\ast_{A,\varpi}$ and ${\rm i}^\ast_{A,\varpi}$. The associated functor   given extend $ {\rm d}^\ast_{A,\Sigma}$ to a functor $ {\rm d}_{A,\Sigma}$ of the required sort. s of the properties in claims (i), (ii) and (iii) in the latter result.

Then, since $A$ is semisimple, every cohomology module $H^i(X)$ of an object $X$ in $\Der^{\rm perf}(A)$ is also itself an object of $\Der^{\rm perf}(A)$ (regarded as a complex concentrated in degree zero) and so the general argument of Knudsen and Mumford in \cite[Rem. b) after Th. 2]{knumum} shows that the associated functor $ {\rm d}_{A,\varpi}$ has the property in (i).

Finally, the commutativity of the scalar extension diagram in (ii) is verified by means of a direct comparison of the explicit construction of the functors ${\rm d}^{\diamond}_{A,\varpi}$ and ${\rm d}^{\diamond}_{\mathcal{A},\varpi}$ (the latter from Proposition \ref{functor construct}). The key point in this comparison is that if $A = A_1 = {\rm M}_n(D)$ is simple (as in the explicit construction made above) and $M$ is a free $A$-module of rank $k$, with basis $\{m_j\}_{1\le j\le k}$, then the $D$-module $D^{n}\otimes_AM$ has as a basis the lexicographically-ordered set $\{b_{j,t}\}_{1\le j\le k,1\le t\le n}$, with $b_{j,t} := x_t\otimes m_j$ where $\{x_t\}_{1\le t\le n}$ is the standard $D$-basis of $D^{n}$. In particular, if one uses this basis as the set $\{b_{a}\}_{1\le a \le nk}$ that occurs in the definition (\ref{ss graded def}) of ${\rm d}^\diamond_{A,\varpi}(M)$, then one has
\[ \wedge_{a=1}^{a=nk}(\wedge_{s=1}^{s=m}(w_s^\ast\otimes b_a)) = \wedge_{j=1}^{j=k}((\wedge_{t=1}^{t=n}(\wedge_{s=1}^{s=m}(w_s^\ast\otimes x_t))) \otimes m_j ) =  \wedge_{j=1}^{j=k}m_j,\]
where the exterior product on the right hand side is as defined in (\ref{non-commutative wedge}) with respect to the ordered basis $\varpi =  \underline{\varpi}_1$ specified in \S\ref{ordered basis sec}. \end{proof}

\subsubsection{}Turning now to the proof of Theorem \ref{ext det fun thm} we note that Proposition \ref{intermediate prop} directly implies all assertions except for (i).

In addition, the argument used by Knudsen and Mumford to prove \cite[Prop. 6]{knumum} shows that (i) of Theorem \ref{ext det fun thm} will also follow formally upon combining Proposition \ref{intermediate prop} with the following technical observation.

\begin{proposition}\label{killer prop} Let $X$ and $Y$ be complexes in $\Der^{{\rm lf}}(\mathcal{A})$ and $\alpha$ and $\beta$ quasi-isomorphisms $X \to Y$ of complexes of $\mathcal{A}$-modules with the following property:  in each degree $i$ there are finite filtrations $F^\bullet(H^i(X))$ and $F^\bullet(H^i(Y))$ that are compatible with the maps $H^i(\alpha)$ and $H^i(\beta)$ and such that ${\rm gr}(H^i(\alpha)) = {\rm gr}(H^i(\beta))$ for all $i$.

Then the morphisms ${\rm d}_{\mathcal{A},\varpi}(\alpha)$ and ${\rm d}_{\mathcal{A},\varpi}(\beta)$ coincide.
\end{proposition}

\begin{proof} We note first that, since ${\rm d}_{\mathcal{A},\varpi}(\alpha)$ and ${\rm d}_{\mathcal{A},\varpi}(\beta)$ are homomorphisms between  invertible $\xi(\mathcal{A})$-modules, they coincide if and only if they are equal after applying the scalar extension functor $\zeta(A)\otimes_{\xi(\mathcal{A})}-$.

Given the commutativity of the diagram in Proposition \ref{semisimple det}(ii) we are therefore reduced to showing that the given hypotheses imply
 ${\rm d}_{A,\varpi}(\alpha') = {\rm d}_{A,\varpi}(\beta')$ with $\alpha' := \zeta(A)\otimes_{\xi(\mathcal{A})}\alpha$ and $\beta' :=
 \zeta(A)\otimes_{\xi(\mathcal{A})}\beta$. The result of Proposition \ref{semisimple det}(i) then reduces us to showing that in each degree $i$ the maps ${\rm d}^\diamond_{A,\varpi}(H^i(\alpha'))$ and  ${\rm d}^\diamond_{A,\varpi}(H^i(\beta'))$ coincide. This in turn follows easily from the given hypotheses and the general
property of the functor ${\rm d}^\diamond_{A,\varpi}$ that is described in Definition \ref{det functor def}(b).
\end{proof}

This completes the proof of Theorem \ref{ext det fun thm}.

\end{document}